\documentclass[11pt]{article}
\usepackage[colorlinks=True,allcolors=blue]{hyperref}

\usepackage{wrapfig}

\usepackage{graphicx}
\usepackage{amsmath}
\usepackage{amsthm}
\usepackage{bm}
\usepackage[labelfont=bf,font={small}]{caption}
\usepackage[labelfont=bf]{subcaption}

\usepackage{tabularx}
\newcolumntype{C}{>{\centering\arraybackslash}p{19mm}}
\newcolumntype{G}{>{\centering\arraybackslash}p{4mm}}
\newcolumntype{S}{>{\centering\arraybackslash\scriptsize}p{4mm}}

\usepackage{color}
\usepackage[usenames,dvipsnames]{xcolor}

\usepackage{xspace}

\usepackage{enumitem} % 

\usepackage{cite} % 

\usepackage{amssymb}% 
\usepackage{algorithm}
\makeatletter
\renewcommand*{\ALG@name}{Alg.}
\makeatother
\usepackage[noend]{algpseudocode}

\newtheoremstyle{break}
  {}% 
  {}% 
  {\itshape}% 
  {}% 
  {\bfseries}% 
  {.}% 
  {\newline}% 
  {}% 

\usepackage{cleveref} % 

\theoremstyle{break}

\theoremstyle{definition}
\newtheorem{definition}{Definition}[section]

\newcommand\revFirst[1]{#1} 
\newcommand\rev[1]{\textcolor{black}{#1}}

\newcommand\mydots{\hbox to 1em{.\hss.\hss.}}

\newcommand{\comp}{\mathsf{c}} % 
\newcommand{\dd}{\textrm{d}} 
\newcommand{\dH}{d_H} 
\newcommand{\hull}{\textrm{H}} 
 
\newcommand{\codim}{\textrm{codim}} 
\newcommand{\Ker}{\textrm{Ker}} 
\newcommand{\Int}{\textrm{Int}} 
\newcommand{\reach}{\textrm{reach}} 
\newcommand{\Span}{\textrm{Span}} 
\newcommand{\x}{\bm{x}} % 

\newcommand{\M}{\mathcal{M}}

\newcommand{\Prob}{\mathbb{P}}

\newcommand{\K}{\mathcal{K}} 

\def\Inf{\operatornamewithlimits{inf\vphantom{p}}}

\newcommand{\sff}{\mathrm{I\hspace{0mm}I}}

\newcommand{\G}{\mathcal{G}}

\newcommand{\B}{\mathcal{B}}
\newcommand{\C}{\mathcal{C}}
\newcommand{\X}{\mathcal{X}}
\newcommand{\Y}{\mathcal{Y}}

\newcommand{\U}{U}
\newcommand{\UU}{\mathcal{U}}

\newcommand{\R}{\mathbb{R}}
\newcommand{\bN}{\mathbb{N}}

\newtheoremstyle{nolinebreakstyle}
  {}% 
  {}% 
  {} % 
  {} % 
  {\bfseries} % 
  {.} % 
  {.2em} % 
  {} % 
\theoremstyle{nolinebreakstyle}

\newtheoremstyle{exampstyle}
  {1em plus .2em minus .1em}% 
  {1em plus .2em minus .1em}% 
  {} % 
  {} % 
  {\bfseries} % 
  {.} % 
  {.5em} % 
  {} % 

\theoremstyle{exampstyle}
\newtheorem{remark}{Remark}[section]
\newtheorem{theorem}{Theorem}[section]
\newtheorem{lemma}[theorem]{Lemma}
\newtheorem{corollary}[theorem]{Corollary}

\newtheorem{example}{Example}[section]

\theoremstyle{exampstyle}
\newtheorem{property}{Property}[section]

\theoremstyle{exampstyle}
\newtheorem{assumption}{Assumption}[section]

\usepackage{mdframed}
\usepackage{lipsum}
\newmdtheoremenv{theo}{Theorem}

\usepackage{multirow}

\newcommand\blfootnote[1]{%
  \begingroup
  \renewcommand\thefootnote{}\footnote{#1}%
  \addtocounter{footnote}{-1}%
  \endgroup
}

\usepackage[margin=1in]{geometry}

\makeatletter
\let\@fnsymbol\@arabic
\makeatother

\title{\vspace{-3mm}
Estimating the Convex Hull of the Image of a Set with Smooth Boundary: % 
Error Bounds and Applications}
\author{Thomas Lew\thanks{\rev{Toyota Research Institute.}} \thanks{Department of Aeronautics and Astronautics, Stanford University. % (\texttt{thomasjonathanlew@gmail.com})
}, \ 
Riccardo Bonalli\thanks{Laboratory of Signals and Systems, University of Paris-Saclay, CNRS, CentraleSup\'{e}lec.}, \ 
Lucas Janson\thanks{Department of Statistics, Harvard University.}, \  
Marco Pavone${}^2$}
\date{}

\usepackage{titlesec}
\titleformat{\section}
{\normalfont\Large\bfseries}{\thesection}{1em}{}
\titleformat{\subsection}
{\normalfont\large\bfseries}{\thesubsection}{1em}{}

\begin{document}
\maketitle

\begin{abstract}
We study the problem of estimating the convex hull of the image $f(X)\subset\mathbb{R}^n$ of a compact set $X\subset\mathbb{R}^m$ with smooth boundary through a smooth function $f:\mathbb{R}^m\to\mathbb{R}^n$. Assuming that $f$ is a submersion, we derive a new bound on the Hausdorff distance between the convex hull of $f(X)$ and the convex hull of the images $f(x_i)$ of $M$ sampled inputs $x_i$ on the boundary of $X$. When applied to the problem of geometric inference from a random sample, our results give \revFirst{error bounds that are} tighter and more general \revFirst{than in previous work}. We present applications to the problems of robust optimization, of reachability analysis of dynamical systems, and of robust trajectory optimization under bounded uncertainty. 
\end{abstract}

{
\small
\hspace{3.5mm}\textbf{Keywords:}
convex hull, geometric inference, manifold reconstruction, reach
}% 
\blfootnote{\textbf{Acknowledgements}: 
\rev{The authors thank the anynomous reviewers for their helpful feedback.} The NASA University Leadership Initiative
(grant \#80NSSC20M0163) provided funds to assist the authors with their research, but this article solely reflects the opinions and conclusions of its authors and not any NASA entity.  \revFirst{Toyota Research Institute provided funds to support this work.} L.J. was supported by the National Science Foundation via grant CBET-2112085. }

\section{Introduction}\label{sec:problem_formulation}
Let $\X$ be a compact subset of $\R^m$,   
$f:\R^m\rightarrow\R^n$ be a continuous map, 
$\Y=f(\X)$, and $\hull(\Y)$ be the convex hull of $\Y$. 
Given $M$ inputs $x_i$ sampled from $\X$, we study bounds on the Hausdorff distance 
$\dH(\hull(\Y),\hull(\{f(x_i)\}_{i{=}1}^M))$ 
between the convex hull of $\Y$ and the convex hull of the outputs $f(x_i)$.

Convex hull reconstructions from samples have shown to be surprisingly accurate in complicated settings (e.g., $f$ characterizes a dynamical system parameterized by a neural network \cite{LewJansonEtAl2022}). However, deriving tight error bounds that match empirical results remains an open problem. 
D\"{u}mbgen and Walther \cite{Dumbgen1996} showed that sets $\Y$ that are convex and have a smooth boundary can be accurately estimated using the convex hull of a sample on the boundary of $\Y$.  However, even if the boundary of $\X$ and the map $f$ are smooth, the boundary of $\Y=f(\X)$ may not be smooth, e.g., the boundary of $\Y$ may self-intersect (see Example \ref{ex:intersection}). % 
Thus, it is reasonable to ask: 
Can we derive similar tight error bounds for the estimation of the convex hull of a non-convex set $\Y=f(\X)$ under suitable assumptions on $\X$ and $f$? 
\textbf{Applications}. 
Set reconstruction techniques have found a plethora of applications such as in ecology \cite{DeHaan1994,Cholaquidis2016}, geography \cite{Rodriguez2016}, 
anomaly detection \cite{Devroye1980}, 
data visualization \cite{RayChaudhuri2004}, and astronomy \cite{Jang2007}. 
In many applications, reconstructing the convex hull of the set of interest suffices. 
For instance, verifying that a dynamical system satisfies convex constraints (e.g., a drone avoids obstacles for any given payload $x\in\X$) amounts to estimating the convex hull of the  set $f(\X)$ of all reachable states of the system at a given time in the future \cite{LewEtAl2022,LewJansonEtAl2022,Everett21_journal}, with applications to robust model predictive control \cite{Schurmann2018,Sieber2022}. 
In robust optimization of programs with constraints that must be satisfied for a bounded range $\X$ of parameters \cite{Bertsimas2011}, many problems can be reformulated using the convex hull of the uncertain parameters \cite{BenTal1998,Leyffer2020} or of their image (see Section \ref{sec:applications:optimization}). % 
When the map $f$ is complicated % 
and directly computing $\Y=f(\X)$ is intractable, one may resort to an approximation from sampled outputs $f(x_i)$ instead. 
This approach has the advantages of being problem-agnostic, simple to implement, and computationally efficient for problems of relatively small dimensionality $m$. For instance, in reachability analysis of feedback control loops, this approach can be an order of magnitude faster and more accurate than alternative approaches \cite{LewJansonEtAl2022}. 
However, deriving tight error bounds matching empirical results remains an open problem.

\textbf{Related work}. 
The literature studies the accuracy of different set estimators including 
union\rev{s} of balls \cite{Devroye1980,Baillo2001}, 
$r$-convex hulls \cite{Rodriguez2016,AriasCastro2019},
Delaunay complexes \cite{Boissonnat2013,AamariPhD2017,Aamari2018,Boissonnat2018}, 
and kernel-based estimators \cite{DeVito2014,Rudi2017}. 
If the set to reconstruct is convex, taking the convex hull of a sample yields an estimator with accuracy guarantees \revFirst{in Hausdorff distance %\cite{McClure1975,
\cite{Schneider1987,Schneider1988,Dumbgen1996,Braker1998}, volume %\cite{McClure1975,
\cite{RipleyPoissonForest1977,Schneider1988}, and mean width \cite{Schneider1987}, see \cite[Chapter 8.2]{Schneider2008} for a review.} However, previous works do not study the problem of estimating the convex hull of non-convex sets. % 
Deriving finite-sample error bounds requires making geometric regularity assumptions on the set of interest. 
One such assumption is that the reach \cite{Federer1959} of the set to reconstruct is strictly positive \cite{Cuevas2009,AamariPhD2017,Aamari2018,Aamari2019,Aamari2022}, see Section \ref{sec:smooth_sets:smoothness_defs}. 
Intuitively, a submanifold of reach $R>0$ has a curvature bounded by $1/R$ (see Lemma \ref{lem:2nd_fund_form}) and cannot curve too much onto itself, which limits the minimal size of bottleneck structures \cite{AamariPhD2017,Aamari2019,Berenfeld2021} and guarantees the absence of self-intersections. 
Manifolds of positive reach admit tight bounds on the variation of tangent spaces at different points \cite{Niyogi2008,Boissonnat2019}, which allows deriving tight error bounds for sample-based reconstructions \cite{AamariPhD2017}.

A challenge in applying previous analysis techniques to the estimation of the convex hull of $\Y=f(\X)$ is that the reach of $\Y$ may be zero in many problems of interest, including in problems where both $f$ and the boundary of $\X$ are very regular. % 
For instance, in Example \ref{ex:intersection}, the reach of $\Y$ is zero ($\Y$ self-intersects) although the reach of $\X$ is strictly positive and $f$ is a local diffeomorphism. 
Requiring that $f$ is a diffeomorphism over $\X$ suffices to ensure that the reach of $\Y$ is strictly positive (see Lemma \ref{lemma:reach:diffeo} \cite{Federer1959}), but is an unnecessarily strong assumption that does not allow considering interesting problems with a larger number of inputs than outputs ($f:\R^m\to\R^n$ with $m>n$) such as in the case of reachability analysis of uncertain dynamical systems, see Section \ref{sec:applications}.    
Instead of relying on additional assumptions on $\Y$ or on its convex hull, we seek error bounds that are broadly applicable and that only depend on assumptions on $f$ and $\X$ that can be verified. 

\textbf{Contributions}. We derive new error bounds for reconstructing the convex hull of the image $\Y=f(\X)$ of a set $\X$. The set $\Y$ may be non-convex and its boundary may self-intersect. 
Our results rely on the smoothness of $f$ and of the boundary of $\X$, and on the surjectivity of the differential of $f$, denoted by $\dd f$. 
Our main result is stated below. 

\vspace{-3mm}

\begin{theorem}[Estimation error for the convex hull of $f(\X)$]
\label{thm:error_bound:smooth_boundary}
Let $r,\delta>0$, $\X\subset\R^m$ be a non-empty path-connected compact set that is $r$-smooth (see Definition \ref{def:rsmooth}), 
$f:\R^m\to\R^n$, 
$\Y=f(\X)$, 
and $Z_\delta=\{x_i\}_{i=1}^M\subset\partial\X$ be a $\delta$-cover of the boundary $\partial\X$.  
If $f$ is a $C^1$ submersion such that $(f,\dd f)$ are $(\bar{L},\bar{H})$-Lipschitz, then
\vspace{-3mm}
\begin{equation}\label{eq:thm:error_bound}
\dH(\hull(\Y),\hull(f(Z_\delta)))\leq \frac{1}{2}\left(\frac{\bar{L}}{r}+\bar{H}\right)\delta^2.
\end{equation}
\end{theorem}
Submersions form a large class of functions, see Section \ref{sec:applications} for examples.  
If we apply this result to the special case where $\X$ is convex and $f(x)=x$ (so that $\Y=\X$ is convex and $\bar{L}/r+\bar{H}=1/r$), we obtain $\dH(\X,\hull(Z_\delta))\leq \delta^2/(2r)$, which is $2\times$ tighter than the bound in \cite[Theorem 1]{Dumbgen1996}, see Section \ref{sec:error_bounds_dumbgen_convex}. 
\rev{Thus, Theorem \ref{thm:error_bound:smooth_boundary} tightens the bound in \cite[Theorem 1]{Dumbgen1996} by a factor of $2$ and extends it to the reconstruction of convex hulls of images of non-convex sets with smooth boundary under submersions.} 
\rev{Further discussion on the error bound is provided} in Section \ref{sec:thm:error_bound:smooth_boundary:discussion}.

\textbf{Consequences of Theorem \ref{thm:error_bound:smooth_boundary}}. The derivation of this result is motivated by applications:
\begin{enumerate}
    \item \textit{Geometric inference} (Section \ref{sec:applications:random_samples}): The convex hull of the image $f(\X)$ of sets $\X$ with smooth boundary can be accurately approximated using inputs $x_i$ sampled from a distribution supported on the boundary of $\X$. Theorem \ref{thm:error_bound:smooth_boundary} gives tighter and more general high-probability error bounds (Corollar\revFirst{ies \ref{cor:conservative_finite_sample} and \ref{cor:conservative_asymptotic}}) than prior work \cite{Dumbgen1996,LewJansonEtAl2022}.
    
    \item \textit{Robustness analysis of dynamical systems} (Section \ref{sec:applications:reachability}): Theorem \ref{thm:error_bound:smooth_boundary} justifies approximating convex hulls of reachable sets of dynamical systems from a finite sample (Corollary \ref{cor:reachability_analysis}).  
    Such sampling-based approaches can be used to quickly verify properties of complex systems (e.g., checking that a dynamical system controlled by a neural network satisfies constraints) but previous error bounds do not explain promising empirical results \cite{LewJansonEtAl2022}. 
    As Theorem \ref{thm:error_bound:smooth_boundary} applies to submersions, it applies to systems with a larger number of uncertain parameters than reachable states 
(e.g., characterizing the reachable set of a drone transporting a payload of uncertain mass for a given set of initial states).
       
    \item \textit{Robust programming} (Section \ref{sec:applications:optimization}), \textit{planning, and control} (Section \ref{sec:applications:robust_planning}): The numerical resolution of non-convex optimization problems with constraints that should be satisfied for a range of parameters (e.g., for all parameters in a ball of radius $r$) remains challenging. Theorem \ref{thm:error_bound:smooth_boundary} implies that sampling constraints can yield feasible relaxations % 
    of a class of robust programs (Corollary \ref{cor:robust_programming}), with applications to robust planning and controller design (Section \ref{sec:applications:robust_planning}). 
\end{enumerate}

\textbf{Sketch of proof}. 
To prove Theorem 1.1, we express the approximation error as a function of distances between sampled outputs $f(x_i)$ and tangent spaces of the boundary of $\hull(\Y)$, and exploit the smoothness of $f$ and of the boundary of $\X$. 
Deriving this result is complicated by the absence of a smooth manifold structure for the output set boundary $\partial\Y$ that precludes the direct application of tools from differential
geometry to $\Y$, since its boundary $\partial\Y$ may self-intersect, % 
see Example \ref{ex:intersection}. Our analysis relies on three steps:
\begin{enumerate}
\item Proving that the boundary of the convex hull $\hull(\Y)$ is smooth  under suitable assumptions.

\begin{theorem}[\rev{The convex hull of $f(\X)$ has a smooth boundary}]\label{thm:submersion_rsmooth}
\rev{Let $r>0$, $\X\subset\R^m$ be a non-empty compact set \revFirst{such that a ball of radius $r$ rolls freely in $\X$} (see Definition \ref{def:rolling_ball}), 
$f:\R^m\to\R^n$ be a $C^1$ submersion such that $(f,\dd f)$ are Lipschitz, and $\Y=f(\X)$. 
Then, for some $R>0$, the convex hull $\hull(\Y)$ is $R$-smooth.}
\end{theorem}
In particular, the boundary $\partial\hull(\Y)$ is an $(n-1)$-dimensional submanifold \rev{(Corollary \ref{cor:sub_hull:boundary_manifold})}
and the tangent spaces $T_y\partial\hull(\Y)$ at points $y\in\partial\hull(\Y)$ are well-defined. % 
\item Relating bounds on distances to the convex hull tangent spaces $T_y\partial\hull(\Y)$ % 
at boundary outputs $y$ % 
to bounds on the Hausdorff distance error of convex hull approximations (Lemma \ref{lem:tangent_implies_hausdorff}).  
% 
% %
% \begin{lemma}\label{lem:tangent_implies_hausdorff}
% \rev{Let $\Y,A\subset\R^n$ be non-empty compact sets such that $A\subseteq\hull(\Y)$ and $\partial\hull(\Y)$ is a submanifold of dimension $(n-1)$. 
% Then,
% }
% \begin{equation}\label{eq:tangent_implies_hausdorff}
% \rev{\dH(\hull(\Y),\hull(A))\leq\sup_{y\in\partial\hull(\Y)\cap\partial\Y}\left(\Inf_{a\in A}d_{T_y\partial\hull(\Y)}(y-a)\right).
% }
% \end{equation}
% \end{lemma}
% 
% 
\item Deriving a bound % 
on distances to the tangent spaces % 
$T_y\partial\hull(\Y)$ % 
(Lemma \ref{lem:bound_delta_Tphull}). 
% \begin{lemma}%[Bound on $d_{T_y\partial\hull(\Y)}$ if $f$ is a submersion]
% \label{lem:bound_delta_Tphull}
% \rev{Let $r>0$,  $\X\subset\R^m$ be a non-empty path-connected $r$-smooth compact set, 
% $f:\R^m\to\R^n$ be a $C^1$ submersion such that $(f,\dd f)$ are $(\bar{L},\bar{H})$-Lipschitz, 
% $\Y=f(\X)$, and  
% $x,z\in\partial\X$ \rev{be such that $y=f(x)\in\partial\Y\cap\partial\hull(\Y)$}. Then, 
% }
% \begin{equation}
% % 
% \rev{d_{T_y\partial\hull(\Y)}(f(z)-\rev{f(x)})\leq \frac{1}{2}\left(\frac{\bar{L}}{r}+\bar{H}\right)\|z-x\|^2.}
% \end{equation}
% \end{lemma}
This bound relies on showing that tangent spaces are mapped to tangent spaces ($\dd f_x(T_x\partial\X)=T_y\partial\hull(\Y)$) for particular choices of inputs $x$ and outputs $y$ using the rank theorem (Lemma  \ref{lem:tangent_mapped_to_tangent_hull}), and subsequently studying images of inputs using the smoothness of $f$ and $\partial\X$ and by decomposing % 
components that are tangential and normal to the tangent spaces $T_y\partial\hull(\Y)$.  
\end{enumerate}
This approach gives an error bound (Theorem \ref{thm:error_bound:smooth_boundary}) that is tighter than a bound more easily derived by assuming that $f$ is a diffeomorphism (Lemma \ref{lem:bound_delta_Tphull:diffeo}), which allows exploiting the smoothness of the boundary of $\Y$ (Corollary \ref{cor:rsmooth:diffeo}) but is a more restrictive assumption, % 
see Section \ref{sec:error_bounds:diffeo}. It is also tighter than % 
a bound of the form 
$\dH(\hull(\Y),\hull(f(Z_\delta)))\leq\bar{L}\delta$ from a naive covering argument (Lemma \ref{lem:error_bound:covering_lipschitz}).

% The paper is organized as follows. 
\textbf{Outline}. In \textbf{Section \ref{sec:notations_background}}, we \rev{introduce} notations, different notions of geometric regularity, and review connections between these concepts. 
In \textbf{Section \ref{sec:smoothness_output}}, we study the structure of the convex hull $\hull(\Y)$ \rev{and prove Theorem \ref{thm:submersion_rsmooth}}. 
In \textbf{Section \ref{sec:error_bounds}}, we derive error bounds %\rev{(including Lemmas \ref{lem:tangent_implies_hausdorff} and \ref{lem:bound_delta_Tphull})} 
and prove Theorem \ref{thm:error_bound:smooth_boundary}. 
In \textbf{Section \ref{sec:applications}}, we provide  applications. 
In \textbf{Section \ref{sec:conclusion}}, we conclude and discuss future research directions. 
For conciseness, the proofs of various intermediate results are provided in the appendix.

\section{Notations and background}\label{sec:notations_background}
\textbf{Notations}. 
We denote by $a^\top b$ the Euclidean inner product of $a,b\in\R^n$, 
by $\|a\|$ the Euclidean norm of $a\in\R^n$, 
by $\K$ the family of non-empty compact subsets of $\R^n$, 
by $\B(\R^n)$ the Borel $\sigma$-algebra for the Euclidean topology on $\R^n$ associated to $\|\cdot\|$, 
%by $\lambda_n(\cdot)$ the Lebesgue measure over $\R^n$, 
% 
% 
by $B(x,r)=\{y\in\R^n: \|y-x\|\leq r\}$ the closed ball of center $x\in\R^n$ and radius $r\geq 0$, and 
by $\mathring{B}(x,r)$ the open ball. 
Given $A,B\subset\R^n$ and $c\in\R$, we denote by 
$\Int(A)$, $\overline{A}$, $\partial A=\overline{A}\setminus\Int(A)$, and $A^\comp=\R^n\setminus A$ the interior, closure, boundary, and complement of $A$, 
$cA=\{c a: a\in A\}$, 
by $A+B=\{a+b: a\in A, b\in B\}$ and $A-B=(A^\comp+(-B))^\comp$ the Minkowski addition and difference, 
by $\hull(A)$ the convex hull of $A$, 
%by $\Extreme(A)$ the set of extreme points of $A$,  
by $d_A(x)=\inf_{a\in A} \|x-a\|$ the distance from $x\in\R^n$ to $A$, and by $\dH(A,B)
=
\max(
\sup_{x\in A}
d_B(x), 
\sup_{y\in B}
d_A(y)
)$ the Hausdorff distance between $A,B\in\K$.

\textbf{Differential geometry}. 
Let $\M\subseteq\R^n$ be a $k$-dimensional submanifold without boundary. 
Equipped with the induced metric from the ambient Euclidean norm $\|\cdot\|$, $\M$ is a Riemannian submanifold. The geodesic distance on $\M$ between $x,y\in\M$ is denoted by 
$d^{\M}(x,y)$. For any $p\in\M$, $T_p\M$ and $N_p\M$ denote the tangent and normal spaces of $\M$ \cite{Lee2012}, respectively, which we view as linear subspaces of $\R^n$.  
For any $p\in\M$, the second fundamental form of $\M$ is denoted as $\sff^{\M}_p:T_p\M\times T_p\M\to N_p\M$ and describes the curvature of $\M$ \cite{Lee2018}.

Given a map $f:\R^m\to\R^n$ and $x,v,w\in\R^n$, $\dd f_x$ and $\dd^2 f_x$ denote the first- and second-order differentials of $f$ at $x$, with $\dd f_x(v)=\sum_{i=1}^m\frac{\partial f}{\partial x_i}(x)v_i$ and $\dd^2 f_x(v,w)=\sum_{i=1}^m\sum_{j=1}^m\frac{\partial^2 f}{\partial x_i \partial x_j}(x)v_iw_j$, and 
 $\Ker(\dd f_x)=\{v\in T_x\R^m:\dd f_x(v)=0\}$ denotes the kernel of $\dd f_x$. A differentiable map $f$ is a submersion if $\dd f_x:T_x\R^m\to T_{f(x)}\R^n$ is surjective for all $x\in\R^m$, and a diffeomorphism if it is a bijection and its inverse is differentiable.

\subsection{Geometric regularity: reach, $R$-convexity, rolling balls, and $R$-smoothness}\label{sec:smooth_sets:smoothness_defs}
We introduce different important notions of geometric smoothness.

% \vspace{1mm}

% % 
% \noindent
% \begin{minipage}{0.65\linewidth}
\begin{definition}[Reach]
\rev{Let $\X\subseteq\R^n$ be a closed set.} The \textit{reach} of $\X$ is defined as 
$\reach(\X)=\inf_{x\in\X} d(x,\text{Med}(\X))$, 
where the medial axis of $\X$, denoted as $\text{Med}(\X)$, is the set of points that have at least two nearest neighbors on $\X$. % 
\end{definition}
For a convex set $\X$, $\text{Med}(\X)=\emptyset$ and $\reach(\X)=\infty$.
\begin{definition}[$R$-convexity] 
Let $R>0$ and $\X\subset\R^n$. We say that $\X$ is \textit{$R$-convex} % 
if  
$\X=\bigcap_{\{\rev{x:\,\mathring{B}(x,R)\cap\X=\emptyset\}}}\left(\mathring{B}(x,R))\right)^\comp$. 
\end{definition}
\begin{definition}[Rolling ball]\label{def:rolling_ball}
Let $R>0$ and $\X\subset\R^n$ be a non-empty closed set. We say that \textit{a ball of radius $R$ rolls freely in $\X$} if for any $x\in\partial\X$, there exists $a\in\X$ such that $x\in B(a,R)\subseteq\X$.
\end{definition}
% \end{minipage}% 
% \hspace{2mm}
% \begin{minipage}{0.35\linewidth}
%     \centering
%     \vspace{3mm}	\includegraphics[width=0.6\linewidth]{figs/reach.jpg}
%     \vspace{-2mm}
% \captionof{figure}{Reach and medial axis.}
% \vspace{1mm}	\includegraphics[width=0.48\linewidth]{figs/rolling.jpg}
% % 
% \captionof{figure}{A ball of radius $R$ rolls freely in $\X$ and in $\overline{\X^{\comp}}$. }
% \end{minipage}% 
% 

\begin{definition}[$R$-smooth set]\label{def:rsmooth}
Let $R\geq 0$ and $\X\subset\R^n$ be a non-empty closed set. We say that $\X$ is $R$-\textit{smooth} if a ball of radius $R$ rolls freely in $\X$ and in  $\overline{\X^{\comp}}$.  
Specifically, for any $x\in\partial\X$, there exists $a\in\X$ and $\bar{a}\in\overline{\X^{\comp}}$ such that $x\in B(a,R)\subseteq\X$ and $x\in B(\bar{a},R)\subseteq\overline{\X^{\comp}}$
\end{definition}

\begin{figure}[t]
  \hspace{0.03\linewidth}
  \begin{subfigure}{0.45\linewidth}
  \centering
    \includegraphics[width=0.7\linewidth]{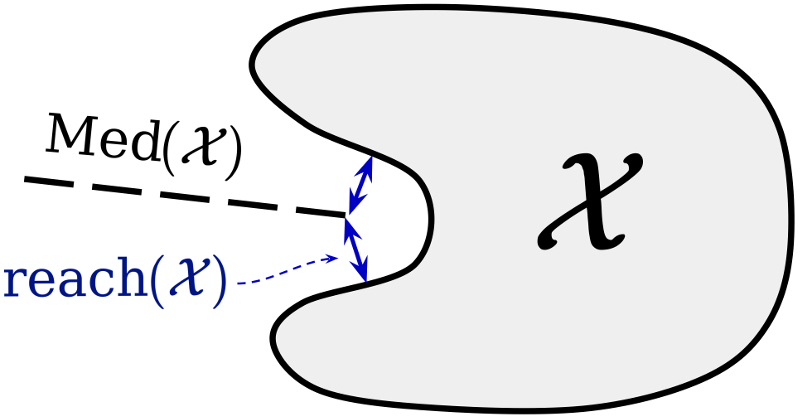}
    \caption{}
      % \label{fig:ift:proof_of_thm:submersion_rsmooth}
  \end{subfigure}
  \begin{subfigure}{0.45\linewidth}
  \centering
    \includegraphics[width=0.63\linewidth]{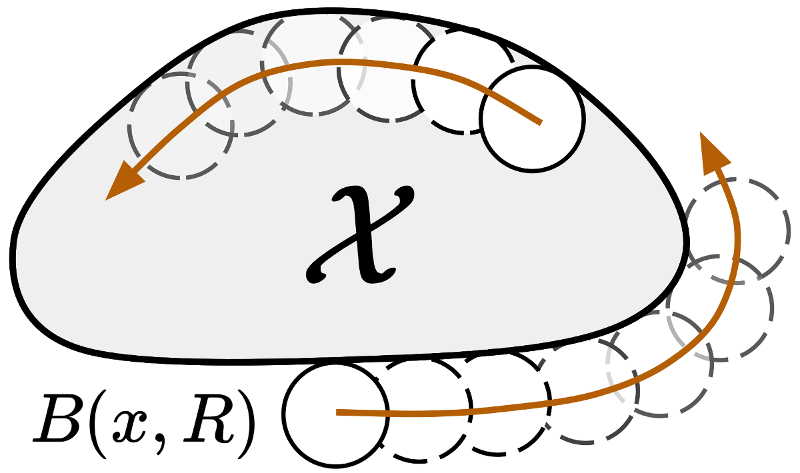}
      \caption{}
  \end{subfigure}
  \caption[Images]{(a) Reach and medial axis. (b) A ball of radius $R$ rolls freely in $\X$ and in $\overline{\X^{\comp}}$, so $\X$ is $R$-smooth. }
  \label{fig:geometric_smoothness}
\end{figure}

A submanifold $\M\subset\R^n$ of reach $R>0$ has a curvature bounded by $1/R$ (Lemma \ref{lem:2nd_fund_form}) and a tubular neighborhood \cite{Lee2018} of radius $R$. A set $\X$ is $R$-convex if it is the intersection of complements of balls of radius $R$ \cite{Cuevas2009}. Thus, $R$-convexity generalizes the notion of convexity, since convex sets $\X$ can be expressed as intersections of halfspaces containing $\X$. % 
The rolling ball condition \cite{Walther1997,Walther1999} is an intuitive notion that we use to define $R$-smooth sets $\X$, for which it is possible to roll a ball of radius $R$ both inside and outside $\X$. 
We will use these four different concepts to derive error bounds. % 

\subsection{Equivalences and connections between definitions}\label{sec:smooth_sets:blaschke}

A key result is the following generalization of Blaschke's Rolling Theorem \cite{Walther1999}. 
\begin{theorem} \label{thm:walther1999}\cite{Walther1999}
Let $\X\subset\R^n$ be a non-empty path-connected compact set and $R>0$. Then, the following are equivalent:
\begin{enumerate}
\item $\X=(\X+ B(0,\lambda))- B(0,\lambda)$ for all $\lambda\in [0,R)$ 
and 
$\X=(\X- B(0,\lambda))+ B(0,\lambda)$ 
 for all $\lambda\in [0,R]$.
\item $\X$ and $\overline{\X^{\comp}}$ are $R$-convex and $\Int(\X)\neq\emptyset$.
\item $\X$ is $\lambda$-smooth for all $0\leq\lambda\leq R$. 
\item $\partial\X$ is an $(n-1)$-dimensional submanifold in $\R^n$ with the outward-pointing unit-norm normal $n(x)$ at $x\in\partial\X$ satisfying 
$
\|n(x)-n(y)\|\leq\frac{1}{R}\|x-y\|$ for all $x,y\in\partial\X
$. 
\end{enumerate}
\end{theorem}
The path-connectedness assumption on $\X$ can be replaced by assumptions on path-connected components of $\X$, see \cite{Walther1997}; we do not study such extensions in this work.

The next lemmas gather known results in the literature that are used in our subsequent proofs. Results that are not cited are not explicitly stated in the literature and are proved in Section  \ref{sec:smooth_sets:proofs}.

\begin{lemma}
\label{lem:roll_implies_reach}\label{lem:roll_implies_reach_boundary}
\cite[Lemmas A.0.6 and A.0.7]{PateiroPhD2008}
Let $\X\subset\R^n$ be a non-empty closed set and $R>0$. Assume that $\X$ is $R$-smooth. % 
Then, $\reach(\X)\geq R$, $\reach(\overline{\X^{\comp}})\geq R$, and $\reach(\partial\X)\geq R$.
\end{lemma}

\begin{lemma}\label{lem:reach_implies_rconvex}\rev{\cite[Theorem 2.6, (2)]{Cotsakis2024}} % 
Let $\X\subset\R^n$ be a closed set % 
with $\reach(\X)\geq R>0$. Then, $\X$ is $R$-convex. %\rev{Morevoer, if $\partial\X$ is an $(n-1)$-dimensional submanifold, then $\reach(\X)\geq R>0$ if and only if $\X$ is $R$-convex.}
\end{lemma}
% Lemma \ref{lem:reach_implies_rconvex} corresponds to \cite[Proposition 1]{Cuevas2012}, which assumes that $\X$ is compact. Notably, boundedness of $\X$ is not used in the proof of the result, see Section \ref{sec:smooth_sets:proofs}.

% \rev{\cite[Theorem 2.6]{Cotsakis2024} additionally states that if $\partial\X$ is an $(n-1)$-dimensional submanifold, then $\reach(\X)\geq R>0$ if and only if $\X$ is $R$-convex. We will not use this result in this work.}

% \rev{\begin{lemma}%[$R$-convexity and rolling condition]
% \label{lem:rconvex_rolling}
% Let $\X\subset\R^n$ be a non-empty closed set such that $\overline{\X^\comp}$ is $R$-convex. Then, a ball of radius $R$ rolls freely in $\X$. 
% \end{lemma}
% \begin{proof}[Proof of Lemma \ref{lem:rconvex_rolling}]
% Let $a>0$ be large-enough so that the set $A=\overline{\X^\comp}\cap B(0,a)$ satisfies $\X\subset \overline{A^\comp}$, $\partial \X\subset\partial A$, and $A$ is $R$-convex (indeed, $\overline{\X^\comp}$ is $R$-convex and $B(0,a)$ is $R$-convex since it is convex, and the intersection of $R$-convex sets is $R$-convex). $A$ is compact, so a ball of radius $R$ rolls freely inside $\overline{A^\comp}$ by \cite[Proposition 2]{Cuevas2012}. Thus, a ball of radius $R$ rolls freely inside $\X$.
% \red{TODO just follow the original proof?}
% \end{proof}
% Lemmas \ref{lem:reach_implies_rconvex} and \ref{lem:rconvex_rolling} imply that $\reach(\overline{\X^\comp})\geq R\implies$ inside rolling condition}

% \red{TODO: add figure to say the converse is not true.}

% \red{REMOVE COMPLETELY THIS SECTION?}
% 
% 
% 
% 
% 
% 
% 

\begin{lemma}\label{lem:rsmooth_implies_lambdasmooth}
Let $R>0$ and $\X\subset\R^n$ be a non-empty closed set. \revFirst{Assume that a ball of radius $R$ rolls freely in $\X$. Then, for all $0\leq\lambda\leq R$, a ball of radius $\lambda$ rolls freely in} $\X$.
% if $\X$ is $R$-smooth, then $\X$ is $\lambda$-smooth for all $0\leq\lambda\leq R$.
\end{lemma}

\begin{lemma}\label{lem:convex_rolling_outside}
Let $\X\subset\R^n$ be a non-empty convex closed set. Then, a ball rolls freely in $\overline{\X^{\comp}}$. 
\end{lemma}

\begin{lemma}\cite[Proposition 6.1]{Niyogi2008}\cite[Proposition III.22]{AamariPhD2017}\label{lem:2nd_fund_form}
Let $\M\subset\R^n$ be a submanifold with $\reach(\M)\geq R>0$. Then, $\|\sff_x^{\M}(v,v)\|\leq \frac{1}{R}$ for all $x\in\M$ and unit-norm $v\in T_x\M$.
\end{lemma}
The next important result is an alternative characterization of the reach of a set.% 
\begin{theorem}\label{thm:reach}\cite[Theorem 4.18]{Federer1959}
Let $\M\subset\R^n$ be a submanifold and $0<r<\infty$. Then, $\reach(\M)\geq r$ if and only if $d_{T_p\M}(q-p)\leq\frac{\|q-p\|^2}{2r}$ for all $p,q\in\M$. 
Thus,
$$
\reach(\M)
=
\inf_{p\neq q\in \M}\frac{\|q-p\|^2}{2d_{T_p\M}(q-p)}.
$$ 
\end{theorem}
Theorem \ref{thm:reach} provides an alternative definition of the reach of a submanifold, as well as a bound on the distance to the tangent space. We note that \cite[Theorem 4.18]{Federer1959} applies to closed sets $\M$ that are not necessarily submanifolds, after appropriately defining the tangent space $T_p\M$. If $\M$ is a submanifold, the definition of the tangent space in \cite[Theorem 4.18]{Federer1959} matches the usual definition, see \cite[Remark 4.6]{Federer1959}. We only apply Theorem \ref{thm:reach} to submanifolds in this work.

\section{Smoothness of the convex hull $\hull(\Y)$}\label{sec:smoothness_output}

As we will see in Section \ref{sec:error_bounds}, 
a set that is $R$-smooth can be accurately reconstructed from a sample \rev{(Theorem \ref{thm:error_bound:smooth_boundary})}. 
Thus, in this section, we study the smoothness properties of $\Y$ \rev{and of its convex hull}. 
The main result of this section is that the convex hull of $\Y$ is always $R$-smooth if \rev{a ball rolls freely in $\X$ and} $f$ is a submersion (Theorem \ref{thm:submersion_rsmooth}), so the boundary of the convex hull is necessarily a submanifold (Corollary \ref{cor:sub_hull:boundary_manifold}).

\rev{The next example illustrates the main difficulty in obtaining smoothness properties of images $\Y=f(\X)$ of smooth sets $\X$. E}ven if $\X\subset\R^m$ is $r$-smooth and $f$ is a local diffeomorphism (in particular, a submersion), $\Y$ may not be $R$-smooth. Indeed, if $\Y$ is $R$-smooth, then $\partial\Y$ must be an $(n-1)$-submanifold by Theorem \ref{thm:walther1999}. However, this may not be the case due to self-intersections. % as shown next.

\vspace{1mm}

\noindent\begin{minipage}{0.6\linewidth}
\begin{example}[\rev{Self-intersections and $R$-smoothness}]\label{ex:intersection}
Let $L=\{(r,\theta)\in\R^2: r=1.5, (-\pi+0.5)\leq\theta\leq (3\pi/2-0.5)\}$ and define the input set and map 
$$
\X = L+B(0,0.5), \  
f:\R^2\to\R^2: (r,\theta)\mapsto(r\cos(\theta),r\sin(\theta)).
$$
One can freely roll a ball of radius $0.5$ inside $\X$ and $\overline{\X^{\comp}}$, so $\X$ is $r$-smooth. 
Moreover, the map 
$f$ is a local diffeomorphism on $\X$ (hence a submersion). However, one cannot roll a ball outside $\Y$, so $\Y$ is not $R$-smooth for any $R>0$.  We represent $\Y$ in Figure \ref{fig:self_intersect}. 
\end{example}
\end{minipage}% 
\hspace{2mm}
\begin{minipage}{0.39\linewidth}
    \centering	
    \includegraphics[width=0.95\linewidth]{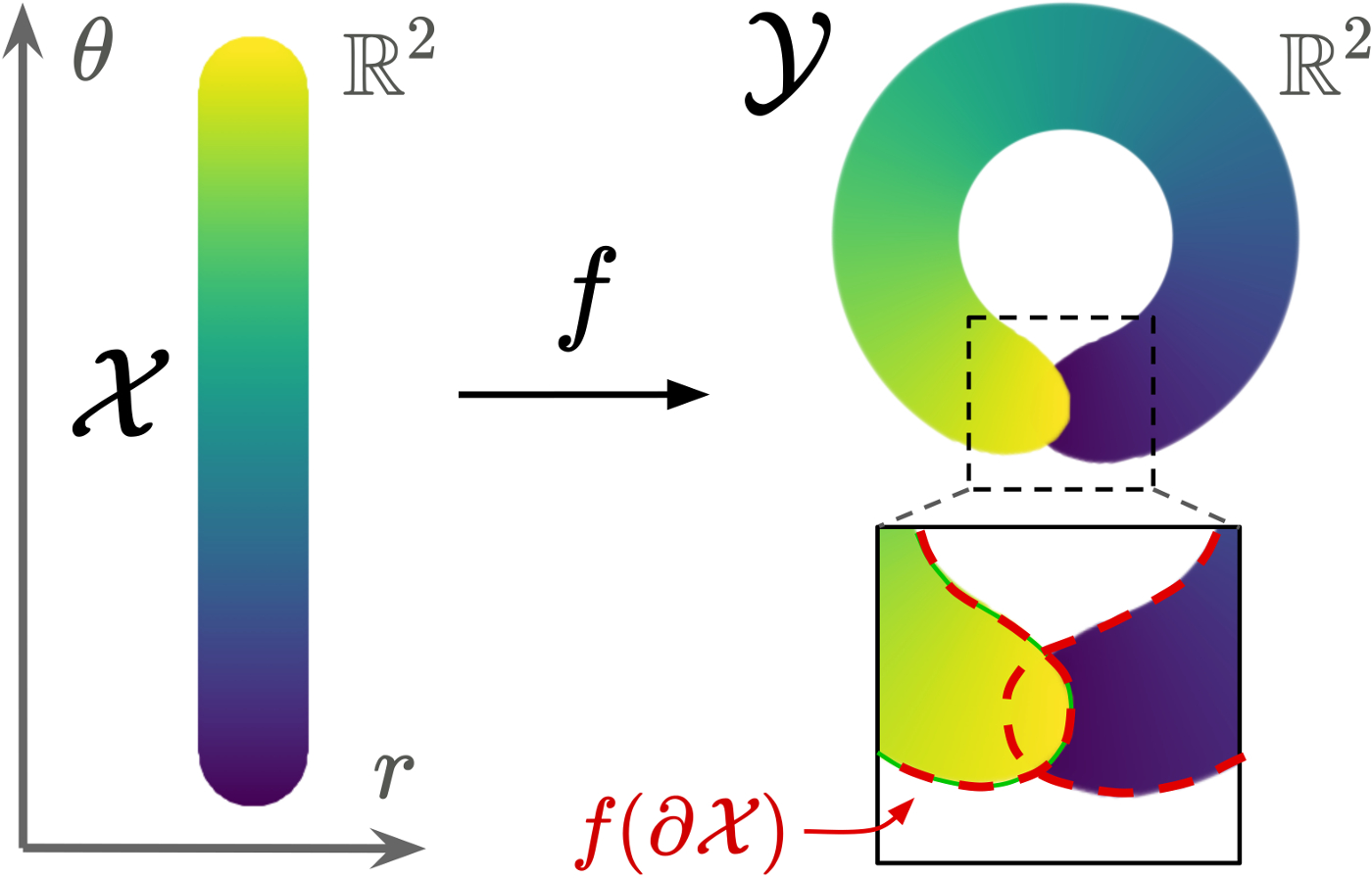}
    \captionof{figure}{The boundary of images of sets with smooth boundary \rev{may not be smooth due to self-intersections}.}
    \label{fig:self_intersect}
\end{minipage}% 
\\

Theorem \ref{thm:submersion_rsmooth} implies that a ball of radius $R$ rolls freely in $\hull(\Y)$ and $\overline{\hull(\Y)^\comp}$ if $f$ is a submersion. 
We stress that %\rev{although a ball rolls freely in $\Y$,} 
Theorem \ref{thm:submersion_rsmooth} does not imply that $\Y$ is $R$-smooth, as \rev{$\partial\Y$ may not be a smooth submanifold due to self-intersections}. \rev{In general, $f(\X^\comp)\neq\Y^\comp$ (in Example \ref{ex:intersection}, $f(\X^\comp)=\R^n\neq\Y^\comp$), so assuming that a ball rolls freely inside $\X$ (or that $\X$ is $r$-smooth) and $f$ is a submersion is insufficient to ensure that $\Y$ is $R$-smooth, as a ball may not roll freely in $\overline{\Y^\comp}$.}
The main idea of the proof \rev{of Theorem \ref{thm:submersion_rsmooth}} is that intersections disappear after taking the convex hull. 

\rev{We prove Theorem \ref{thm:submersion_rsmooth} in four steps. First, we show that smoothness properties of $\X$ are preserved if $f$ is a diffeomorphism (Section \ref{sec:smoothness_output:diffeo}). Second, we show that rolling-ball properties of $\X$ are preserved if $f$ is a submersion (Section \ref{sec:smoothness_output:sub:rolling}). The proof uses the rank theorem and results in Section \ref{sec:smoothness_output:diffeo}. 
Third, we show that smooothness properties are preserved after taking the convex hull (Section \ref{sec:smoothness_output:hull}). Finally, we prove Theorem \ref{thm:submersion_rsmooth} by combining the previous results (Section \ref{sec:smoothness_output:conclusion}).} 

% \subsection{\rev{$R$-smoothness and rolling-ball properties of $\Y$}}
\subsection{\rev{Smoothness properties are preserved under diffeomorphisms}}\label{sec:smoothness_output:diffeo}

If $f$ is a diffeomorphism, the output set $\Y$ is always $R$-smooth if $\X$ is $r$-smooth (Corollary \ref{cor:rsmooth:diffeo}). This result almost immediately follows from the well-known result that the reach is conserved under diffeomorphisms, see \cite[Theorem 4.19]{Federer1959} and \cite[Lemma III.17]{AamariPhD2017}. 
\\[2mm]
\begin{minipage}{0.6\linewidth} 
\begin{lemma}[Stability of the reach under diffeomorphisms]\label{lemma:reach:diffeo}

\hspace{-1pt}\cite[Theorem 4.19]{Federer1959}
Let $\X\subset\R^n$ be a closed set, 
$f:\R^n\to\R^n$, and $\Y=f(\X)$. Let $r,s>0$, and assume that $\reach(\X)\geq r>0$ and that $f|_{\X+B(0,s)}$ is a $C^1$ diffeomorphism such that $(f,f^{-1},\dd f)$ are $(\bar{L},\underline{L},\bar{H})$-Lipschitz, respectively.  
Then,
$$
\reach(\Y)\geq
\min\left(
\frac{s}{\underline{L}},
\frac{1}{\left(\frac{\bar{L}}{r}+\bar{H}\right)\underline{L}^2}
\right).
$$ 
\end{lemma}
\end{minipage}% 
\hspace{2mm}
\begin{minipage}{0.39\linewidth}
    \centering
    \vspace{2mm}	\includegraphics[width=0.99\linewidth]{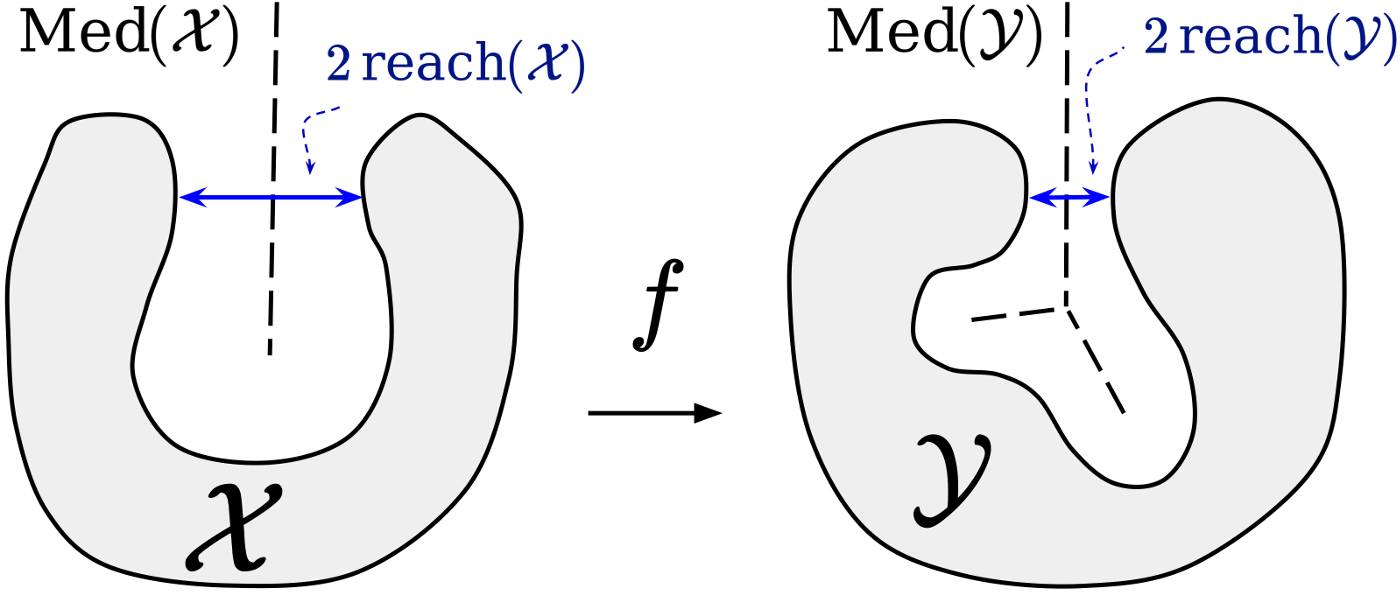}
    \captionof{figure}{Positive reach is conserved under diffeomorphisms.}
    \label{fig:reach_diffeo}
\end{minipage}% 

\begin{corollary}[$r$-smoothness is preserved under diffeomorphisms]\label{cor:rsmooth:diffeo}
Let $r>0$, $\X\subset\R^n$ be a non-empty path-connected $r$-smooth compact set, and 
$f:\R^n\to\R^n$ be a $C^1$ diffeomorphism such that $(f,f^{-1},\dd f)$ are $(\bar{L},\underline{L},\bar{H})$-Lipschitz. Let $R^{-1}=
\left(\frac{\bar{L}}{r}+\bar{H}\right)\underline{L}^2$ and $\Y=f(\X)$. \rev{Then, $\Y$ is $R$-smooth.}
% Then, $\Y$ and $\hull(\Y)$ are $R$-smooth\red{, and the boundaries  $\partial\Y$ and $\partial\hull(\Y)$ are submanifolds of dimension $(n-1)$ with outward-pointing unit-norm normals that are $(1/R)$-Lipschitz.}
\end{corollary}
\begin{proof}
By Lemma \ref{lem:roll_implies_reach}, $\reach(\X)\geq r$ and $\reach(\overline{\X^{\comp}})\geq r$. By Lemma \ref{lemma:reach:diffeo} % 
$\reach(\Y)\geq R$ and $\reach(\overline{\Y^\comp})\geq R$. By Lemma \ref{lem:reach_implies_rconvex}, $\Y$ and $\overline{\Y^\comp}$ are $R$-convex. Thus, $\Y$ is $R$-smooth by Theorem \ref{thm:walther1999} ($\Y$ is a non-empty path-connected compact set since $\X$ is one and $f$ is a diffeomorphism).  
%By applying Lemma \ref{lem:rsmooth:hull}, $\hull(\Y)$ is also $R$-smooth. 
%The last result follows from Theorem \ref{thm:walther1999}.
\end{proof}

\begin{corollary}[\revFirst{The rolling ball condition is preserved under diffeomorphisms}]\label{cor:rsmooth:diffeo:rolling}
\revFirst{Let $\X\subset\R^n$ be a non-empty closed set such that a ball of radius $r>0$ rolls freely in $\X$, and 
$f:\R^n\to\R^n$ be a $C^1$ diffeomorphism such that $(f,f^{-1},\dd f)$ are $(\bar{L},\underline{L},\bar{H})$-Lipschitz. Let $R^{-1}=
\left(\frac{\bar{L}}{r}+\bar{H}\right)\underline{L}^2$ and $\Y=f(\X)$. 
Then, a ball of radius $R$ rolls freely in $\Y$.}%\red{ and in $\hull(\Y)$.}}
\end{corollary}
\begin{proof}
\revFirst{Let $y\in\partial\Y$ and $x\in\partial\X$ be such that $y=f(x)$. Since a ball of radius $r$ rolls freely in $\X$, there exists a ball $B_x\triangleq B(\hat{x},r)$ for some $\hat{x}\in\X$ such that $x\in B_x\subseteq\X$. By Corollary \ref{cor:rsmooth:diffeo} and since $B_x$ is $r$-smooth, $f(B_x)$ is $R$-smooth. Thus, a ball of radius $R$ rolls freely in $f(B_x)$. Thus, there exists $\hat{y}\in\Y$ such that $y\in B(\hat{y},R)\subseteq f(B_x)\subseteq\Y$. Thus, a ball of radius $R$ rolls freely in $\Y$. %The proof that a ball of radius $R$ rolls freely in $\hull(\Y)$ follows from the proof of Lemma \ref{lem:rsmooth:hull}.
}
\end{proof}

\subsection{\rev{Rolling-ball properties are preserved under submersions}}\label{sec:smoothness_output:sub:rolling}
Assuming that $f$ is a diffeomorphism \rev{as in Section \ref{sec:smoothness_output:diffeo}} is restrictive, as it implies that $f$ maps between two sets  of the same dimension and thus does not allow considering problems with more inputs $x\in\X$ than outputs $y\in\Y$. \rev{Although $\Y$ is not necessarily $R$-smooth (see Example \ref{ex:intersection}), we show that a ball rolls freely in $\Y$ if $f$ is a submersion and \revFirst{a ball rolls freely in $\X$ (e.g., if} $\X$ is $r$-smooth\revFirst{)}. The next result combines the rank theorem and \revFirst{Corollary \ref{cor:rsmooth:diffeo:rolling}}.}

% \rev{The main result of this section is the following.}
\begin{lemma}\label{lem:submersion_rolling_ball_Y}
\rev{Let $r>0$, $\X\subset\R^m$ be a non-empty compact set such that a ball of radius $r$ rolls freely in $\X$, 
$f:\R^m\to\R^n$ be a $C^1$ submersion such that $(f,\dd f)$ are Lipschitz, and $\Y=f(\X)$. 
Then, for some $R>0$, a ball of radius $R$ rolls freely in $\Y$.}
\end{lemma}

To prove Lemma \ref{lem:submersion_rolling_ball_Y}, we use the following intermediate results. % 

\begin{lemma}\label{lem:submersion_open_map}\cite[Corollary C.36]{Lee2012}
Let  $f:\R^m\to\R^n$ be a $C^1$ submersion. Then, $f$ is an open map.
\end{lemma}
\begin{lemma}\label{lem:open_map_boundary}
Let  $f:\R^m\to\R^n$ be a continuous open map and  $\rev{A}\subset\R^m$ be compact. Then, $\partial f(\rev{A})\subseteq f(\partial\rev{A})$.
\end{lemma}
\begin{proof}
$A$ is closed, so $\partial A=\overline{A}\setminus\Int(A)=A\setminus\Int(A)$\rev{, so $A=\Int(A)\cup\partial A$.} 
%$A=\Int(A)\sqcup\partial A$, where $\sqcup$ denotes the disjoint union. 
Similarly, 
$f$ is continuous and $A$ is compact, so $f(A)$ is compact\rev{, so $\partial f(A)= \overline{f(A)}\setminus\Int(f(A))=f(A)\setminus\Int(f(A))$}. 
% Next, we show that $\partial f(A)\subseteq f(\partial A)$. 

Let $\rev{y}\in\partial f(A)$. \rev{Then,}
%Since $\partial f(A)=f(A)\setminus \Int(f(A))$, 
there exists $\rev{x}\in A$ such that $\rev{y}=f(\rev{x})$ and either $\rev{x}\in\Int(A)$ or $\rev{x}\in\partial A$. 
\rev{If $x\in\Int(A)$, then $y=f(x)\in f(\Int(A))\subseteq \Int(f(A))$ since $f$ is open. Thus, $y\in\Int(f(A))$, which contradicts $y\in\partial f(A)$. We conclude that $x\in\partial A$.}
\end{proof}

\begin{lemma}\label{lemma:proj}
Let $A\subset\R^m$ be a non-empty \revFirst{closed} set, 
$\pi:\R^m\to\R^n:(x_1,\dots,x_n,x_{n+1},\dots,x_m)\mapsto(x_1,\dots,x_n)$ be the standard projection\revFirst{, and  assume that a ball of radius $r>0$ rolls freely in $A$.} Then, \revFirst{a ball of radius $r$ rolls freely in} $\pi(A)$.
\end{lemma}
 
\begin{proof} 
$\pi(A)$ is \revFirst{closed} since $A$ is \revFirst{closed} and $\pi$ is continuous.  
\revFirst{Let} $y\in\partial\pi(A)$.  
By Lemma \ref{lem:open_map_boundary}, $\partial\pi(A)\subseteq\pi(\partial A)$. Thus, there exists $x\in\partial A$ such that $y=\pi(x)$.  
Since \revFirst{a ball of radius $r$ rolls freely in} $A$, there exists a ball $B^m\triangleq B(\hat{x},r)$ such that $x\in B^m\subseteq A$. 
\revFirst{The ball}  $B^n\triangleq\pi(B^m)$ with $B^n=B(\hat{y},r)=B(\pi(\hat{x}),r)\subset\R^n$ satisfies $y\in B^n\subseteq\pi(A)$. \revFirst{Thus, a ball of radius $r$ rolls freely in $\pi(A)$.}
\end{proof}

\begin{proof}[Proof of Lemma \ref{lem:submersion_rolling_ball_Y}]
Let $x\in\partial\X$.  
By the rank theorem, since $f$ is a submersion, there exist two charts $(U_x,\varphi_x)$ and $(V_{f(x)},\psi_{f(x)})$ such that $\psi_{f(x)}\circ f\circ\varphi_x^{-1}$ is a coordinate projection:
$$
\psi_{f(x)}\circ f\circ\varphi_x^{-1}: \varphi_x(U_x\cap f^{-1}(V_{f(x)}))\to \psi_{f(x)}(V_{f(x)}): \hat{x}=(\hat{x}_1,\dots,\hat{x}_n,\hat{x}_{n+1},\dots\hat{x}_m)\mapsto (\hat{x}_1,\dots,\hat{x}_n).
$$

We proceed in three steps.
\begin{itemize}[leftmargin=5mm]\setlength\itemsep{0.5mm}
\item \textit{Step 1: Build a suitable finite family of charts for $\partial\X$.} Since \revFirst{a ball of radius $r$ rolls freely in} $\X$, \revFirst{for all $0\leq\lambda\leq r$, a ball of radius $\lambda$ rolls freely in} $\X$ by Lemma \ref{lem:rsmooth_implies_lambdasmooth}. Thus, % 
	for any $0<r_x\leq r$ and any $\tilde{x}\in \partial\X\cap U_x$,  
	there is a ball $B(\bar{x},r_x)$ such that $\tilde{x}\in B(\bar{x},r_x)\subseteq\X$. 
	Let $\epsilon_x>0$ be small-enough so that $B(x,\epsilon_x)\subset U_x$. Then, by choosing $r_x>0$ small-enough,  
$$
(P)\quad \text{for any }\tilde{x}\in\partial\X\cap B(x,\epsilon_x),
\ \text{there is } B(\bar{x},r_x)
\ \text{such that }
\tilde{x}\in B(\bar{x},r_x)\subset\X\cap U_x.
$$
The family $\{(B(x,\epsilon_x),\varphi_x)\}_{x\in\partial\X}$ is thus a family of smooth charts that covers $\partial\X$ and satisfies (P).  
Since $\partial\X$ is compact, there exists a \textit{finite} subcover of $\{(B(x,\epsilon_x),\varphi_x)\}_{x\in\partial\X}$. 
Thus, we restrict our attention to a finite family of such charts $\{(B(x_i,\epsilon_{x_i}),\varphi_{x_i})\}_{i\in I}$ covering $\partial\X$ satisfying (P).  

\item \textit{Step 2: Show that there is a ball at any $y\in\partial\Y$ inside $\Y$.}
Let  $y\in\partial\Y$ be arbitrary and % 
$\tilde{x}\in\partial\X$ be such that $y=f(\tilde{x})$ (this $\tilde{x}$ exists since $\partial\Y\subseteq f(\partial\X)$, by Lemmas \ref{lem:submersion_open_map} and \ref{lem:open_map_boundary} since $f$ is a submersion).  Then, since the $B(x_i,\epsilon_{x_i})$ cover $\partial\X$, $\tilde{x}$ is in the domain of one of the charts $(B(x_i,\epsilon_{x_i}),\varphi_{x_i})$ for some $i\in I$, i.e., $\tilde{x}\in\partial\X\cap B(x_i,\epsilon_{x_i})$. Thus, by  (P), there exists $B(\bar{x},r_{x_i})$ such that $\tilde{x}\in B(\bar{x},r_{x_i})\subset\X\cap U_{x_i}$. 
Next, we prove that \revFirst{for some $R_{x_i}>0$, a ball of radius $R_{x_i}$ rolls freely in} $f(B(\bar{x},r_{x_i}))$ in four steps. For conciseness, we denote $\varphi=\varphi_{x_i}$ and $\psi=\psi_{f(x_i)}$.   
\begin{itemize}[leftmargin=5mm]\setlength\itemsep{0.5mm}
\item \revFirst{A ball of radius $r_{x_i}$ rolls freely in} $B(\bar{x},r_{x_i})$.

\item By Corollary \ref{cor:rsmooth:diffeo:rolling}, \revFirst{a ball of radius $\tilde{r}_{x_i}$ rolls freely in} $\varphi(B(\bar{x},r_{x_i}))$ for some $\tilde{r}_{x_i}>0$, since %\revFirst{a ball of radius $r_{x_i}$ rolls freely in} $B(\bar{x},r_{x_i})$  and 
$\varphi$ is a diffeomorphism on $U_{x_i}$ with $B(\bar{x},r_{x_i})\subset U_{x_i}$.

\item By Lemma \ref{lemma:proj}, \revFirst{a ball of radius $\tilde{r}_{x_i}$ rolls freely in} $(\psi\circ f)(B(\bar{x},r_{x_i}))$, since $(\psi\circ f)(B(\bar{x},r_{x_i}))=(\psi\circ f\circ\varphi^{-1})(\varphi(B(\bar{x},r_{x_i})))
=\pi(\varphi(B(\bar{x},r_{x_i})))$. % and \revFirst{a ball of radius $r_{x_i}$ rolls freely in} $\varphi(B(\bar{x},r_{x_i}))$.

\item By \revFirst{Corollary \ref{cor:rsmooth:diffeo:rolling}}, \revFirst{for some $R_{x_i}>0$, a ball of radius $R_{x_i}$ rolls freely in} $f(B(\bar{x},r_{x_i}))$, since %\revFirst{a ball of radius $\tilde{r}_{x_i}$ rolls freely in} $(\psi\circ f) (B(\bar{x},r_{x_i}))$  and 
$\psi^{-1}$ is a diffeomorphism.
\end{itemize}
Since \revFirst{a ball of radius $R_{x_i}$ rolls freely in} $f(B(\bar{x},r_{x_i}))$ 
and $y\in\partial f(B(\bar{x},r_{x_i}))$,  
for any $0<R\leq R_{x_i}$, 
there exists a ball $B(\bar{y},R)$ such that $y\in B(\bar{y},R)\subseteq f(B(\bar{x},r_{x_i}))\subseteq\Y$.   

\item \textit{Step 3: Show that there is a ball of fixed radius $R$ at any $y\in\partial\Y$ inside $\Y$.} Let $R=\inf_{i\in I}R_{x_i}$. Since $R_{x_i}>0$ for all $i\in I$ and $I$ is finite, % 
$R>0$. Since $0<R\leq R_{x_i}$ for all $i\in I$, we apply \textit{Step 2} and obtain that for any $y\in\partial\Y$, 
there exists a ball $B(\bar{y},R)$ such that $y\in B(\bar{y},R)\subseteq\Y$. 
\end{itemize}
Thus, a ball of radius $R$ rolls freely in $\Y$ for some $R>0$. 
\begin{figure}[t]
\centering
\includegraphics[width=0.6\linewidth]{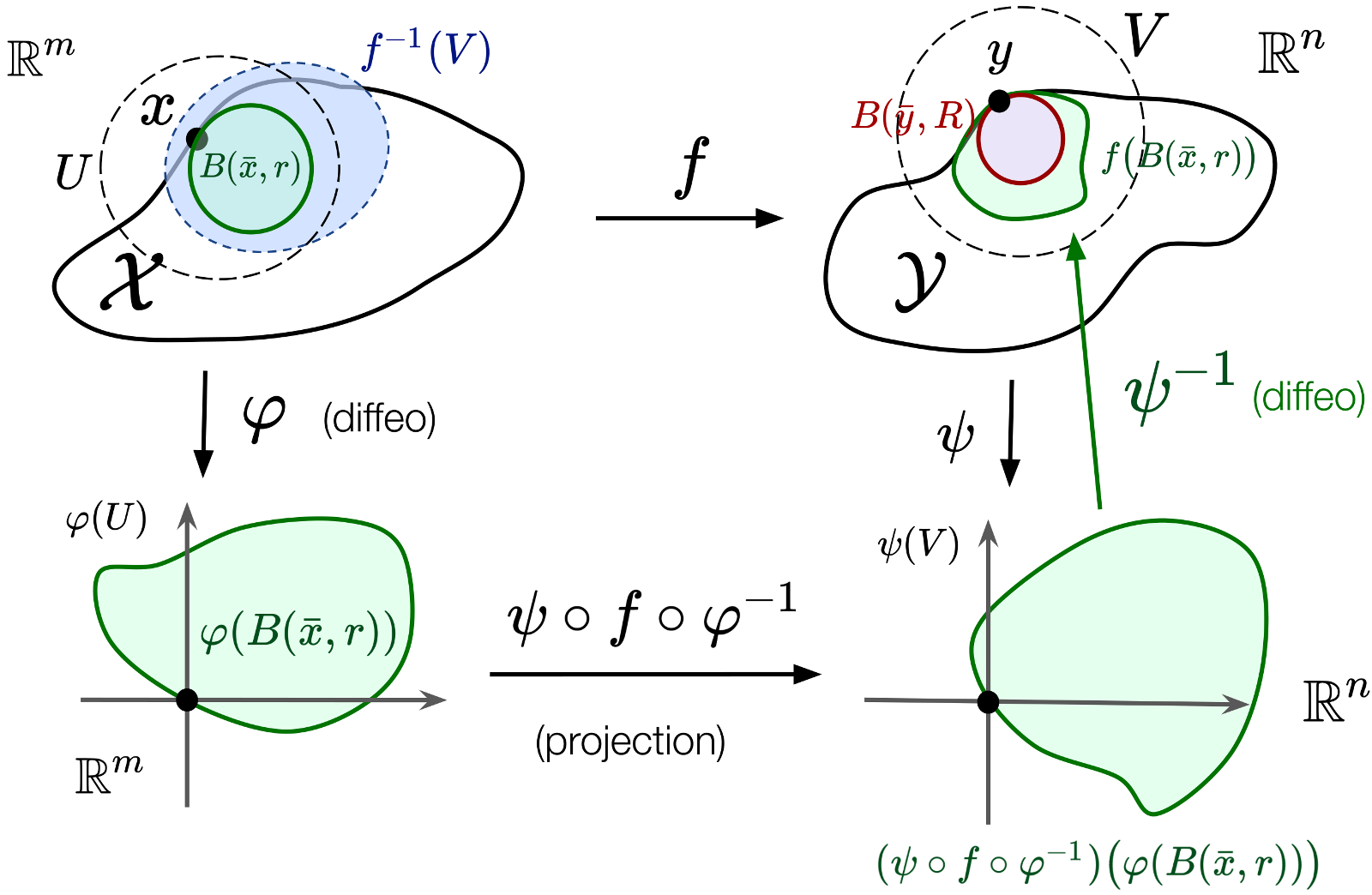}
\caption{Definitions for the proof of \rev{Lemma \ref{lem:submersion_rolling_ball_Y}}.}
\label{fig:ift:proof_of_thm:submersion_rsmooth}
\end{figure}
\end{proof}

\subsection{\rev{Smoothness properties are preserved when taking the convex hull}}\label{sec:smoothness_output:hull}
\rev{The next lemma states that showing that a ball rolls freely in $\Y$ is sufficient to conclude that the convex hull of $\Y$ is $R$-convex. This result formalizes the idea that intersections (see Example \ref{ex:intersection}) disappear after taking the convex hull.}
\begin{lemma}[\rev{$\hull(\Y)$ is $R$-smooth if a ball of radius $R$ rolls freely in $\Y$}]\label{lem:rsmooth:hull}
Let $R>0$ and $\Y\subset\R^n$ be a non-empty compact set. Assume that \rev{a ball of radius $R$ rolls freely in} $\Y$. Then, $\hull(\Y)$ is $R$-smooth.
\end{lemma}

\begin{proof}% 
Since $\hull(\Y)$ is convex, a ball of radius $R$ rolls freely in $\overline{\hull(\Y)^\comp}$ by Lemma \ref{lem:convex_rolling_outside}. 
Next, we show that a ball of radius $R$ rolls freely inside $\hull(\Y)$. 
\noindent\begin{minipage}{0.65\linewidth}
\hspace{5mm} % 
To do so, we decompose the boundary $\partial\hull(\Y)$ as the \rev{union} $$\partial\hull(\Y)=\rev{(\partial\hull(\Y)\cap\partial\Y)\cup(\partial\hull(\Y)\setminus\partial\Y)}$$
and study boundary points in these two subsets. % 
We represent this decomposition in Figure \ref{fig:decomposition_hull}. 
\begin{itemize}[leftmargin=5mm]
\item First, let $y\in\partial\hull(\Y)\cap\partial\Y$.  Since \rev{a ball of radius $R$ rolls in} $\Y$, there exists a ball $B(\bar{y},R)$ such that $y\in B(\bar{y},R)\subseteq\Y\subseteq\hull(\Y)$.
\end{itemize}
\end{minipage}% 
\hspace{0.045\linewidth}
\begin{minipage}{0.3\linewidth}
	\vspace{-1mm}
    \centering	\includegraphics[width=0.85\linewidth]{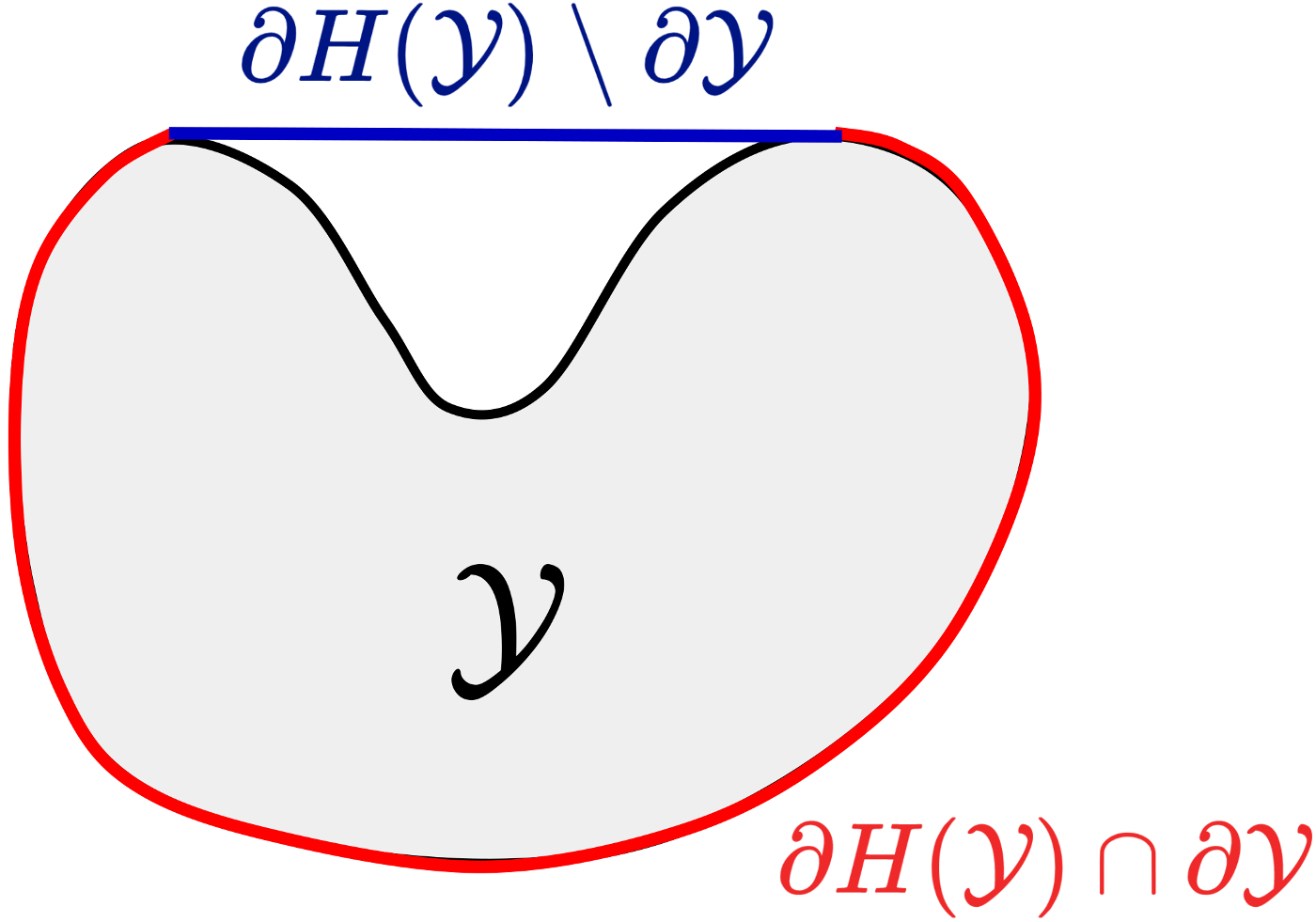}
\captionof{figure}{Decomposition of the boundary of $\hull(\Y)$.}
\label{fig:decomposition_hull}
\end{minipage}% 
 
\noindent\begin{minipage}{0.65\linewidth}
\begin{itemize}[leftmargin=5mm]
\item Second, let \rev{$y\in\partial\hull(\Y)\setminus\partial\Y$}. 
\rev{Then, there exist $y_1,y_2\in\Y$ and $t\in(0,1)$ such that $y=ty_1+(1-t)y_2$. If $y_1\in\Int(\Y)$ or $y_2\in\Int(\Y)$, then $y\in\Int(\hull(\Y))$. Indeed, if $y_1\in\Int(\Y)$ (the proof for the case $y_2\in\Int(\Y)$ is identical), then $B(y_1,\epsilon)\subseteq\Y$ for some $\epsilon>0$, see Figure \ref{fig:cone}. Then, $B(y,t\epsilon)\subseteq\hull(\Y)$. Indeed, any $z\in B(y,t\epsilon)$ can be written as $z=y+t\epsilon u$ for some unit-norm $u$. Then, $z=t(y_1+\epsilon u)+(1-t)y_2$, so $z\in\hull(B(y_1,\epsilon)\cup\{y_2\})\subseteq\hull(\Y)$, so $B(y,t\epsilon)\subseteq\hull(\Y)$, which implies that $y\in\Int(\hull(\Y))$. This contradicts $y\in\partial\hull(\Y)$, so $y_1,y_2\in\partial\Y$.}

% Then, $y$ lies between two distinct extreme points $y_1,y_2\in\partial\hull(\Y)\cap\partial\Y$, such that $y=ty_1+(1-t)y_2$ for some $t\in(0,1)$, see Figure \ref{fig:tangentballmilman} (note that $\hull(\Y)=\hull(\Extreme(\Y))=\hull(\Extreme(\partial\Y))$ by the Krein-Milman theorem \cite{Grothendieck1973} since $\Y$ is compact, so that $y_1,y_2\in\partial\Y$).

Since \rev{$y_1,y_2\in\partial\Y$ and} \rev{a ball of radius $R$ rolls in} $\Y$, there exist two balls $B(\bar{y}_1,R)$ and $B(\bar{y}_2,R)$ satisfying $y_1\in B(\bar{y}_1,R)\subseteq\Y\subseteq\hull(\Y)$ and $y_2\in B(\bar{y}_2,R)\subseteq\Y\subseteq\hull(\Y)$\rev{, see Figure \ref{fig:tangentballmilman}}. 
Define $\bar{y}=t\bar{y}_1+(1-t)\bar{y_2}$. 
Then, the ball $B(\bar{y},R)$ satisfies $y\in B(\bar{y},R)\subseteq\hull(\Y)$, since $B(\bar{y},R)\subset\hull(B(\bar{y}_1,R)\cup B(\bar{y}_2,R))\subseteq\hull(\Y)$.
\end{itemize}
\end{minipage}% 
\hspace{0.045\linewidth}
\begin{minipage}{0.3\linewidth}
\vspace{4mm}
    \centering	\includegraphics[width=1\linewidth]{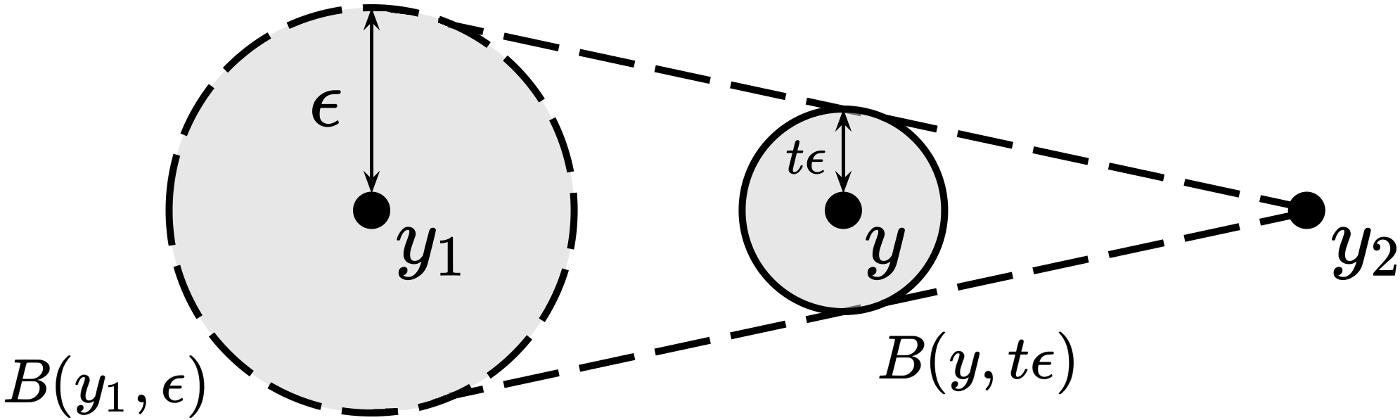}
    \captionof{figure}{\rev{If $y=ty_1+(1-t)y_2$, then $B(y,t\epsilon)\subseteq\hull(B(y_1,\epsilon)\cup y_2)$.}}
    \label{fig:cone}
    \centering
    \includegraphics[width=1\linewidth]{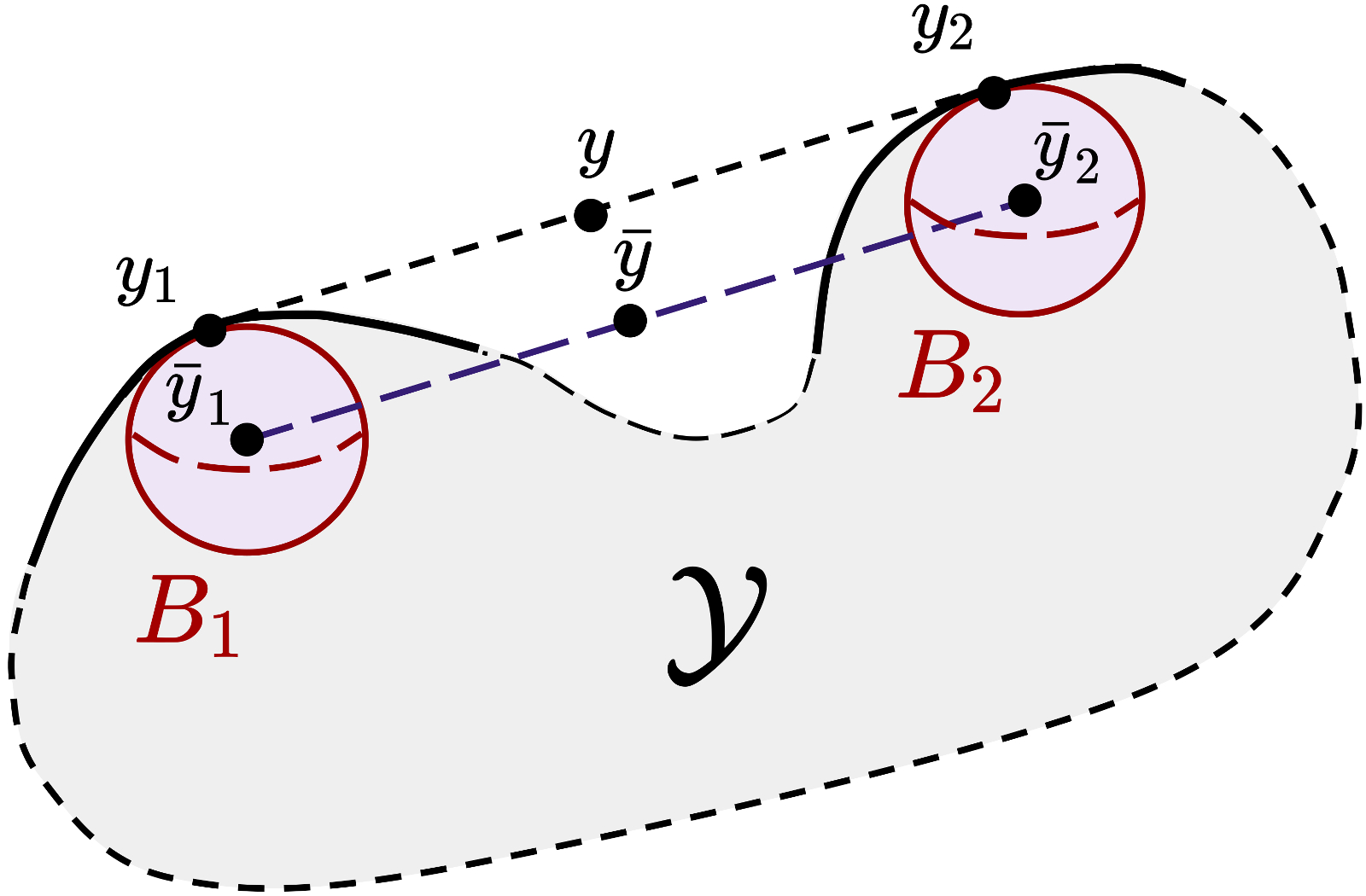}
    \captionof{figure}{Finding a ball tangent to \rev{$\partial\hull(\Y)\setminus\partial\Y$}.}
    \label{fig:tangentballmilman}
\end{minipage}% 

By the last two steps, a ball of radius $R$ rolls freely in $\hull(\Y)$ and in $\overline{\hull(\Y)^\comp}$. 
We conclude that $\hull(\Y)$ is $R$-smooth.
% \begin{figure}[t]
%   \begin{subfigure}{0.3\linewidth}
%   \centering
%     \includegraphics[width=1\linewidth]{figs/PartialHY.jpg}
%     \caption{}
%     \label{fig:decomposition_hull}
%       % \label{fig:ift:proof_of_thm:submersion_rsmooth}
%   \end{subfigure}
%   %
%   \hspace{0.01\linewidth}
%   \begin{subfigure}{0.3\linewidth}
%   \centering
%     \includegraphics[width=1\linewidth]{figs/cone.jpg}
%       \caption{}
%   \end{subfigure}
%   \hspace{0.01\linewidth}
%   %
%   \begin{subfigure}{0.3\linewidth}
%   \centering
%     \includegraphics[width=1\linewidth]{figs/HullBallsBoundary.jpg}
%       \caption{}
%       \label{fig:tangentballmilman}
%   \end{subfigure}
%   \caption[Images]{(a) Decomposition of the boundary of $\hull(\Y)$. (b) \rev{If $y=ty_1+(1-t)y_2$, then $B(y,t\epsilon)\subseteq\hull(B(y_1,\epsilon)\cup y_2)$.} 
%   (c) Finding a ball tangent to \rev{$\partial\hull(\Y)\setminus\partial\Y$}.}
%   % \label{fig:geometric_smoothness}
% \end{figure}
\end{proof}

% \begin{theorem}\label{thm:submersion_rsmooth}
% Let $r>0$, $\X\subset\R^m$ be a non-empty compact set \revFirst{such that a ball of radius $r$ rolls freely in $\X$}, 
% $f:\R^m\to\R^n$ be a $C^1$ submersion such that $(f,\dd f)$ are $(\bar{L},\bar{H})$-Lipschitz, and $\Y=f(\X)$. 
% \rev{Then, for some $R>0$, a ball of radius $R$ rolls freely in $\Y$, and $\hull(\Y)$ is $R$-smooth.}
% \end{theorem}
% \subsection{\rev{Proof of Theorem \ref{thm:submersion_rsmooth} (a ball rolls freely in $\Y$ and $\hull(\Y)$ is $R$-smooth)}}
% \noindent 

\subsection{\rev{Proof of Theorem \ref{thm:submersion_rsmooth} ($\hull(\Y)$ is $R$-smooth)}}\label{sec:smoothness_output:conclusion}
\rev{Theorem \ref{thm:submersion_rsmooth} follows directly from Lemmas \ref{lem:submersion_rolling_ball_Y} and \ref{lem:rsmooth:hull}.}
% 
% \revFirst{The $r$-smoothness assumption in Theorem \ref{thm:submersion_rsmooth} can be relaxed to only assuming that a ball of radius $r>0$ rolls freely in $\X$. We will not use this result later  do not pursue} \red{actually, keep it}
% To prove Theorem \ref{thm:submersion_rsmooth}, we need the following intermediate results. % 

% 
% 
% 
% 
% 
\begin{proof}[Proof of Theorem \ref{thm:submersion_rsmooth}]
\rev{By Lemma \ref{lem:submersion_rolling_ball_Y}, a ball of radius $R>0$ rolls freely in $\Y$. The conclusion follows from Lemma \ref{lem:rsmooth:hull}.} 
% \rev{Thus, a ball of radius $R$  rolls freely in $\hull(\Y)$ (see the proof of Lemma \ref{lem:rsmooth:hull}). 
% Moreover, a ball of radius $R$ rolls freely \revFirst{in} $\overline{\hull(\Y)^\comp}$  for any $R>0$ by Lemma \ref{lem:convex_rolling_outside}, since $\hull(\Y)$ is convex.} This concludes the proof. 
\end{proof}

\begin{corollary}\label{cor:sub_hull:boundary_manifold}
Let $r>0$, $\X\subset\R^m$ be a non-empty compact set \rev{such that a ball of radius $r$ rolls freely in $\X$}, 
$f:\R^m\to\R^n$ be a $C^1$ submersion, and $\Y=f(\X)$. 
Then, $\partial\hull(\Y)$ is a submanifold of dimension $(n-1)$.
\end{corollary}
\begin{proof}
$\hull(\Y)$ is $R$-smooth for some $R>0$ by Theorem \ref{thm:submersion_rsmooth}. Since $\hull(\Y)$ is always path-connected since it is convex, the conclusion follows from Theorem \ref{thm:walther1999}.
\end{proof}

Albeit $\hull(\Y)$ is $R$-smooth, the constant $R>0$ in Theorem \ref{thm:submersion_rsmooth} depends on the smoothness of local charts given by the rank theorem \rev{(see the proof of Lemma \ref{lem:submersion_rolling_ball_Y})}.  These smoothness constants are difficult to characterize. The error bounds we derive in Theorem \ref{thm:error_bound:smooth_boundary} are independent of this constant. Nevertheless, the fact that $\partial\hull(\Y)$ is an $(n-1)$-dimensional submanifold will be essential to our derivations \rev{in the next section}.

\section{Error bounds}\label{sec:error_bounds}
In this section, we derive error bounds for the reconstruction of the convex hull of $\Y=f(\X)$ from a sample. In Section \ref{sec:error_bounds:lipschitz}, as a baseline, we first derive a coarse error bound using a standard $\delta$-covering argument. % 
In Section \ref{sec:error_bounds:smooth}, we derive tighter error bounds that exploit the smoothness of the boundary $\partial\hull(\Y)$ \rev{and prove Theorem \ref{thm:error_bound:smooth_boundary}}. \rev{In Sections \ref{sec:thm:error_bound:smooth_boundary:discussion} and \ref{sec:error_bounds:diffeo}, we discuss Theorem \ref{thm:error_bound:smooth_boundary} and give a simpler proof of a more conservative  error bound under the additional assumption that $f$ is a diffeomorphism.}

\subsection{Naive error bound via covering ($f$ is Lipschitz)}\label{sec:error_bounds:lipschitz}
Lemma \ref{lem:Hausdorff_hulls} provides sufficient conditions to obtain an $\epsilon$-accurate approximation of the convex hull.
\begin{lemma}% 
\label{lem:Hausdorff_hulls}
Let $\epsilon\geq 0$, $A,\Y\in\K$ satisfy
$\partial\Y\subseteq A+B(0,\epsilon)$ and $A\subseteq \Y$. 
Then, $\dH(\hull(\Y), \hull(A))\leq \epsilon$.
\end{lemma}

\begin{proof} An equivalent definition of the Hausdorff distance % 
is $\dH(\hull(\Y),\hull(A))=\min\{\epsilon\geq 0: \hull(\Y)\subseteq\hull(A)+B(0,\epsilon),\hull(A)\subseteq\hull(\Y)+B(0,\epsilon)\}$ \cite{Schneider2014}.

$\partial\Y\subseteq A+B(0,\epsilon)$ implies that $\hull(\Y)=\hull(\partial\Y)\subseteq\rev{\hull(A+B(0,\epsilon))\subseteq\hull(A)+B(0,\epsilon)}$. 
$A\subseteq\Y$ implies % 
that 
$\hull(A)\subseteq\hull(\Y)\subset\hull(\Y)+B(0,\epsilon)$. 
Thus, together, 
$\hull(\Y)\subseteq\hull(A)+B(0,\epsilon)$
and $\hull(A)\subseteq\hull(\Y)+B(0,\epsilon)$
imply that $\dH(\hull(A),\hull(\Y))\leq\epsilon$. 
\end{proof}

Combined with a standard covering argument, Lemma \ref{lem:Hausdorff_hulls} allows deriving an error bound for the convex hull reconstruction from a sample $Z_\delta=\{x_i\}_{i=1}^M$ in $\X$.

\begin{lemma}\label{lem:error_bound:covering_lipschitz}
Let $\X\subset\R^m$ be a non-empty compact set, 
$f:\R^m\to\R^n$ be a $\bar{L}$-Lipschitz function, 
$\Y=f(\X)$, $\delta>0$, and $Z_\delta\subset\X$. Assuming that either
\begin{itemize}[leftmargin=6mm]\setlength\itemsep{0.5mm}
\item $Z_\delta$ is a $\delta$-cover of $\X$ (i.e., $\X\subseteq Z_\delta+B(0,\delta)$), or
\item $Z_\delta$ is a $\delta$-cover of $\partial\X$ (i.e., $\partial\X\subseteq Z_\delta+B(0,\delta)$) and $\partial\Y\subseteq f(\partial\X)$,% 
\end{itemize}
then
$$
\dH(\hull(\Y),\hull(f(Z_\delta)))\leq\bar{L}\delta.
$$
\end{lemma}
\begin{proof}

In both cases, $f(Z_\delta)$ is an $(\bar{L}\delta)$-cover of $\partial\Y$. % 
Thus, $\partial\Y\subseteq f(Z_\delta)+B(0,\bar{L}\delta)$. Also, $f(Z_\delta)\subseteq\Y$ since $f(Z_\delta)\subseteq f(\X)=\Y$.  
The conclusion follows from Lemma \ref{lem:Hausdorff_hulls} with $A=f(Z_\delta)$ and $\epsilon=\bar{L}\delta$. 
\end{proof}

\subsection{Error bound via smoothness of the boundary ($f$ is $C^1$ and $\X$ is $r$-smooth)}\label{sec:error_bounds:smooth}
\rev{In this section, we prove Theorem \ref{thm:error_bound:smooth_boundary}. We proceed in three steps.}

\rev{First, we show that a bound on the distance to the tangent space $T_y\partial\hull(\Y)$ gives a Hausdorff distance error bound on the convex hull reconstruction.}
\begin{lemma}\label{lem:tangent_implies_hausdorff}
\rev{Let $\Y,A\subset\R^n$ be non-empty compact sets such that $A\subseteq\hull(\Y)$ and $\partial\hull(\Y)$ is a submanifold of dimension $(n-1)$. 
Then,
}
\begin{equation}\label{eq:tangent_implies_hausdorff}
\rev{\dH(\hull(\Y),\hull(A))\leq\sup_{y\in\partial\hull(\Y)\cap\partial\Y}\left(\Inf_{a\in A}d_{T_y\partial\hull(\Y)}(y-a)\right).
}
\end{equation}
\end{lemma}

\newpage
\rev{Second, we derive a bound % 
on distances to the tangent spaces % 
$T_y\partial\hull(\Y)$.
\begin{lemma}%[Bound on $d_{T_y\partial\hull(\Y)}$ if $f$ is a submersion]
\label{lem:bound_delta_Tphull}
\rev{Let $r>0$,  $\X\subset\R^m$ be a non-empty path-connected $r$-smooth compact set, 
$f:\R^m\to\R^n$ be a $C^1$ submersion such that $(f,\dd f)$ are $(\bar{L},\bar{H})$-Lipschitz, 
$\Y=f(\X)$, and  
$x,z\in\partial\X$ \rev{be such that $y=f(x)\in\partial\Y\cap\partial\hull(\Y)$}. Then, 
}
\begin{equation}
\rev{d_{T_y\partial\hull(\Y)}(f(z)-\rev{f(x)})\leq \frac{1}{2}\left(\frac{\bar{L}}{r}+\bar{H}\right)\|z-x\|^2.}
\end{equation}
\end{lemma}
Third, we combine Lemmas \ref{lem:tangent_implies_hausdorff} and \ref{lem:bound_delta_Tphull}, which gives Theorem \ref{thm:error_bound:smooth_boundary}.
}

\begin{remark}
[\rev{Does Theorem \ref{thm:error_bound:smooth_boundary} hold if a ball rolls freely in $\X$, but $\X$ is not $r$-smooth?}] \label{remark:rolling_ball_but_not_rsmooth} 
\rev{The facts that $\hull(\Y)$ is $R$-smooth (Theorem \ref{thm:submersion_rsmooth}) and $\partial\hull(\Y)$ is an $(n-1)$-dimensional submanifold (Corollary \ref{cor:sub_hull:boundary_manifold}) are crucial, as they ensure that the tangent spaces $T_y\partial\hull(\Y)$ are well-defined and $(n-1)$-dimensional (a property used in the proof of Lemma \ref{lem:tangent_implies_hausdorff}). 
Similarly, the $r$-smoothness of $\X$ ensures that the tangent spaces $T_x\partial\X$ are well-defined and facilitates the use of standard tools from differential geometry (see Lemma \ref{lem:tangent_mapped_to_tangent_hull}). We conjecture that Theorem \ref{thm:error_bound:smooth_boundary} holds if a ball of radius $r>0$ rolls freely in $\X$, since $\hull(\Y)$ is $R$-smooth under this assumption (Theorem \ref{thm:submersion_rsmooth}). Proving this result, if it holds, would require a significantly different approach.}
\end{remark}
\subsubsection{Bound on $d_{T_y\partial\hull(\Y)}$ implies bound on Hausdorff distance %$\dH(\hull(\Y),\hull(A))$
\rev{(proof of Lemma \ref{lem:tangent_implies_hausdorff})}}
\rev{The proof of Lemma \ref{lem:tangent_implies_hausdorff} relies on the following result.}
% In this section, we show that a bound on the distance to the tangent space $T_y\partial\hull(\Y)$ gives a Hausdorff distance error bound on the convex hull reconstruction.

\begin{lemma}\label{lem:dist_TxM_normal}
Let $\M\subset\R^n$ be a compact submanifold of dimension $(n-1)$, $x\in\M$, and \rev{$v\in\R^n$}. \rev{Let $n^\M(x)\in N_x\M$ be a} unit-norm normal vector of $\M$ at $x$, and $\pi_x^\top$ and $\pi_x^\perp$ \rev{be} the linear projection operators onto $T_x\M$ and $N_x\M$, respectively. Then, $d_{T_x\M}(v)
=\|\pi_x^\perp(v)\|=
|v^\top n^\M(x)
|$.
\end{lemma}
\begin{proof}
$d_{T_x\M}(v)
=
\inf_{w\in T_x\M}\|v-w\|
=
\|v-\pi_x^\top(v)\|$. Then, 
note that $(v-\pi_x^\top(v))\in N_x\M$ since $\pi_x^\top(v-\pi_x^\top(v))=0$ by linearity. Thus,  $v-\pi_x^\top(v)=\pi_x^\perp(v-\pi_x^\top(v))$ and we obtain that $d_{T_x\M}(v)
=
\|\pi_x^\perp(v)\|$. 
    To show that $\|\pi_x^\perp(v)\|=|v^\top n^\M(x)|$, we write $T_x\M=\{v\in\R^n:v^\top n^\M(x)=0\}$. Then, $\pi^\top(v)=v-(v^\top n^\M(x))n^\M(x)$ so $\pi^\perp(v)=(v^\top n^\M(x))n^\M(x)$ and the conclusion follows. 
\end{proof}

% \begin{lemma}\label{lem:tangent_implies_hausdorff}
% Let $\Y,A\subset\R^n$ be non-empty compact sets such that $A\subseteq\hull(\Y)$ and $\partial\hull(\Y)$ is a submanifold of dimension $(n-1)$. 
% Then,
% \begin{equation}\label{eq:tangent_implies_hausdorff}
% \dH(\hull(\Y),\hull(A))\leq\sup_{y\in\partial\hull(\Y)\cap\partial\Y}\left(\Inf_{a\in A}d_{T_y\partial\hull(\Y)}(y-a)\right).
% \end{equation}
% \end{lemma}
\begin{proof}[\rev{Proof of Lemma \ref{lem:tangent_implies_hausdorff}}]
For any non-empty convex compact set $C\subset\R^n$, the support function of $C$ is defined as 
$h(C,\cdot):\R^n\to\R, \ u\mapsto h(C,u)=\sup_{y\in C}y^\top u$  
and is convex and continuous \cite{Schneider2014}. 
$\hull(\Y)$ and $\hull(A)$ are both convex, non-empty, and compact. 
Thus, by \cite[Lemma 1.8.14]{Schneider2014}, 
\begin{align*}
\dH(\hull(\Y),\hull(A))
&=
\sup_{u\in\partial B(0,1)}
	|h(\hull(\Y),u)-h(\hull(A),u)|
=
	|h(\hull(\Y),u_0)-h(\hull(A),u_0)|
\\
&=
\bigg|\sup_{y\in\hull(\Y)}y^\top u_0-\sup_{q\in\hull(A)}q^\top u_0\bigg|
\end{align*}
for some $u_0\in\R^n$ with $\|u_0\|=1$. 

Let $y_0\in\partial\hull(\Y)$ be such that $h(\hull(\Y),u_0)=y_0^\top u_0$, so $y^\top u_0\leq y_0^\top u_0$ for all $y\in\hull(\Y)$. Then, $u_0=n^{\partial\hull(\Y)}(y_0)$ \cite{Dumbgen1996}, where $n^{\partial\hull(\Y)}$ denotes the unit-norm outward-pointing normal of $\partial\hull(\Y)$\rev{, and %, and $(y_0-y)^\top n^{\partial\hull(\Y)}(y_0)\geq 0$ for all $y\in\hull(\Y)$. % 
\begin{align}\label{eq:haussdorff_equality_normal_vectors_sup}
\dH(\hull(\Y),\hull(A))
&=
\bigg|y_0^\top n^{\partial\hull(\Y)}(y_0)-\sup_{q\in \hull(A)}q^\top n^{\partial\hull(\Y)}(y_0)\bigg|.
\end{align}
}

% \begin{minipage}{0.75\linewidth}
Without loss of generality, we may assume that $y_0\in\partial\hull(\Y)\cap\partial\Y$. Indeed, if $y_0\notin\partial\Y$, then there exists $y_1\in\partial\hull(\Y)\cap\partial\Y$ with \rev{$y_0^\top n^{\partial\hull(\Y)}(y_0)=y_1^\top n^{\partial\hull(\Y)}(y_1)$  and $n^{\partial\hull(\Y)}(y_0)=n^{\partial\hull(\Y)}(y_1)$}. %\footnote{%Indeed, by the Krein-Milman theorem and since $\Y$ is compact, $\hull(\Y)=\hull(\Extreme(\partial\Y))$. % 
\rev{Indeed, since $y_0\in\hull(\Y)$, there exist $y_1,y_2\in\partial\Y$ and $t\in(0,1)$ such that $y_0=ty_1+(1-t)y_2$ (see Lemma \ref{lem:rsmooth:hull} for a proof that $y_1,y_2\in\partial\Y$) and} %. Moreover, as shown in the proof of Lemma \ref{lem:rsmooth:hull}, $y_1,y_2\in\partial\Y$.} 
 % Thus, there exist two extreme points $y_1,y_2\in\partial\Y$ such that $y_0=ty_1+(1-t)y_2$ for some $t\in(0,1)$ and 
 such that the line $L=\{y_s=sy_1 +(1-s)y_2, s\in(0,1)\}$ satisfies $L\subseteq\hull(\Y)$. 
Since \rev{$y_0\in\partial\hull(\Y)$}, $(y-y_0)^\top n^{\partial\hull(\Y)}(y_0)\leq 0$ for all $y\in\hull(\Y)$, so $(y_s-y_0)^\top n^{\partial\hull(\Y)}(y_0)\leq 0$ for all $s\in(0,1)$. 
Plugging in the values for $y_0$ and $y_s$, we obtain that $(s-t)(y_1-y_2)^\top n^{\partial\hull(\Y)}(y_0)\leq 0$ for all $s\in(0,1)$, 
which implies that $(y_1-y_2)^\top n^{\partial\hull(\Y)}(y_0)=0$. Since $y_2=(y_0-ty_1)/(1-t)$, \rev{we obtain} $(y_0-y_1)^\top n^{\partial\hull(\Y)}(y_0)=0$. % }
\rev{From $(y_0-y_1)^\top n^{\partial\hull(\Y)}(y_0)=0$, we obtain $n^{\partial\hull(\Y)}(y_0)=n^{\partial\hull(\Y)}(y_1)$ since $\partial\hull(\Y)$ is an $(n-1)$-dimensional submanifold. Thus, without loss of generality, $y_0\in\partial\hull(\Y)\cap\partial\Y$.}
% \end{minipage}
% \begin{minipage}{0.2\linewidth}
%     \centering  
%     \includegraphics[width=1\linewidth]{figs/normal_vectors_on_hull_boundary.jpg}
%   \caption{Birds}
% \end{minipage}
% 
% 
% 
% 
% 
% 
% 
% 
% 
% 
% 
% 
% 
% 
% 
% 
% 
% 
% 
% 
% 

Thus, for some $y_0\in\partial\hull(\Y)\cap\partial\Y$, 
\begin{align*}
\dH(\hull(\Y),\hull(A))
% &=
% \bigg|y_0^\top n^{\partial\hull(\Y)}(y_0)-\sup_{q\in \hull(A)}q^\top n^{\partial\hull(\Y)}(y_0)\bigg|
% % 
% % 
\mathop{=}^{\eqref{eq:haussdorff_equality_normal_vectors_sup}}
\bigg|\inf_{q\in \hull(A)} (y_0-q)^\top n^{\partial\hull(\Y)}(y_0)\bigg|
=
\inf_{q\in \hull(A)} 
(y_0-q)^\top n^{\partial\hull(\Y)}(y_0),
\end{align*}
since $(y_0-q)^\top n^{\partial\hull(\Y)}(y_0)\geq 0$ for all $q\in \hull(A)$ because $\hull(A)\subseteq\hull(\Y)$.

% \rev{Also, $
% \inf_{q\in \hull(A)} 
% (y_0-q)^\top n^{\partial\hull(\Y)}(y_0)
% =
% \inf_{a\in A} 
% (y_0-a)^\top n^{\partial\hull(\Y)}(y_0) 
% $, since
% $
% \sup_{q\in \hull(A)} 
% q^\top n^{\partial\hull(\Y)}(y_0)
% =
% \sup_{a\in A} 
% a^\top n^{\partial\hull(\Y)}(y_0) 
% $. Indeed,  with $u_0 = n^{\partial\hull(\Y)}(y_0)$, we have 
%     $\sup_{q\in \hull(A)} 
% q^\top u_0\geq
% \sup_{a\in A} 
% a^\top u_0$. 
%     To show the other inequality, let $q_0\in\partial\hull(A)$ be such that $q_0^\top u_0\geq q^\top u_0$ for all $q\in\hull(A)$. Then, there exist $q_1,q_2\in A$ and $t\in(0,1)$ such that $q_0=tq_1+(1-t)q_2$ with $q_1^\top u_0=q_0^\top u_0$ (see the proof above that $y_0^\top n^{\partial\hull(\Y)}(y_0)=y_1^\top n^{\partial\hull(\Y)}(y_1)$), so $\sup_{q\in \hull(A)} 
% q^\top u_0=q_0^\top u_0=q_1^\top u_0\leq\sup_{a\in A}a^\top u_0$. 
% Thus,}
\rev{Moreover, $\inf_{q\in \hull(A)} 
(y_0-q)^\top n^{\partial\hull(\Y)}(y_0)
\leq
\inf_{a\in A} 
(y_0-a)^\top n^{\partial\hull(\Y)}(y_0) 
$. Thus,}
\begin{align*}
\dH(\hull(\Y),\hull(A)\rev{)
\leq
\inf_{a\in A}}
(y_0-a)^\top n^{\partial\hull(\Y)}(y_0).
\end{align*}

Since $(y_0-a)^\top n^{\partial\hull(\Y)}(y_0)\geq 0$ for all $a\in A$ because $A\subseteq\hull(\Y)$, we obtain 
\begin{align*}
\dH(\hull(\Y),\hull(A))
&\ \rev{\leq}
\inf_{a\in A}
(y_0-a)^\top n^{\partial\hull(\Y)}(y_0)
=
\inf_{a\in A} \bigg|(y_0-a)^\top n^{\partial\hull(\Y)}(y_0)\bigg|
\\
&=
\inf_{a\in A}d_{T_{y_0}\partial\hull(\Y)}(y_0-a)
\end{align*}
for some $y_0\in\partial\hull(\Y)\cap\partial\Y$, where the last equality follows from Lemma \ref{lem:dist_TxM_normal}. 
The conclusion follows.
\end{proof}
Lemma \ref{lem:tangent_implies_hausdorff} implies that bounding the distance to the tangent space $T_y\partial\hull(\Y)$ at all $y\in\partial\hull(\Y)\cap\partial\Y$ suffices to obtain a bound on the Hausdorff distance. Indeed, $\partial\hull(\Y)\cap\partial\Y$ is compact and $y\mapsto\inf_{a\in A} d_{T_{\partial\hull(\Y)}}(y-a)=(y-a)^\top n^{\partial\hull(\Y)}(y)$ is continuous ($n^{\partial\hull(\Y)}(y)$ is continuous since $\partial\hull(\Y)$ is a submanifold \cite{Lee2018}), so the supremum in \eqref{eq:tangent_implies_hausdorff} is attained at some $y\in\partial\hull(\Y)\cap\partial\Y$.

\subsubsection{Bound on $d_{T_y\partial\hull(\Y)}$ if $f$ is a submersion \rev{(proof of Lemma \ref{lem:bound_delta_Tphull})}}\label{sec:error_bounds:bound_on_deltaTphull}
% The main result of this section is the following.
% \begin{lemma}[Bound on $d_{T_y\partial\hull(\Y)}$ if $f$ is a submersion]\label{lem:bound_delta_Tphull}
% Let $r>0$,  $\X\subset\R^m$ be a non-empty path-connected $r$-smooth compact set, 
% $f:\R^m\to\R^n$ be a $C^1$ submersion such that $(f,\dd f)$ are $(\bar{L},\bar{H})$-Lipschitz, 
% $\Y=f(\X)$, and  
% $x,z\in\partial\X$ \rev{be such that $y=f(x)\in\partial\Y\cap\partial\hull(\Y)$}. Then, 
% \begin{equation*}
% % 
% d_{T_y\partial\hull(\Y)}(f(z)-\rev{f(x)})\leq \frac{1}{2}\left(\frac{\bar{L}}{r}+\bar{H}\right)\|z-x\|^2.
% \end{equation*}
% \end{lemma}
The proof of Lemma \ref{lem:bound_delta_Tphull} relies on Lemmas  \ref{lem:curve_intersects_ball} and \ref{lem:tangent_mapped_to_tangent_hull}, whose proofs are given in Section \ref{sec:error_bounds:bound_on_deltaTphull:proofs}.

\begin{lemma}[Curve intersecting a ball]\label{lem:curve_intersects_ball}
Let $r>0$, \rev{$B=B(0,r)\subset\R^n$}, $p\in\partial B$,  and $v\notin T_p\partial B$. 
Let $\epsilon>0$ and $\gamma:(-\epsilon,\epsilon)\to\R^n$ be a smooth curve with $\gamma(0)=p$ and $\gamma'(0)=v$. 

Then, there exists  $t\in(-\epsilon,\epsilon)$ such that $\gamma(t)\in\Int(B)$.
\end{lemma}

\begin{lemma}\label{lem:tangent_mapped_to_tangent_hull}
Let $r>0$, $\X\subset\R^m$ be a non-empty path-connected $r$-smooth compact set, 
$f:\R^m\to\R^n$ be a $C^1$ submersion, $\Y=f(\X)$, \rev{and $x\in\partial\X$ be such that $y=f(x)\in\partial\Y\cap\partial\hull(\Y)$}. 
Then, $$\dd f_x(T_x\partial\X)=T_y\partial\hull(\Y).
$$
\end{lemma}
Lemma \ref{lem:tangent_mapped_to_tangent_hull} is an important result, as it implies that the image of a geodesic in $\partial\X$ through a submersion $f$ is locally tangent to $\partial\hull(\Y)$. This allows deriving a tight error bound on $d_{T_y\partial\hull(\Y)}$.

\begin{proof}[Proof of Lemma \ref{lem:bound_delta_Tphull}]

Let $t=\|z-x\|$ and $\gamma:[0,t]\to\R^m$ be defined as
$$
\gamma(s)=x+s\frac{z-x}{\|z-x\|},
\quad s\in[0,t]
$$
such that $\gamma(0)=x$, $\gamma(t)=z$,  $\|\gamma'(s)\|=1$, and $\gamma''(s)=0$ for all $s\in[0,t]$. Then, assuming that $f$ is $C^2$ for conciseness (we discuss modifications for the case where $f$ is only $C^1$ in Section \ref{sec:lem:bound_delta_Tphull:proof_C1}), 
\begin{align*}
f(z)-f(x) &=
(f\circ\gamma)(t)-(f\circ\gamma)(0)
=
\int_0^t (f\circ\gamma)'(s)\dd s
\\
&=
\int_0^t \left((f\circ\gamma)'(0)+\int_0^s(f\circ\gamma)''(u)\dd u\right)\dd s
\\
&=
t(f\circ\gamma)'(0)
+
\int_0^t \int_0^s\left((f\circ\gamma)''(u)\right)\dd u\dd s
\\
&=
t\,\dd f_{\gamma(0)}(\gamma'(0))
+
\int_0^t \int_0^s\left(\frac{\dd}{\dd u}
\left(\dd f_{\gamma(u)}(\gamma'(u))\right)\right)\dd u\dd s
\\
&=
t\,\dd f_{\gamma(0)}(\gamma'(0))
+
\int_0^t \int_0^s\left(\dd^2 f_{\gamma(u)}(\gamma'(u),\gamma'(u)) + \dd f_{\gamma(u)}(\gamma''(u))\right)\dd u\dd s
\\
&=
t\,\dd f_{\gamma(0)}(\gamma'(0))
+
\int_0^t \int_0^s\dd^2 f_{\gamma(u)}(\gamma'(u),\gamma'(u))\dd u\dd s
\end{align*} 
by the chain rule. % 
By Lemma \ref{lem:dist_TxM_normal} \rev{($\partial\hull(\Y)$ is an $(n-1)$-dimensional submanifold by Corollary \ref{cor:sub_hull:boundary_manifold})}, denoting by $\pi_y^\perp$ the linear projection operator onto $N_y\partial\hull(\Y)$,
\begin{align}
d_{T_y\partial\hull(\Y)}(f(z)-y)
&=
\left\|\pi_y^\perp(f(z)-f(x))\right\|
\nonumber
\\
&=
\left\|\pi^\perp_y\left(
t\,\dd f_{\gamma(0)}(\gamma'(0))
+
\int_0^t \int_0^s\dd^2 f_{\gamma(u)}(\gamma'(u),\gamma'(u))\dd u\dd s
\right)\right\|
\nonumber
\\
&=
\left\|t% 
\pi^\perp_y\left(
    \dd f_{\gamma(0)}(\gamma'(0))
    \right)% 
+
\int_0^t \int_0^s\pi^\perp_y\left(\dd^2 f_{\gamma(u)}(\gamma'(u),\gamma'(u))\right)\dd u\dd s
\right\|
\nonumber
\\
&\leq
t\left\| 
\pi^\perp_y\left(
    \dd f_{\gamma(0)}(\gamma'(0))
    \right)% 
\right\|
+
\left\| 
\int_0^t \int_0^s\pi^\perp_y\left(\dd^2 f_{\gamma(u)}(\gamma'(u),\gamma'(u))\right)\dd u\dd s
\right\|,
\label{eq:lem:bound_delta_Tphull:1}
\end{align}
where the third equality follows from the linearity of $\pi^\perp_y$. We bound the two terms in \eqref{eq:lem:bound_delta_Tphull:1} below:
\begin{itemize}[leftmargin=5mm]
    \item To bound the first term, we first decompose $\gamma'(0)$ into its tangential component $\gamma'(0)^\parallel\in T_x\partial\X$ and its normal component $\gamma'(0)^\perp\in N_x\partial\X$, so that $\gamma'(0)=\gamma'(0)^\parallel+\gamma'(0)^\perp$. Then, by linearity,
    \begin{align*}
    \pi^\perp_y\left(
    \dd f_{\gamma(0)}(\gamma'(0))
    \right)
    &=
    \pi^\perp_y\left(
    \dd f_{\gamma(0)}\left(\gamma'(0)^\parallel+\gamma'(0)^\perp\right)
    \right)
    \\
    &=
    \pi^\perp_y\left(
    \dd f_{\gamma(0)}\left(\gamma'(0)^\parallel\right)
    \right)
    +
    \pi^\perp_y\left(
    \dd f_{\gamma(0)}\left(\gamma'(0)^\perp\right)
    \right)
    \\
    &=
    \pi^\perp_y\left(
    \dd f_{\gamma(0)}\left(\gamma'(0)^\perp\right)
    \right),
    \end{align*}
    where we used the fact that $\pi^\perp_y\left(
    \dd f_{\gamma(0)}\left(\gamma'(0)^\parallel\right)
    \right)=0$ since $\dd f_{\gamma(0)}\left(\gamma'(0)^\parallel\right)\in  T_y\partial\hull(\Y)$ thanks to Lemma \ref{lem:tangent_mapped_to_tangent_hull}. Thus, since $f$ is $\bar{L}$-Lipschitz and $\gamma'(0)=(z-x)/\|z-x\|$,
    \begin{align}
    \left\|
    \pi^\perp_y\left(
    \dd f_{\gamma(0)}(\gamma'(0))
    \right)
    \right\|
    &=
    \left\|\pi^\perp_y\left(
    \dd f_{\gamma(0)}\left(\gamma'(0)^\perp\right)
    \right)
    \right\|
    \leq
    \left\|
    \dd f_{\gamma(0)}\left(\gamma'(0)^\perp\right)
    \right\|
\nonumber
    \\
    &\leq
    \bar{L}\left\|\gamma'(0)^\perp\right\|
\nonumber
    \\
    &\leq
    \frac{\bar{L}}{\|z-x\|}\left\|(z-x)^\perp\right\|
\nonumber
    \\
    &\leq
    \frac{\bar{L}}{\|z-x\|}\frac{\|z-x\|^2}{2r}
    =
    \frac{\bar{L}t}{2r},
\label{eq:lem:bound_delta_Tphull:2}
    \end{align}
    where the last inequality follows from \rev{Lemma \ref{lem:roll_implies_reach} and} Theorem \ref{thm:reach} since \rev{$\X$} is $r$-smooth. 
\item To bound the second term, since $\dd f$ is $\bar{H}$-Lipschitz and $\|\gamma'(u)\|=1$ for all $u\in[0,t]$,
\begin{align}
\left\| 
\int_0^t \int_0^s\pi^\perp_y\left(\dd^2 f_{\gamma(u)}(\gamma'(u),\gamma'(u))\right)\dd u\dd s
\right\|
&\leq
\int_0^t \int_0^s
\left\|\pi^\perp_y\left(\dd^2 f_{\gamma(u)}(\gamma'(u),\gamma'(u))
\right)
\right\|\dd u\dd s
\nonumber
\\
&\leq
\int_0^t \int_0^s
\left\|\dd^2 f_{\gamma(u)}(\gamma'(u),\gamma'(u)) 
\right\|\dd u\dd s
\nonumber
\\
&\leq
	\int_0^t 
	\int_0^s
\left(\bar{H} \|\gamma'(u)\|^2\right)
	\dd u\dd s
\nonumber
\\
&=
\bar{H}\frac{t^2}{2}.
\label{eq:lem:bound_delta_Tphull:3}
\end{align}
\end{itemize}
Combining \eqref{eq:lem:bound_delta_Tphull:1}-\eqref{eq:lem:bound_delta_Tphull:3}, we obtain 
$
d_{T_y\partial\hull(\Y)}(f(z)-y)
\leq
\frac{1}{2}\big(\frac{\bar{L}}{r} + \bar{H}\big)t^2
$, where \rev{$t=\|z-x\|$}. 
\end{proof}

\subsubsection{Hausdorff distance error bound (proof of Theorem \ref{thm:error_bound:smooth_boundary})}
We combine the results from the last two sections and obtain Theorem \ref{thm:error_bound:smooth_boundary}. % 
\begin{proof}[Proof of Theorem \ref{thm:error_bound:smooth_boundary}] 
Let $y\in\partial\hull(\Y)\cap\partial\Y$. Then, there exists $x\in\partial\X$ such that $y=f(x)$ ($\partial\Y\subseteq f(\partial\X)$ by Lemmas \ref{lem:submersion_open_map} and \ref{lem:open_map_boundary} since $f$ is a submersion). Then, there exists $z\in Z_\delta\subset\partial\X$ such that $\|x-z\|\leq\delta$. 
Thus, by Lemma \ref{lem:bound_delta_Tphull} ($\partial\hull(\Y)$ is an $(n-1)$-dimensional submanifold by Corollary \ref{cor:sub_hull:boundary_manifold}), we have 
$
d_{T_y\partial\hull(\Y)}(f(z)-y)
\leq 
\frac{1}{2}\big(\frac{\bar{L}}{r} + \bar{H}\big)\delta^2$.  
 We apply Lemma \ref{lem:tangent_implies_hausdorff} and conclude. 
\end{proof}

\subsection{Comments on Theorem \ref{thm:error_bound:smooth_boundary}}\label{sec:thm:error_bound:smooth_boundary:discussion}
First, the error bound from Theorem \ref{thm:error_bound:smooth_boundary} is quadratic in $\delta$. It is thus tighter than the naive Lipschitz-covering bound from Lemma \ref{lem:error_bound:covering_lipschitz} (that is linear in $\delta$)  for small values of $\delta$ (i.e., for a sufficiently-dense cover $Z_\delta$ so that $\delta\leq(2\bar{L})/(\bar{L}/r+\bar{H})$).

Second, the bound from Theorem \ref{thm:error_bound:smooth_boundary} is $2\times$ tighter than the bound \revFirst{in} \cite[Theorem 1]{Dumbgen1996}\revFirst{. Theorem \ref{thm:error_bound:smooth_boundary} also does not rely on convexity assumptions for $\X$ or for $\Y=f(\X)$, compared to \cite[Theorem 1]{Dumbgen1996} that only applies to convex problems}, see Theorem \ref{thm:dumbgen} and Corollary \ref{cor:dumbgen}. 
    \revFirst{These differences} follow from \revFirst{using} the bound in Lemma \ref{lem:tangent_implies_hausdorff} on the Hausdorff distance as a function of the distances to the tangent spaces of $\partial\hull(\Y)$. % 

    The bound \revFirst{in Theorem \ref{thm:error_bound:smooth_boundary}} does not depend on the smoothness of the inverse of $f$ (see also Lemma \ref{lem:bound_delta_Tphull:diffeo}), which could potentially be defined on a $n$-dimensional submanifold of $\R^m$ given by the rank theorem, see Theorem \ref{thm:submersion_rsmooth}. Such smoothness property would depend on local charts given by the rank theorem that would be difficult to characterize. The fact that the obtained bound only depends on the smoothness of $(f,\dd f)$ is desirable. We note that Theorem \ref{thm:submersion_rsmooth} is still needed to apply arguments from differential geometry, e.g., when bounding the distances to the tangent spaces $T_y\partial\hull(\Y)$.

\subsection{Naive \rev{error} bound for the diffeomorphism case}\label{sec:error_bounds:diffeo}
Before deriving the bound in Theorem \ref{thm:error_bound:smooth_boundary}, we studied the problem under the additional assumption that $f$ is a diffeomorphism. 
This assumption is restrictive, as it implies that $f$ maps between two spaces of the same dimension and thus does not apply to problems where the dimensionality of the input set $\X$ is larger than the dimensionality of the output set $\Y$. 
Nevertheless, this setting simplifies the analysis, as it prevents the presence of self-intersections in the image $\Y$ (see Example \ref{ex:intersection}) and allows directly obtaining a bound on the reach of $\Y$ (Lemma \ref{lemma:reach:diffeo}). This curvature bound for the boundary of $\Y$ yields a bound on the convex hull approximation error.

\begin{lemma}[Naive error bound if $f$ is a diffeomorphism]\label{lem:bound_delta_Tphull:diffeo}
Let $r,\delta>0$, $\X\subset\R^n$ be a non-empty path-connected $r$-smooth compact set, 
$f:\R^n\to\R^n$ be a $C^1$ diffeomorphism such that  $(f,f^{-1},\dd f)$ are $(\bar{L},\underline{L},\bar{H})$-Lipschitz, 
$\Y=f(\X)$, 
$x,z\in\partial\X$, and $y=f(x)\in\partial\Y$. % 
Then, 
\begin{equation}\label{eq:R:diffeo}
d_{T_y\partial\Y}(f(z)-y)\leq \frac{\|z-x\|^2}{2R_{\textrm{diffeo}}},
\quad
\text{where}
\quad
\frac{1}{R_{\textrm{diffeo}}}=% 
\left(\frac{\bar{L}}{r}+\bar{H}\right)(\underline{L}\bar{L})^2,
\end{equation}
In particular, if $y\in\partial\Y\cap\partial\hull(\Y)$, then $d_{T_y\partial\hull(\Y)}(f(z)-y)\leq \delta^2/(2R_{\textrm{diffeo}})$. Thus, if $Z_\delta\subset\partial\X$ is a $\delta$-cover of $\partial\X$, then $\dH(\hull(\Y),\hull(f(Z_\delta)))\leq \delta^2/(2R_{\textrm{diffeo}})$.
\end{lemma}

\begin{proof}
$\X$ is $r$-smooth, so $\reach(\partial\X)\geq r$ by Lemma \ref{lem:roll_implies_reach_boundary}.  
Thus, by Lemma \ref{lemma:reach:diffeo}, $\reach(\partial\Y)\geq\tilde{R}$ with $\tilde{R}^{-1}=\left(\frac{\bar{L}}{r}+\bar{H}\right)\underline{L}^2$ (note that $\partial\Y=f(\partial\X)$ since $f$ is a diffeomorphism). In addition, $\partial\X$ is a submanifold thanks to Theorem \ref{thm:walther1999}, so $\partial\Y$ is also a submanifold since $f$ is a diffeomorphism. 

Since $x,z\in\partial\X$ and $f$ is a diffeomorphism, $y,f(z)\in\partial\Y$. Also, $\|f(z)-y\|\leq\bar{L}\|z-x\|$. 
Thus, by Theorem \ref{thm:reach} applied to $\partial\Y$, we have $d_{T_y\partial\Y}(f(z)-y)\leq(\bar{L}\|z-x\|)^2/(2\tilde{R})=\|z-x\|^2/(2R_{\textrm{diffeo}})$. 

At $y\in\partial\Y\cap\partial\hull(\Y)$, we have $T_y\partial\Y=T_y\partial\hull(\Y)$ (since there exists a ball $B$ inside $\Y\subseteq\hull(\Y)$ that is tangent at $y$ to both $\partial\hull(\Y)$ and $\partial\Y$). 

The error bound on the Hausdorff distance follows from  \eqref{eq:R:diffeo} and Lemma \ref{lem:tangent_implies_hausdorff}.
\end{proof}
The error bound from Lemma \ref{lem:bound_delta_Tphull:diffeo} is more conservative than the error bound from Theorem \ref{thm:error_bound:smooth_boundary} by a factor $(\underline{L}\bar{L})^2$. This additional conservatism comes from the use of Lemma \ref{lemma:reach:diffeo} to bound the curvature of $\Y$. By directly working with the convex hull $\hull(\Y)$, the error bound in Theorem \ref{thm:error_bound:smooth_boundary} is tighter and also applies to submersions, although it requires more involved analysis.

\section{Applications}\label{sec:applications}
We apply our results to the problems of 
(1) geometric inference, wherein inputs $x_i\in\partial\X$ are randomly sampled from a distribution $\Prob_\X$ and one seeks a reconstruction of the convex hull of the output set, 
(2) reachability analysis of dynamical systems, (3) robust programming, wherein constraints should be satisfied for a given range parameters,  
and (4) robust optimal control of uncertain dynamical systems. 
For conciseness, proofs and details are deferred to Section \ref{sec:applications:appendix}. 
Code to reproduce experiments is available at \url{https://github.com/StanfordASL/convex-hull-estimation}.

\subsection{Geometric inference}\label{sec:applications:random_samples}
The error bound in Theorem \ref{thm:error_bound:smooth_boundary} implies the fast convergence of convex hull estimators of $\hull(\Y)$ from a random sample of inputs $x_i\in\partial\X$. 
We define the following.
\begin{itemize}[leftmargin=4mm]\setlength\itemsep{0.5mm}
\item Let $\Prob_\X$ be a probability measure on $(\R^m,\B(\R^m))$ with $\Prob_\X(\partial\X)=1$ (i.e., $\Prob_\X$ has support $\partial\X$). 
\item Let $\{x_i\}_{i=1}^M$ be $M\in\bN$ independent and identically-distributed (iid) inputs sampled from $\Prob_\X$.
\item Let $y_i=f(x_i)$ for $i=1,\dots,M$ and $\hat{\Y}^M=\hull\left(\{y_i\}_{i{=}1}^M\right)$.
\end{itemize}
$\hat\Y^M$ is a random compact set \cite{Molchanov_BookTheoryOfRandomSets2017}; we refer to Section \ref{sec:random_compact_set_estimator} for details. 
Intuitively, different sampled inputs $x_i(\omega)$  induce different sampled outputs $y_i(\omega)$, resulting in  different approximated compact sets $\hat\Y^M(\omega)\in\K$, where $\omega\in\Omega$ is drawn from a probability space $(\Omega,\G,\Prob)$.  

To derive high-probability bounds for the approximation error $\dH(\hull(\Y),\hat{\Y}^M)$, 
we need an assumption on the sampling distribution. % 
\begin{assumption}\label{assum:sampling_density:boundary}
Given $\delta>0$, there is a $\Lambda_\delta^{\partial\X}>0$ such that $\Prob_\X\left(B\left(x,\frac{\delta}{2}\right)\right)\geq\Lambda_\delta^{\partial\X}$ for all $x\in\partial\X$.
\end{assumption}
\noindent Assumption \ref{assum:sampling_density:boundary} gives a lower bound on sampling inputs that are $\delta/2$-close to any $x\in\partial\X$. In particular, it is satisfied if $\X$ is $r$-smooth and if the sampled inputs $x^i$ are drawn from a uniform distribution over $\partial\X$ \cite[Lemma III.23]{AamariPhD2017} or over $\X$ \cite[Lemma 6]{LewJansonEtAl2022}. Assumption % 
\ref{assum:sampling_density:boundary}, combined with a standard covering argument and a union bound, allows bounding the probability that the sample $X^M=\{x_i\}_{i=1}^M$ covers % 
$\partial\X$. 
\begin{lemma}\label{lem:covering_probability}
Let $\X\subset\R^m$ be a non-empty compact set, 
$\delta>0$, 
$N(\partial\X,\delta/2)$ denote the internal $(\delta/2)$-covering number of $\partial\X$\footnote{$N(\partial\X,\delta/2)$ denotes the minimum number of points $z_j\in\partial\X$ such that $\partial\X\subseteq\cup_{j}B(z_j,\delta/2)$.},  
$\Prob_\X$ be a probability measure over $\partial\X$ satisfying Assumption \ref{assum:sampling_density:boundary} with $\Lambda_\delta^{\partial\X}$, and
\begin{equation}\label{eq:delta:sampling_density:boundary}
\beta_{M,\delta}^{\partial\X}= 
N(\partial\X,\delta/2)(1 -
	\Lambda_\delta^{\partial\X}
)^M.
\end{equation}
Let $X^M=\{x_i\}_{i=1}^M$ be a sample of $M$ inputs $x_i$ drawn iid from $\Prob_\X$. 
Then, $\partial\X\subseteq X^M+B(0,\delta)$ with probability at least $1-\beta_{M,\delta}^{\partial\X}$.
\end{lemma}
Theorem \ref{thm:error_bound:smooth_boundary} (and Lemma \ref{lem:error_bound:covering_lipschitz}), combined with Lemma \ref{lem:covering_probability}, yields high-probability error bounds for the reconstruction of $\hull(\Y)$ using the convex hull of the images $f(x_i)$ of the sampled inputs $x_i$.
\begin{corollary}[Finite-sample error bounds] \label{cor:conservative_finite_sample} 
Let $\X\subset\R^m$ be a non-empty compact set, 
$f:\R^m\to\R^n$, 
$\Y=f(\X)$, 
$\delta>0$, 
$\Prob_\X$ be a probability measure over $\partial\X$ satisfying Assumption \ref{assum:sampling_density:boundary} with $\Lambda_\delta^{\partial\X}$, 
\rev{$M\in\bN$,} $\beta_{M,\delta}^{\partial\X}=\eqref{eq:delta:sampling_density:boundary}$, 
$\{x_i\}_{i=1}^M$ be $M$ inputs sampled iid from $\Prob_\X$, 
$y_i=f(x_i)$ for $i=1,\dots,M$, and  
 $\hat\Y^M=\hull\left(\{y_i\}_{i{=}1}^M\right)$. 
 Then, the following hold with probability at least $1-\beta_{M,\delta}^{\partial\X}$:
\begin{itemize}[leftmargin=4mm]\setlength\itemsep{0.5mm}
\item \textit{First-order error bound:} If $f$ is $\bar{L}$-Lipschitz and  
$\partial\Y\subseteq f(\partial\X)$ (e.g., if $f$ is a submersion by Lemmas \ref{lem:submersion_open_map}-\ref{lem:open_map_boundary}), then $\dH(\hull(\Y),\hat\Y^M)\leq \bar{L}\delta$.
\item \textit{Second-order error bound:} Let $r>0$. If $\X$ is path-connected and $r$-smooth, $f$ is a $C^1$ submersion, and $(f,\dd f)$ are $(\bar{L},\bar{H})$-Lipschitz,
then 
\begin{equation}\label{eq:cor:conservative_finite_sample:error}
\dH(\hull(\Y),\hat\Y^M)\leq \frac{1}{2}\left(\frac{\bar{L}}{r}+\bar{H}\right)\delta^2. 
\end{equation}
\end{itemize}
\end{corollary}

\rev{Corollary \ref{cor:conservative_finite_sample} gives an error bound with associated confidence probability $1-\beta^{\partial\X}_{M,\delta}$ as a function of $\delta^2$ and of the sample size $M$. The error bound in \eqref{eq:cor:conservative_finite_sample:error} gets tighter (with decreasing probability) as $\delta$ decreases, and increasing the sample size  $M$ increases the probability that \eqref{eq:cor:conservative_finite_sample:error}  holds.}

\revFirst{%The discussion in Section \ref{sec:thm:error_bound:smooth_boundary:discussion} also applies to Corollary \ref{cor:conservative_finite_sample}. 
The next asymptotic convergence result (Corollary \ref{cor:conservative_asymptotic}) follows from Corollary \ref{cor:conservative_finite_sample}. Corollary \ref{cor:conservative_asymptotic} resembles \cite[Corollary 2]{Dumbgen1996}, generalizing it to the non-convex problem setting.  Given strictly positive scalar functions $g,h$, we say that $g(M)=O(h(M))$ if $\lim\sup_{M\to\infty}g(M)/h(M)<\infty$.}
\begin{corollary}[\revFirst{Asymptotic convergence rates}] \label{cor:conservative_asymptotic} 
\revFirst{Let $\X\subset\R^m$ be a non-empty path-connected compact set that is $r$-smooth, 
$f:\R^m\to\R^n$ be a $C^1$ submersion where $(f,\dd f)$ are $(\bar{L},\bar{H})$-Lipschitz, 
$\Y=f(\X)$, 
% $\Prob_\X$ be the uniform probability measure over $\partial\X$, 
\rev{$\Prob_\X$ be an absolutely continuous\footnote{\rev{with respect to the Lebesgue measure over $\partial\X$. That is, the inputs $x_i$ are drawn roughly uniformly over $\partial\X$.}} probability measure over $\partial\X$ with a density $p$ %(with respect to the uniform probability measure over $\partial\X$) 
bounded below by a strictly positive constant (i.e., $p(x)\geq p_{\textrm{min}}>0$ for all $x\in\partial\X$),} 
$\{x_i\}_{i=1}^M$ be $M$ inputs sampled iid from $\Prob_\X$, 
$y_i=f(x_i)$ for $i=1,\dots,M$, and  
 $\hat\Y^M=\hull\left(\{y_i\}_{i{=}1}^M\right)$. 
 Then, almost surely,}
 $$
 \revFirst{\dH(\hull(\Y),\hat\Y^M) =  O\big((\log(M)/M)^{2/(m-1)}\big).}
 $$
\end{corollary}

\begin{figure}[t]
\includegraphics[height=3.85cm]{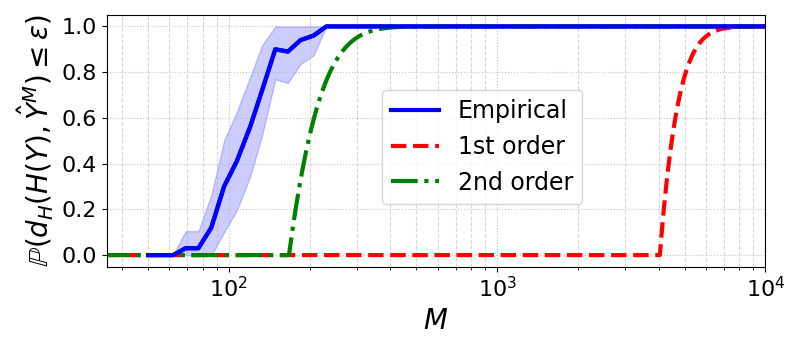}
\includegraphics[height=3.85cm]{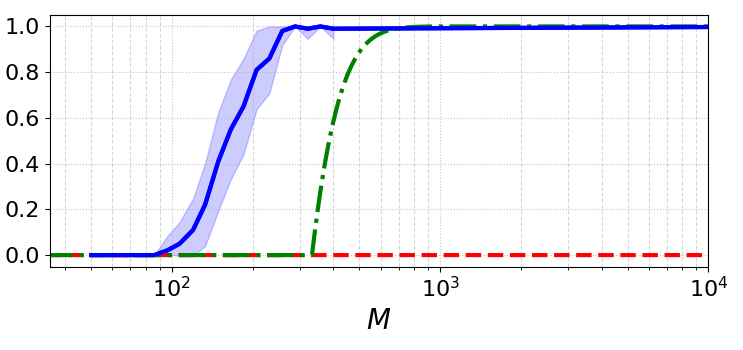}
\caption{Bounds from Corollary \ref{cor:conservative_finite_sample} for $L=1$ (left) and $L=3$ (right) for different sample sizes.}
\label{fig:sensitivity}
\end{figure}
 
We numerically evaluate the tightness of the bounds from Corollary \ref{cor:conservative_finite_sample}. As in \cite{LewJansonEtAl2022}, let $\X=B(0,1)\subset\R^2$ and $f:x\mapsto(Lx^1,x^2)$ for $L>0$, so that $\Y=f(\X)$ % 
is an ellipsoid. 
We sample $M$ inputs $x_i$ from a uniform distribution $\Prob_\X$ over $\partial\X$, which satisfies Assumption \ref{assum:sampling_density:boundary}. 
$\X$ is $1$-smooth, $f$ is a diffeomorphism (and thus a submersion), and $(f,\dd f)$ are $(\max(1,L),0)$-Lipschitz, so the assumptions of Corollary \ref{cor:conservative_finite_sample} hold. 
For a desired accuracy $\epsilon=10^{-2}$, using the first- and second-order bounds from Corollary \ref{cor:conservative_finite_sample}, we determine $\beta_{M,\delta}^{\partial\X}=\eqref{eq:delta:sampling_density:boundary}$ to achieve a Hausdorff distance error $\dH(\hull(\Y),\hat\Y^M)\leq\epsilon$ with probability at least $1-\beta_{M,\delta}^{\partial\X}$ for different sample sizes $M$.

We compare the bounds $1-\beta_{M,\delta}^{\partial\X}$ given by Corollary \ref{cor:conservative_finite_sample} with the empirical average number of trials that achieve $\epsilon$-accuracy (we use $100$ independent trials for each value of $M$). 
Results for different sample sizes $M$ are shown in Figure \ref{fig:sensitivity}. We observe that the second-order error bounds are quite sharp in the case $L=1$ and are an order of magnitude tighter than the first-order error bounds. Bounds become more conservative for larger values of $L$.

\subsection{Reachability analysis of uncertain dynamical systems}\label{sec:applications:reachability}
Next, we consider the problem of estimating the convex hull of all reachable states
of a dynamical system at a given time in the future. Such reachable sets play an important role in many applications ranging from robust predictive control \cite{Schurmann2018,Sieber2022} to 
neural network verification \cite{Everett21_journal}. 
Empirically, sampling-based approaches can provide accurate reconstructions of convex hulls of reachable sets from relatively few inputs \cite{LewPavone2020,LewJansonEtAl2022}. However, previous error bounds \cite{LewJansonEtAl2022} rely on naive Lipschitz-covering arguments and do not match empirical results, see Figure \ref{fig:sensitivity}. 

Let $\X_0\subset\R^n$, $\Theta\subset\R^p$, and $\U\subset\R^m$ be non-empty compact sets of initial conditions, parameters, and admissible control inputs. 
Let $f:\R^n\times\R^p\times\R^m\times\R\to\R^n$ be a continuous map that is Lipschitz in its first argument, i.e., for some $L\geq 0$, 
$\|f(x_1,\theta,u,t)-f(x_2,\theta,u,t)\|\leq L\|x_1-x_2\|$ for all $x_1,x_2\in\R^n, \theta\in\Theta, u\in\U, t\in\R$. 
Given $T>0$, the ordinary differential equation (ODE)
\begin{equation}
\label{eq:ODE}
\dot{x}(t)=f(x(t),\theta,u(t),t),\quad t\in[0,T],
\quad
x(0)=x^0
\end{equation}
has a unique solution $x_u^{x^0,\theta}\in C([0,T],\R^n)$ for any $u\in L^2([0,T],\U)$. For any $t\in[0,T]$, the map $x^0\mapsto x_u^{x^0,\theta}(t)$ is a diffeomorphism, so $(x^0,\theta)\mapsto x_u^{x^0,\theta}(t)$ is a submersion. Define the reachable set 
\begin{align*}
\Y_u(t)&=\left\{x_u^{x^0,\theta}(t)=x^0+\int_0^tf\big(x_u^{x^0,\theta}(s),\theta,u(s),s\big)\dd s: (x^0,\theta)\in\X\right\}.
\end{align*}
where $\X\subset\R^{n+p}$ is any approximation with smooth boundary of $\X_0\times\Theta$.  
Theorem \ref{thm:error_bound:smooth_boundary} implies that $\hull(\Y_u(t))$ can be accurately estimated using inputs in $\partial\X$.
\begin{corollary} \label{cor:reachability_analysis}
Let $r,\delta>0$, $Z_\delta=\{(x_i^0,\theta_i)\}_{i=1}^M\subset\partial\X$ be a $\delta$-cover of $\partial\X$, and assume that $\X$ is non-empty, compact, path-connected, and $r$-smooth. Let $u\in L^2([0,T],\U)$, $t\in\R$, $\hat{\Y}_u^M(t)=\hull\big(\{x_u^{x_i^0,\theta_i}(t)\}_{i=1}^M\big)$, and $(\bar{L}_u,\bar{H}_u)$ the Lipschitz constants of $(x^0,\theta)\mapsto (x_u^{x^0,\theta}(t),\dd x_u^{x^0,\theta}(t))$. Then, 
$\dH\big(\hull(\Y_u(t)),\hat{\Y}_u^M(t)\big)\leq \frac{1}{2}\big(\frac{\bar{L}_u}{r}+\bar{H}_u\big)\delta^2$. 
\end{corollary} 
The smoothness constants of the map $(x^0,\theta)\mapsto x_u^{x^0,\theta}(t)$ depend on properties of $f$ in \eqref{eq:ODE}. For example, tight bounds can be derived if the system in \eqref{eq:ODE} is contracting \cite{Wensing2020}. Such analysis is problem-specific and left for future work.

\subsection{Robust optimization}\label{sec:applications:optimization}
Next, we apply our analysis to study the feasibility of approximations to non-convex robust programs. 
Let $\X\subset\R^m$ be a compact set, 
$\C\subseteq\R^n$ be a closed convex set, 
$\ell:\R^p\to\R$ and 
$f:\R^m\times\R^p\to\R^n$ be two continuous functions, and define the robust optimization problem
\begin{align*}
\mathbf{P}:\ 
\inf_{u\in\R^p} \ell(u)
\ \ \text{s.t.} \ \ f(x,u)\in\C \ \text{for all }x\in\X.
\end{align*}
If $\X$ is infinite (e.g., if $\X=B(0,r)$ is a ball of parameters), then \textbf{P} has an infinite number of constraints that can be approximated as follows. 
Given $M$ sampled inputs $x_i\in\X$ and a padding $\epsilon>0$, define the relaxation
\begin{align*}
\hat{\mathbf{P}}_\epsilon^M:\ 
\inf_{u\in\R^p} \ell(u)
\ \ \text{s.t.} \ \ f(x_i,u)+B(0,\epsilon)\subseteq\C \ \text{for all }i=1,\dots,M.
\end{align*}
For instance, if $\C$ is an intersection of hyperplanes $\C_j=\{y\in\R^n:n_j^\top(y-c_j)\leq 0, c_j\in\R^n, \|n_j\|=1\}$, then the constraints in $\hat{\mathbf{P}}_\epsilon^M$ are equivalent to $n_j^\top(f(x_i,u)-c_j)+\epsilon\leq 0$ for all $i=1,\dots,M$ and $j$. 
In this case, $\hat{\mathbf{P}}_\epsilon^M$ is a tractable finite-dimensional relaxation of $\mathbf{P}$. 

Thanks to Theorem \ref{thm:error_bound:smooth_boundary},  solving $\hat{\mathbf{P}}_\epsilon^M$ yields feasible solutions of $\mathbf{P}$ given sufficiently many sampled inputs $x_i$ on the boundary $\partial\X$. If $f$ is a submersion, a small sample size $M$ suffices.
\begin{corollary}\label{cor:robust_programming}
Let $r,\delta>0$, $\X\subset\R^m$ be a non-empty path-connected compact $r$-smooth set, 
$\{x_i\}_{i=1}^M\subset\partial\X$ be a $\delta$-cover of $\partial\X$,
 $f$ be such that $f_u=f(\cdot,u)$ is a $C^1$ submersion and $(f_u,\dd f_u)$ are $(\bar{L},\bar{H})$-Lipschitz for all $u\in\R^p$,  
and $\epsilon\geq\frac{1}{2}\big(\frac{\bar{L}_u}{r}+\bar{H}_u\big)\delta^2$.  
Then, any solution of $\hat{\mathbf{P}}_\epsilon^M$ is feasible for $\mathbf{P}$.
\end{corollary}
Corollary \ref{cor:robust_programming} justifies solving the relaxed problem $\hat{\mathbf{P}}_\epsilon^M$ to obtain feasible solutions of $\mathbf{P}$. The analysis of the suboptimality gap is left for future work. 
We provide an application of this result next.

\subsection{Numerical example: planning under bounded uncertainty}\label{sec:applications:robust_planning}
We consider the following optimal control problem (OCP) under bounded uncertainty:
\begin{align*}
\textbf{OCP}:\quad 
\inf_{u\in\UU} \qquad &\int_0^T\|u(t)\|^2\dd t
    &&\text{(min. fuel consumption)}
\\
\text{s.t.} \qquad &\dot{p}(t)=v(t),
\  
\dot{v}(t)=\frac{1}{m}(u(t)+F),
\ \ t\in[0,T], 
    &&\text{(dynamics)}
\\
&Hp(t)\leq h, \hspace{39mm} t\in[0,T], 
\ \ 
    &&\text{(obstacle avoidance)}
\\[1mm]
&Gp(T)\leq g, \quad (p(0),v(0))=0,
    &&\text{(initial \& final conditions)}
\\[1mm]
&(m,F)\in\X,
    &&\text{(uncertain parameters)}
\end{align*}
where $(p(t),v(t))\in\R^4$ denote the position and velocity of the system, 
$u(t)\in\R^2$ is the control input, 
$(m,F)\in\X\subset\R^3$ correspond to the uncertain mass of the system and constant disturbance, 
$H,h,G,g$ define the obstacle-free statespace and goal region, 
$T$ is the planning horizon, and 
we optimize over control trajectories $u\in\UU \subset L^2([0,T],\R^2)$ that are piecewise-constant on a partition of $[0,T]$, as is common in applications such as model predictive control \cite{Schurmann2018,Sieber2022}. The dynamics may correspond to a spacecraft system \cite{LewJansonEtAl2022} carrying an uncertain payload subject to constant disturbances. 

Although the dynamics are linear in the control input, the dynamics are nonlinear in the uncertain parameters $(m,F)$. The resulting uncertainty over the state trajectory is correlated over time, which makes solving \textbf{OCP} challenging. This contrasts with problems with additive independent disturbances for which a wide range of numerical resolution schemes exist. We refer to \cite{LewJansonEtAl2022} for a discussion of existing methods for reachability analysis of such systems.

We study the problem in  detail in Section \ref{sec:applications:robust_planning:appendix}. 
We consider the uncertainty set $(m,F)\in[30,34]\times B(0,5\cdot 10^{-3})$ and show how to outer-bound these inputs with an $r$-smooth compact set $\X$. 
Then, we show that the map $(m,F)\mapsto p_u^{m,F}(t)$ is a submersion and study its smoothness. 
The assumptions of Corollaries \ref{cor:reachability_analysis} and \ref{cor:robust_programming} hold, so we evaluate the bound on the Hausdorff distance  $\dH\big(\hull(\Y_u(t)),\hat{\Y}_u^M(t)\big)\leq \epsilon=0.025$ for $M=100$ inputs $(m_i,F_i)$, which is sufficiently accurate. 
We discretize \textbf{OCP} in time and express the finite-dimensional relaxation where constraints are only evaluated at the $M$ inputs $(m_i,F_i)$ as described in Section \ref{sec:applications:optimization}. 
We solve the resulting convex program in approximately $200\,\textrm{ms}$ (measured on a laptop with a 1.10GHz Intel Core i7-10710U CPU) using \textrm{OSQP} \cite{Stellato2020} and present results in Figure \ref{fig:traj}. As guaranteed by Corollary \ref{cor:robust_programming}, the obtained trajectory is collision-free and reaches the goal region for all uncertain parameters $(m,F)$. 
We note that $M=3300$ inputs would be necessary to provably achieve the same level of precision with a naive bound that only leverages the Lipschitzness of $p_u^{m,F}(t)$ (see Lemma \ref{lem:error_bound:covering_lipschitz}), and solving the resulting approximation of \textbf{OCP} would take over $25\,\textrm{s}$. % 

\begin{figure}[t]
\centering
\includegraphics[width=0.48\linewidth, trim={4mm 2mm 2mm 2mm}, clip]{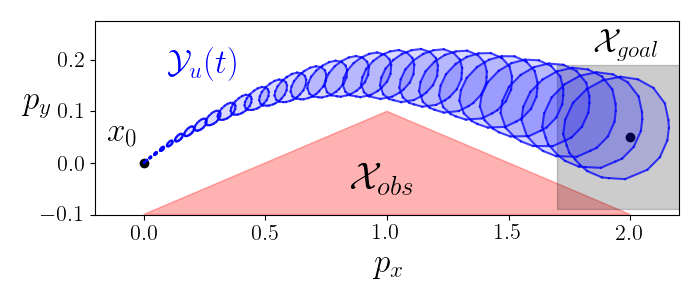}
\vspace{-2mm}
\includegraphics[width=0.48\linewidth, trim={6mm 2mm 2mm 2mm}, clip]{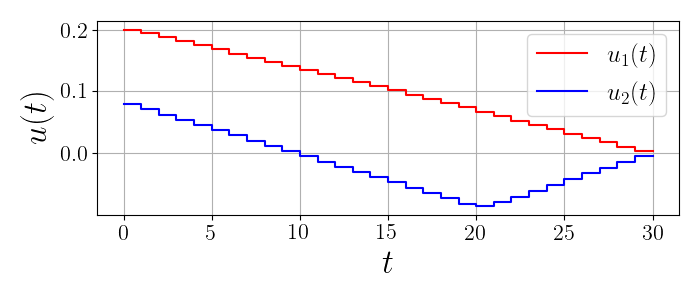}
\caption{Solution of the finite-dimensional approximation of \textbf{OCP}.}
\label{fig:traj}
\end{figure}

\section{Conclusion}\label{sec:conclusion}
We derived new error bounds for the estimation of the convex hull of the image $f(\X)$ of a set $\X$ with smooth boundary. Our results show that accurate reconstructions are possible using a few sampled inputs $x_i$ on the boundary of $\X$. We provided numerical experiments demonstrating the tightness of our bounds in practical applications. 

Of immediate interest for future research is deriving error bounds for non-convex approximations of the output set $f(\X)$ (e.g., for tangential Delaunay complexes \cite{Boissonnat2013,Aamari2018}) from assumptions on $f$ and $\X$. 
Extending the results to the presence of noise corrupting the sample  \cite{Aamari2018,Aamari2022} % 
would allow reconstructing $f(\X)$ from sampled inputs that are not exactly on the boundary of $\X$. Potentially, this would also allow accurate reconstructions using approximate models of $f$. \rev{Finally, exploring whether Theorem \ref{thm:error_bound:smooth_boundary} holds under a weaker rolling ball condition (see Remark \ref{remark:rolling_ball_but_not_rsmooth}) is of interest.}
% Applications of our results include the design of more efficient control algorithms that explicitly account for uncertainty. % 
% 

% \section*{Acknowledgements}
% The NASA University Leadership Initiative (grant \#80NSSC20M0163) provided funds to assist the authors with their research, but this article solely reflects the opinions and conclusions of its authors and not any NASA entity.  L.J. was supported by the National Science Foundation via grant CBET-2112085.
% 
% 
% 

\let\oldbibliography\thebibliography
\renewcommand{\thebibliography}[1]{\oldbibliography{#1}
\setlength{\itemsep}{0pt}} % 
\bibliographystyle{amsalpha}
\bibliography{ASL_papers,main}

\appendix

\section{Proofs for Section \ref{sec:smooth_sets:blaschke}}\label{sec:smooth_sets:proofs}

\begin{proof}[Proof of Lemma \ref{lem:rsmooth_implies_lambdasmooth}]
Let $x\in\partial\X$ and $0<\lambda\leq R$ (the result is trivial if $\lambda=0$). A ball of radius $R$ rolls freely in $\X$; let $a\in\X$ be such that $x\in B(a,R)\subseteq\X$. Denoting by $n(x)=\frac{x-a}{\|x-a\|}$ the outward-pointing unit-norm normal of $\partial B(a,R)$ at $x$, we define $a_\lambda=x-\lambda n(x)$. % 
Then, $x\in B(a_\lambda,\lambda)$ since $\|x-a_\lambda\|=\lambda$, and $n(x)$ is also the outward-pointing unit-norm normal of $\partial B(a_\lambda,\lambda)$. 
Since $\frac{1}{\lambda}\geq\frac{1}{R}$, by \cite{Rauch1974}, $B(a_\lambda,\lambda)\subseteq B(a,R)\subseteq\X$. We conclude that a ball of radius \revFirst{$\lambda$} rolls freely in $\X$. %The proof that a ball of radius $\lambda$ rolls freely in  $\overline{\X^{\comp}}$ is identical. 
\end{proof}

\begin{proof}[Proof of Lemma \ref{lem:convex_rolling_outside}]
    The result follows from the supporting hyperplane theorem, see \cite[Theorem 1.3.2]{Schneider2014}.
\end{proof}

\section{Proofs for Section \ref{sec:error_bounds:bound_on_deltaTphull} (bound on $d_{T_y\partial\hull(\Y)}$)}\label{sec:error_bounds:bound_on_deltaTphull:proofs}
\subsection{Modifications for Lemma \ref{lem:bound_delta_Tphull} if $f$ is only $C^1$}\label{sec:lem:bound_delta_Tphull:proof_C1}
In this section, 
we sketch the minor modifications to the proof of Lemma \ref{lem:bound_delta_Tphull} in Section \ref{sec:error_bounds} to handle the case where $f$ is only $C^1$ (if $f$ is only $C^1$, using $\dd^2 f$ is not rigorous). The result can be justified using  Taylor's Theorem \cite{Folland1990}, which states that any $C^1$ scalar function $g:\R\to\R$ satisfies
$$
g(b)=g(a)+g'(a)(b-a)+(b-a)\int_0^1\left(g'(a+u(b-a))-g'(a)\right)\dd u.
$$
For simplicity, consider the scalar case where $f:\R\to\R$ is $C^1$ and $\gamma:\R\to\R$ is $C^2$. Then,
\begin{align*}
f(\gamma(t))-f(\gamma(0))
-t(f\circ\gamma)'(0)
&=
t\int_0^1\left((f\circ\gamma)'(ut)-(f\circ\gamma)'(0)\right)\dd u
\\
&= 
t\int_0^1\left(
\dd f_{\gamma(ut)}(\gamma'(ut))-
\dd f_{\gamma(0)}(\gamma'(0))
\right)\dd u
\\
&= 
t\int_0^1\left(
\dd f_{\gamma(ut)}(\gamma'(0))-
\dd f_{\gamma(0)}(\gamma'(0))
\right)\dd u.
\end{align*}
using the fact that $\gamma''(s)=0$ so that $\gamma'(ut)=\gamma'(0)$.

Assuming that $\dd f$ is $\bar{H}$-Lipschitz, $\|\dd f_{\gamma(ut)}(\gamma'(0))-
\dd f_{\gamma(0)}(\gamma'(0))\|\leq \bar{H}\|\gamma(ut)-\gamma(0)\|\|\gamma'(0)\|$ and
\begin{align*}
\left\|f(\gamma(t))-f(\gamma(0))
-t(f\circ\gamma)'(0)\right\|
&=
\left\|t\int_0^1\left(
\dd f_{\gamma(ut)}(\gamma'(0))-
\dd f_{\gamma(0)}(\gamma'(0))
\right)\dd u\right\|
\\
&\leq
t
\int_0^1
\bar{H}\|\gamma(ut)-\gamma(0)\|\|\gamma'(0)\|\dd u
\\
&=
\bar{H}\frac{t^2}{2}
\end{align*}
using $\gamma(ut)-\gamma(0)=t\gamma'(0)$ and $\|\gamma'(0)\|=1$.  Thus, with minor modifications, the proof of Lemma \ref{lem:bound_delta_Tphull} only requires assuming that $f\in C^1$ as claimed. 

\subsection{Proof of Lemma \ref{lem:curve_intersects_ball} (curve intersecting a ball)}\label{sec:lem:curve_intersects_ball:proof}

We first define a suitable chart to prove Lemma \ref{lem:curve_intersects_ball}. 

Let $B\subset\R^n$ be an $n$-dimensional closed ball of radius $r>0$ in $\R^n$, whose boundary $\partial B\subset\R^n$ is an $(n-1)$-dimensional submanifold of $\R^n$. We denote the normal bundle of $\partial B$ by $N(\partial B)$ % 
and the outward-pointing unit-norm normal of $\partial B$ by $n^{\partial B}$, which defines a smooth frame for $N(\partial B)$. The restriction of the exponential map in $\R^n$ to the normal bundle of $\partial B$ is defined as
$$
E:N(\partial B)\to\R^n:(x,v)\mapsto x+v.
$$
Let $U\subset\R^n$ be a uniform tubular neighborhood of $\partial B$ in $\R^n$, see \cite[Theorem 5.25]{Lee2018}. Then, there is an open set
$$
V_{\delta} = \Big\{ (x,s n^{\partial B}(x)) \in N(\partial B) :  s \in (-\delta,\delta) \Big\}% 
$$
for some $\delta>0$ such that the map
\begin{align*}
E : V_{\delta} \to U, \ \big(x,sn^{\partial B}(x)\big) \mapsto x + s n^{\partial B}(x) 
\end{align*}
is a diffeomorphism. Next, we define the smooth map
\begin{align*}
    \psi : \partial B \times (-\delta,\delta) \to V_{\delta}, \ (x,s) \mapsto \big(x,sn^{\partial B}(x)\big)
\end{align*}
which is a diffeomorphism since its differential is an isomorphism. Therefore, the smooth map
$$
(E\circ\psi) : \partial B \times (-\delta,\delta) \to U
$$
is a diffeomorphism. Finally, we define the following chart of $\R^n$
$$
\varphi=(E\circ\psi)^{-1} : U \to \partial B \times (-\delta,\delta).
$$
which satisfies 
$\varphi(y)=(x,s)\in\partial B\times(-\delta,\delta)$ for any $y=x+s n^{\partial B}(x)\in U$. 
Since $n^{\partial B}$ is outward-pointing, the last component  of $\varphi$ satisfies
\begin{itemize}
\item $\varphi^n(y)>0\iff y\in B^\comp$,
\item $\varphi^n(y)=0\iff y\in \partial B$,
\item $\varphi^n(y)<0\iff y\in \Int(B)$.
\end{itemize}

\begin{proof}[Proof of Lemma \ref{lem:curve_intersects_ball}]
In the following, we use the coordinates defined previously.

Define $f(t)=\varphi^n(\gamma(t))$. Since $\gamma(t)\in\Int(B)$ if and only if $f(t)=\varphi^n(\gamma(t))<0$, it suffices to prove that $f(t)<0$ for some $t\in (-\epsilon,\epsilon)$. Since $f(t)$ is smooth and $f(0)=\varphi^n(\gamma(0))=\varphi^n(p)=0$, it suffices to prove that $f'(0)<0$ or $f'(0)>0$. 

$f'(0)=\frac{\dd}{\dd t}\left(\varphi^n(\gamma(t))\right)\big|_{t=0}=\left(\dd\varphi^n_{\gamma(t)}
(\gamma'(t))\right)\Big|_{t=0}=\dd\varphi^n_p(v)$. 
Since $\dd\varphi^n_p(v)=0$ if and only if $v\in T_p\partial B$, but $v\notin T_p\partial B$ by assumption, we obtain that $f'(0)\neq 0$. This concludes the proof.
\end{proof}

\subsection{Proof of Lemma \ref{lem:tangent_mapped_to_tangent_hull} ($\dd f_x(T_x\partial\X)=T_y\partial\hull(\Y)$)}\label{sec:tangent_to_tangent}
\subsubsection{Problem definition and setup}

\begin{itemize}[leftmargin=5mm]\setlength\itemsep{0.5mm}
\item 
Let $r>0$ and $\X\subset\R^m$ be a non-empty path-connected $r$-smooth compact set. 
	\begin{itemize}[leftmargin=5mm]
		\vspace{-2mm}\setlength\itemsep{0.5mm}
	\item By Theorem \ref{thm:walther1999}, $\partial\X\subset\R^m$ is an $(m-1)$-dimensional submanifold with  unique normal $n^{\partial\X}(x)$.
\item By Theorem \ref{thm:walther1999}, an $m$-dimensional ball of radius $r>0$ rolls freely in $\X$ and $\overline{\X^{\comp}}$. 
	\vspace{-1mm}
	\end{itemize}
\item Let $f:\R^m\to\R^n$ be a smooth submersion.
\item Let $\Y=f(\X)$.
\item Let $y\in\partial\Y\cap\partial\hull(\Y)$. 
\item Let $x\in\partial\X$ be such that $y=f(x)$ (this input $x$ exists since $\partial\Y\subseteq f(\partial\X)$ by Lemmas \ref{lem:submersion_open_map} and \ref{lem:open_map_boundary} since $f$ is a submersion).  
	\begin{itemize}[leftmargin=5mm]
		\vspace{-2mm}\setlength\itemsep{0.5mm}
	\item  By the rank theorem, there exist two charts $((U,\varphi), (V,\psi))$ centered at $x$ such that $\psi\circ f\circ\varphi^{-1}$ is a coordinate projection:
$$
\psi\circ f\circ\varphi^{-1}: 
\varphi(U\cap f^{-1}(V))\to\psi(V),
\ 
(\hat{x}_1,\dots,\hat{x}_n,\hat{x}_{n+1},\dots\hat{x}_m)\mapsto (\hat{x}_1,\dots,\hat{x}_n).
$$
In other words, $(\psi\circ f\circ\varphi^{-1})|_{\varphi(U\cap f^{-1}(V))}(\cdot)=\pi(\cdot)$ is the projection from $\R^m$ to $\R^n$.
	\vspace{-2mm}
	\end{itemize}
\item Let $P=\{\hat{x}\in\R^m:\hat{x}_{n+1}=\dots=\hat{x}_m=0\}$ and $Q=\{\hat{x}\in\R^m:\hat{x}_1=\dots=\hat{x}_n=0\}$. 
\item Let $S=\varphi^{-1}\left(P\cap\varphi(U)\right)$.
\item Let $B\subseteq\X$ be a tangent ball at $x$ inside $\X$ of radius $\min(r,\reach(S))>0$ with $T_x\partial B=T_x\partial\X$. This ball exists by Theorem \ref{thm:walther1999}, and $\reach(S)>0$ since $S$ is a submanifold.
\item Let $Z=S\cap\partial B$. 
\end{itemize}
\vspace{-2mm}
We claim that $\dd f_x(T_x\partial\X)=T_y\partial\hull(\Y)$.

\vspace{3mm}

\begin{figure}[!htb]
    \centering\includegraphics[width=0.8\linewidth]{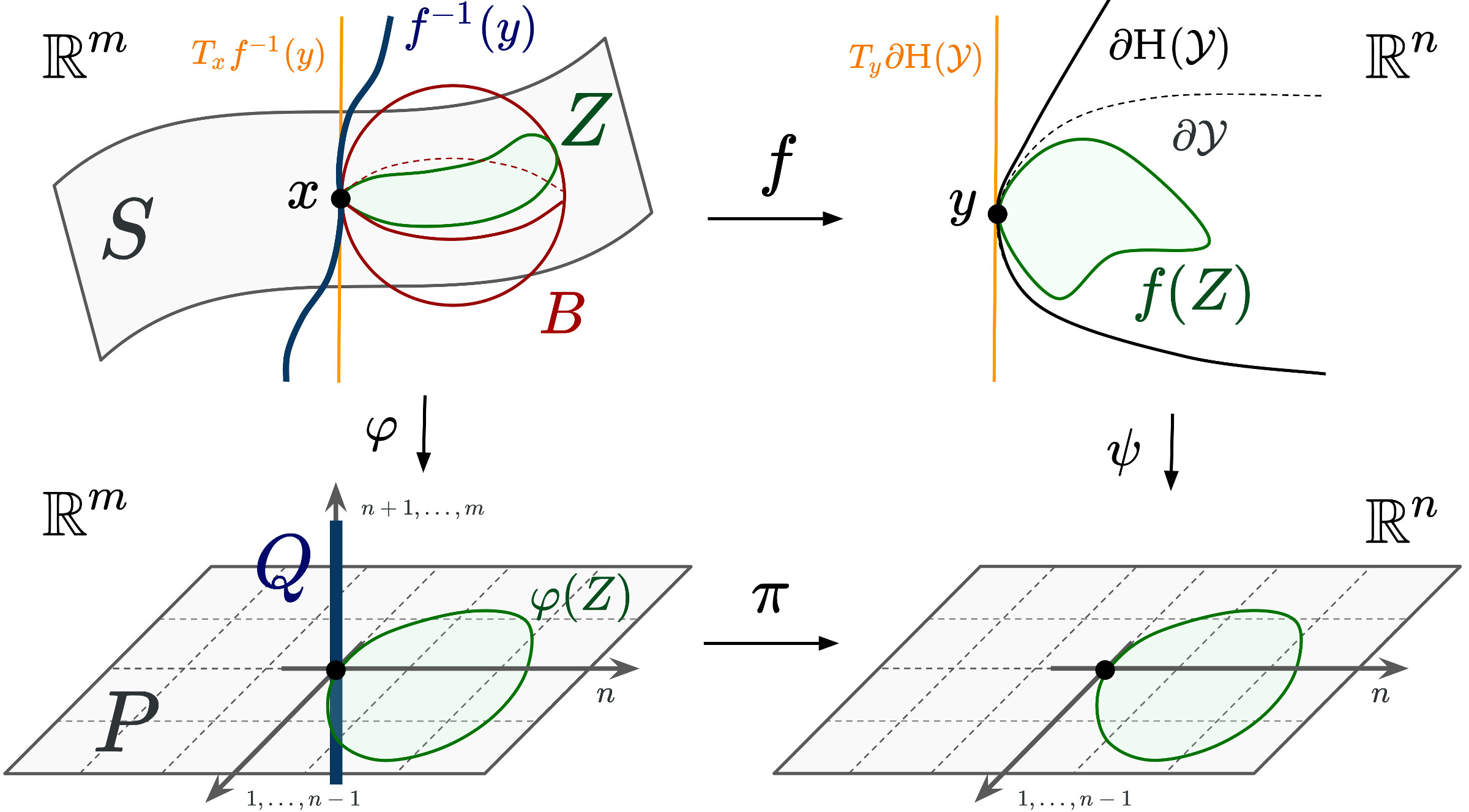}\caption{Definitions for the proof of Lemma \ref{lem:tangent_mapped_to_tangent_hull}.}
\label{fig:ift:proof_of_lem:tangent_mapped_to_tangent_hull}
    \end{figure}

\subsubsection{Properties}
First, $f^{-1}(y)\subset\R^m$ is an $(m-n)$-dimensional submanifold  and $T_xf^{-1}(y)=\Ker(\dd f_x)$ since $f$ is a submersion (see Theorem 5.12 and Proposition 5.38 in \cite{Lee2012}). The following properties are used in the proof of Lemma \ref{lem:tangent_mapped_to_tangent_hull}.

\begin{property}\label{property:2}
$\Ker(\dd f_x)=\dd(\varphi^{-1})_{\varphi(x)}\left(T_{\varphi(x)}Q\right)$. 
\end{property}

\begin{property}\label{property:3}
$T_x\R^m=T_xS+\Ker(\dd f_x)$. 
\end{property}

\begin{property}\label{property:4}
$T_xf^{-1}(y)\subseteq T_x\partial\X$. 
\end{property}

\begin{property}\label{property:5}
$Z$ is a submanifold of dimension $(n-1)$. 
\end{property}

\begin{property}\label{property:6}
$T_x\partial\X=T_xZ+ T_xf^{-1}(y)$.
\end{property}

\begin{property}\label{property:7}
$T_y\partial\hull(\Y)=T_yf(Z)$.
\end{property}

\subsubsection{Proofs of Properties \ref{property:2}-\ref{property:7}}

\begin{proof}[Proof of Property \ref{property:2}: \underline{$\Ker(\dd f_x)=\dd(\varphi^{-1})_{\varphi(x)}\left(T_{\varphi(x)}Q\right)$}] First, we show that $\dd(\varphi^{-1})_{\varphi(x)}(T_{\varphi(x)}Q)\subseteq\Ker(\dd f_x)$. 
Let $v\in\dd(\varphi^{-1})_{\varphi(x)}(T_{\varphi(x)}Q)$ and $w=\dd\varphi_x(v)\in T_{\varphi(x)}Q$. Then,  
\begin{align*}
\dd f_x(v)
&= \dd (\psi^{-1}\circ\pi\circ\varphi)_x(\dd(\varphi^{-1})_{\varphi(x)}(w))
=
\dd(\psi^{-1})_{\pi(\varphi(x))}(\dd\pi_{\varphi(x)}(w))
=
\dd(\psi^{-1})_{\pi(\varphi(x))}(0)=0
\end{align*}
since $\dd\pi_{\varphi(x)}(T_{\varphi(x)}Q)=0$ by definition of $Q$. % 
Thus, $v\in\Ker(\dd f_x)$. 

Second, $\dim(\dd(\varphi^{-1})_{\varphi(x)}(T_{\varphi(x)}Q))=\dim(T_{\varphi(x)}Q)=m-n$, since $\varphi^{-1}$ is a diffeomorphism.

Thus, $\dim(\dd(\varphi^{-1})_{\varphi(x)}(T_{\varphi(x)}Q))=\dim(\Ker(\dd f_x))$ and $\dd(\varphi^{-1})_{\varphi(x)}(T_{\varphi(x)}Q)\subseteq\Ker(\dd f_x)$. 

The conclusion follows. 
\end{proof}

\begin{proof}[Proof of Property \ref{property:3}: \underline{$T_x\R^m=T_xS+\Ker(\dd f_x)$}] 
Indeed, since $\varphi^{-1}$ is a diffeomorphism, 
\begin{align*}
T_x\R^m
&= \dd(\varphi^{-1})_{\varphi(x)}(T_{\varphi(x)}\R^m)
\\
&=
\dd(\varphi^{-1})_{\varphi(x)}(T_{\varphi(x)}P+ T_{\varphi(x)}Q)
\\
&=
T_xS+ \dd(\varphi^{-1})_{\varphi(x)}(T_{\varphi(x)}Q))
\\
&=
T_xS+ \Ker(\dd f_x). &&\text{(Property \ref{property:2})}
\end{align*}

\vspace{-8mm}
\end{proof}

We observe that up to Property \ref{property:3}, we did not use the fact that $y\in\partial\Y$.
\begin{proof}[Proof of Property \ref{property:4}: \underline{$T_xf^{-1}(y)\subseteq T_x\partial\X$}] 

By contradiction, assume that $T_xf^{-1}(y)\nsubseteq T_x\partial\X$. Then, there is some $v\in T_xf^{-1}(y)$ but $v\notin T_x\partial\X=T_x\partial B$.

Let $I=(-\epsilon,\epsilon)$ and $\gamma:I\to f^{-1}(y)$ be a curve with $\gamma(0)=x$ and $\gamma'(0)=v$. By definition, this curve satisfies $f(\gamma(t))=y\in\partial\Y$ for all $t\in I$. 
However, by Lemma \ref{lem:curve_intersects_ball}, there exists $s\in I$ such that $\gamma(s)\in\Int(B)$, since $\gamma'(0)=v\notin T_x\partial B$. 
Thus, $y=f(\gamma(s))\in f(\Int(B))\subseteq f(\Int(\X))\subseteq\Int(f(\X))=\Int(\Y)$, since $f$ is an open map (since it is a submersion, see Lemma \ref{lem:submersion_open_map}). Thus, $y\in\Int(\Y)$, which is a contradiction since $y\in\partial\Y$. 
\end{proof}

\begin{proof}[Proof of Property \ref{property:5}: \underline{$Z$ is a submanifold of dimension $(n-1)$}] We proceed in three steps.

\begin{itemize}[leftmargin=5mm]\setlength\itemsep{0.5mm}
\item \textit{Step 1: The open ball $\mathring{B}$ intersects $S$.} 
First, we show that $T_x\R^m=T_xS+ T_x\partial B$. Indeed,
\begin{align*}
T_x\R^m &= 
	T_xS+\Ker(\dd f_x)&&\text{(Property \ref{property:3})}
	\\
	&=
	T_xS+T_xf^{-1}(y)% 
	\\
	&\subseteq 
	T_xS+T_x\partial\X&&\text{(Property \ref{property:4})}
	\\
	&=
	T_xS+T_x\partial B.
\end{align*}
Thus, $T_x\R^m=T_xS+ T_x\partial B$. 
Since $\dim(T_x\partial B)=m-1$, this implies that there exists $v\in T_xS$ with $v\notin T_x\partial B$.

 Let $\gamma:(-\epsilon,\epsilon)\to S$ be any smooth curve with $\gamma(0)=x$ and $\gamma'(0)=v$. By Lemma \ref{lem:curve_intersects_ball}, $\gamma(t)\in\Int(B)$ for some $t\in(-\epsilon,\epsilon)$. Thus, $\gamma(t)\in S\cap\Int(B)$. We obtain that the open ball $\mathring{B}$ intersects $S$.

\item \textit{Step 2: $T_{\tilde{x}}\R^m=T_{\tilde{x}}S+T_{\tilde{x}}\partial B$ for all $\tilde{x}\in S\cap\partial B$.} 
By contradiction, assume that $T_{\tilde{x}}\R^m\neq T_{\tilde{x}}S+T_{\tilde{x}}\partial B$ for some $\tilde{x}\in S\cap\partial B$. 
Since $\dim(T_{\tilde{x}}\partial B)=m-1$, this implies that $T_{\tilde{x}}S\subseteq T_{\tilde{x}}\partial B$.  Thus, the ball $B$ is tangent to $S$ at $\tilde{x}$. 

Since $B$ has a radius smaller than $\reach(S)$ and is tangent to $S$ at $\tilde{x}$ by the above, the open ball $\mathring{B}$ does not intersect $S$ \cite[Corollary 2]{Boissonnat2019}. This contradicts \textit{Step 1}.

\item \textit{Step 3: Conclude by transversality.} By \textit{Step 2}, $\Span(T_{\tilde{x}}S,T_{\tilde{x}}\partial B)=T_{\tilde{x}}\R^m$ for all $\tilde{x}\in S\cap\partial B$. Thus, by transversality \cite[Theorem 6.30]{Lee2012}, $Z=S\cap\partial B$ is a submanifold of dimension $m-(\codim(S)+\codim(\partial B))=m-(m-n+1)=n-1$.
The conclusion follows. 
\end{itemize}\vspace{-7mm}
\end{proof}

\begin{proof}[Proof of Property \ref{property:6}: \underline{$T_x\partial\X=T_xZ+ T_xf^{-1}(y)$}]
\begin{align*}
\dim\left(\Span(
	T_xZ,\, 
	T_xf^{-1}(y)
)\right)
&=
\dim\left(\Span(
	\dd\varphi_x(T_xZ),\,  
	\dd\varphi_x(T_xf^{-1}(y))
)\right)
&&\text{($\varphi$ is a diffeomorphism)}
\\
&=
\dim\left(\Span(
	\dd\varphi_x(T_xZ),\,  
	T_{\varphi(x)}Q
)\right)
&&\text{(Property \ref{property:2})}% 
\\
&\hspace{-2cm}=
\dim(\dd\varphi_x(T_xZ)) 
+\dim(T_{\varphi(x)}Q)
&&\hspace{-3cm}\text{($\dd\varphi_x(T_xZ)\subset T_{\varphi(x)}P\text{ and }T_{\varphi(x)}P\perp T_{\varphi(x)}Q$)}
\\
&\hspace{-2cm}=(n-1)+(m-n)=m-1.
\end{align*}
Thus, $\dim\left(\Span(
	T_xZ,\, 
	T_xf^{-1}(y)
)\right)=\dim(T_x\partial\X)$. Since 
$T_xZ\subseteq T_x\partial\X$ (since $Z\subseteq\partial\X$) and $T_xf^{-1}(y)\subseteq T_x\partial\X$ (Property \ref{property:4}), we obtain that  $T_x\partial\X=T_xZ+ T_xf^{-1}(y)$.
\end{proof}

We recall that  $\partial\hull(\Y)$ is a submanifold of dimension $(n-1)$ by Corollary \ref{cor:sub_hull:boundary_manifold}. Also, $f(Z)\subset\R^n$ is a submanifold, since $Z\subset S$ is a submanifold (Property \ref{property:5}) and $f|_S$ is a diffeomorphism.

\begin{proof}[Proof of Property \ref{property:7}: \underline{$T_y\partial\hull(\Y)=T_yf(Z)$}]  By contradiction, assume that $T_y\partial\hull(\Y)\neq T_yf(Z)$. Then, there exists $v\in T_yf(Z)$ such that $v\notin T_y\partial\hull(\Y)$, since $\dim(\partial\hull(\Y))=n-1=\dim(T_yf(Z))$ (note that $f|_S$ is a diffeomorphism and $Z\subset S$, so $\dim(T_yf(Z))=\dim(T_xZ)=n-1$).

Let $\tilde{B}\subset\overline{\hull(\Y)^\comp}$ be a ball outside $\hull(\Y)$ that is tangent to $\hull(\Y)$ at $y$ (such that $T_y\partial\tilde{B}=T_y\partial\hull(\Y)$). Such a ball exists since $\hull(\Y)$ is convex.

Let $\gamma:(-\epsilon,\epsilon)\to f(Z)$ be a smooth curve with $\gamma(0)=y$ and $\gamma'(0)=v$. Since $v\notin T_y\partial\tilde{B}=T_y\partial\hull(\Y)$, by Lemma \ref{lem:curve_intersects_ball},  $\gamma(s)\in\Int(\tilde{B})\subset\hull(\Y)^\comp$ for some $s\in (-\epsilon,\epsilon)$. 

However, $f(Z)\subset f(B) \subseteq f(\X)=\Y\subseteq\hull(\Y)$, so $\gamma(t)\in\hull(\Y)$ for all $t\in(-\epsilon,\epsilon)$. This is a contradiction. 
\end{proof}

\subsubsection{Proof that $\dd f_x(T_x\partial\X)=T_y\partial\hull(\Y)$}
\begin{proof} We have
\begin{align*}
\dd f_x(T_x\partial\X)
&=
\dd f_x(T_xZ+\Ker(\dd f_x))
&&\text{(Property \ref{property:6} and $T_xf^{-1}(y)=\Ker(\dd f_x)$)}
\\
&=
\dd f_x(T_xZ)
\\
&=
T_{f(x)}f(Z)
&&\text{($f|_S$ is a diffeomorphism and $Z\subset S$)}
\\
&=
T_yf(Z)
\\
&=T_y\partial\hull(\Y).
&&\text{(Property \ref{property:7})}
\end{align*}
Thus,  $\dd f_x(T_x\partial\X)=T_y\partial\hull(\Y)$.
\end{proof}

\section{Proofs and details for Section \ref{sec:applications} (applications)}\label{sec:applications:appendix}

\subsection{Geometric inference}

\subsubsection{The estimator $\hat{\Y}^M$ is a random compact set}\label{sec:random_compact_set_estimator}

Let $(\Omega,\G,\Prob)$ be a probability space such that the $x_i$ are $\G$-measurable independent random variables whose laws $\Prob_\X$ satisfy  $\Prob_\X(A)=\Prob(x_i\in A)$ for any $A\in\B(\R^m)$\footnote{For a canonical construction, 
let $\Omega=\R^m\times\mydots\times\R^m$ ($M$ times), 
$\G=\B(\R^m)\otimes\mydots\otimes\B(\R^m)$,  
$\Prob=\Prob_\X\otimes\mydots\otimes\Prob_\X$ the product measure, and % 
$x=(x_1,\mydots,x_M): \Omega\rightarrow\Omega:\omega\mapsto\omega$. Then, the $x_i$ are independent and have the law $\Prob_\X$.}. 
Then, the $y_i=f(x_i)$ are $\R^n$-valued random variables, whose laws $\Prob_{\Y}$ satisfy % 
$\Prob_{\Y}(B)=\Prob(y_i\in B)=\Prob_\X(f^{-1}(B))$ for any $B\in\B(\R^n)$.

The Hausdorff distance $\dH$ induces the \textit{myopic topology} on $\K$ \cite{Molchanov_BookTheoryOfRandomSets2017} with its associated generated Borel $\sigma$-algebra $\B(\K)$. As such, 
$(\K,\B(\K))$ is a measurable space, and the map 
$\hat\Y^M:(\Omega,\G)\to(\K,\B(\K))$  is a random compact set, i.e., a random variable taking values in the space of compact sets $\K$. The measurability of $\hat\Y^M$ follows from the measurability of the convex hull of a random closed set \cite[Theorem 1.3.25]{Molchanov_BookTheoryOfRandomSets2017}, and allows studying the probability of achieving a desired reconstruction accuracy with $\hat\Y^M$ (in particular, $\omega\mapsto \dH(\hull(\Y),\hat\Y^M(\omega))$ is measurable, as is taking countable intersections and unions of random compact sets \cite[Theorem 1.3.25]{Molchanov_BookTheoryOfRandomSets2017} as in the proof of Lemma \ref{lem:covering_probability}). 
Intuitively, different sampled inputs $x_i(\omega)$  induce different sampled output $y_i(\omega)$, resulting in  different approximated compact sets $\hat\Y^M(\omega)\in\K$, where $\omega\in\Omega$.  
\subsubsection{Proofs}

\begin{proof}[Proof of Lemma \ref{lem:covering_probability}] 
Let $X^M_\delta=\{x_i\}_{i{=}1}^M+B(0,\delta)$ and 
$\pi(\partial\X,X_{\delta/2}^M)
=
\sup_{x\in\partial\X}\Prob(B(x,\delta/2) \cap X^M = \emptyset)$, 
which gives the worst probability over $x\in\partial\X$ of not sampling an input $x_i$ that is $(\delta/2)$-close to some $x\in\partial\X$. 
Since the $x_i$'s are iid and by Assumption \ref{assum:sampling_density:boundary}, 
\begin{align*}
\pi(\partial\X,X_{\delta/2}^M) 
&=
\sup_{x\in\partial\X}\Prob(B(x,\delta/2) \cap X^M = \emptyset)
= \sup_{x\in\partial\X}\Prob\left(\bigcap_{i=1}^M
(x_i\notin B(x,\delta/2))\right)
\\
&=
\left(1 -
	\inf_{x\in\partial\X}\Prob_\X(B(x,\delta/2)) 
\right)^M
\leq 
\left(1 -
	\Lambda_\delta^{\partial\X}
\right)^M.
\end{align*}
Let $F_{\partial\X}\subset\partial\X$ be a minimal internal $(\delta/2)$-covering of $\partial\X$, so that $\partial\X\subset F_{\partial\X}+B(0,\delta/2)$ and the number of elements in $F_{\partial\X}$ is the internal $(\delta/2)$-covering number $N(\partial\X,\delta/2)$. Then,  
$\partial\X\nsubseteq X_\delta^M\implies
F_{\partial\X}+B(0,\delta/2)\nsubseteq X_\delta^M$ and
\begin{align*}
\Prob(\partial\X\nsubseteq X_\delta^M)
&\leq 
\Prob(F_{\partial\Y}+B(0,\delta/2)\nsubseteq X_\delta^M)
=
\Prob(F_{\partial\X}\nsubseteq X_{\delta/2}^M)
=
\Prob\bigg(\bigcup_{x\in F_{\partial\X}} x\notin X_{\delta/2}^M\bigg)
\\
&\leq 
\sum_{x\in F_{\partial\X}}\Prob(x_i\notin X_{\delta/2}^M)
=
\sum_{x\in F_{\partial\X}}
\Prob(\{x\}\cap X_{\delta/2}^M = \emptyset)
\\
&\leq 
|F_{\partial\X}| \cdot
\sup_{x\in\partial\X}\Prob(\{x\}\cap X_{\delta/2}^M = \emptyset)
=
N(\partial\X,\delta/2)\pi(\partial\X,X_{\delta/2}^M)
\\
&\leq
N(\partial\X,\delta/2)(1 -
	\Lambda_\delta^{\partial\X})^M
=
\beta_{M,\delta}^{\partial\X}.
\end{align*}
Thus, the sample $X^M=\{x_i\}_{i{=}1}^M$ is a $\delta$-cover of $\partial\X$ with probability at least $1-\beta_{M,\delta}^{\partial\X}$. 
\end{proof}

\begin{proof}[Proof of Corollary \ref{cor:conservative_finite_sample}]
By Lemma \ref{lem:covering_probability}, with probability at least $1-\beta_{M,\delta}^{\partial\X}$, the sample $X^M=\{x_i\}_{i{=}1}^M\subset\partial\X$ is an internal $\delta$-cover of $\partial\X$. The result follows from Lemma \ref{lem:error_bound:covering_lipschitz} and Theorem \ref{thm:error_bound:smooth_boundary}.
\end{proof}

\begin{proof}[\revFirst{Proof of Corollary \ref{cor:conservative_asymptotic}}]
\revFirst{Let $d=m-1$. First, since $\X$ is compact and $r$-smooth, $N(\partial\X,\delta/2)\leq C_{d,r}\delta^{-d}$ for some constant $C_{d,r}>0$ for all $\delta\leq r/2$ by \cite[Lemma 2.2]{Aamari2018} (see also \cite[Lemma 10]{Chazal2015}). Second, since $\Prob_\X$ is \rev{absolutely continuous over $\partial\X$ and its density $p(x)$ is bounded below by a strictly positive constant,} %uniform over $\partial\X$, 
$\Prob_\X$ satisfies Assumption \ref{assum:sampling_density:boundary} with $\Lambda_{\delta}^{\partial\X}=\bar{C}_{d,r}\delta^d$ for some constant $\bar{C}_{d,r}>0$ for all $\delta\leq r/2$ by \cite[Lemma 9.1]{Aamari2018}. 
Let $\delta=(\frac{3}{\bar{C}_{d,r}}\frac{\log(M)}{M})^{1/d}$ with $M$ large-enough so that $\delta\leq r/2$. Let $\tilde{C}_{f,r}=(\bar{L}/r+\bar{H})/2$. Then, by Corollary \ref{cor:conservative_finite_sample}, 
$$
\dH(\hull(\Y),\hat\Y^M) > \tilde{C}_{f,r}\delta^2
= 
\tilde{C}_{f,r}\left(\frac{3}{\bar{C}_{d,r}}\right)^{2/d}\left(\frac{\log(M)}{M}\right)^{2/d}
$$
with probability less than 
\begin{align*}
\rev{\beta^{\partial\X}_{M,\delta}}=C_{d,r}\delta^{-d}(1-\bar{C}_{d,r}\delta^d)^M\leq 
C_{d,r}\delta^{-d}\exp(-M\bar{C}_{d,r}\delta^d)
&=
C_{d,r}\frac{\bar{C}_{d,r}}{3}\frac{M}{\log(M)}
\exp\left(-3\log(M)\right).
% \\
% &=
% \frac{3C_{d,r}}{\bar{C}_{d,r}}\frac{1}{M^2\log(M)}.
\end{align*}
Since $\sum_{M=1}^\infty\rev{\beta^{\partial\X}_{M,\delta}}<\infty$, the conclusion follows from the Borel-Cantelli lemma.}
\end{proof}

\subsection{Reachability analysis}
\begin{proof}[Proof of Corollary \ref{cor:reachability_analysis}]
The map $(x^0,\theta)\mapsto x_u^{x^0,\theta}(t)$ is a $C^1$ submersion. The result follows from Theorem \ref{thm:error_bound:smooth_boundary}.
\end{proof}

\subsection{Robust optimization}
\begin{proof}[Proof of Corollary \ref{cor:robust_programming}]
Given $u\in\R^p$, define $\Y_u=f(\X,u)$ and $\hat{\Y}_u^M=\hull(\{f(x_i,u)\}_{i=1}^M)$. Then, $\mathbf{P}$ is equivalent to
\begin{align*}
\mathbf{P}: \, \inf_{u\in\R^p} \ \ell(u)\ \ 
\text{s.t.} \ \ \Y_u\subseteq\C.
\end{align*}
By Theorem \ref{thm:error_bound:smooth_boundary}, 
$\dH(\hull(\Y_u),\hat{\Y}_u^M)\leq\epsilon$,  
which implies that $\hull(\Y_u)\subseteq\hat{\Y}_u^M+B(0,\epsilon)$.  
Since $\C$ is convex, $f(x_i,u)+B(0,\epsilon)\in\C$ for all $i=1,\dots,M$ if and only if $\hat{\Y}_u^M+B(0,\epsilon)\subseteq\C$. 
Thus, any solution $u\in\R^p$ of $\hat{\mathbf{P}}_\epsilon^M$ satisfies
$$
\Y_u\subseteq\hull(\Y_u)\subseteq\hat{\Y}_u^M+B(0,\epsilon)\subseteq\C.
$$
Thus, any solution of $\hat{\mathbf{P}}_\epsilon^M$ is feasible for $\mathbf{P}$.
\end{proof}

\subsection{Numerical example: planning under bounded uncertainty}\label{sec:applications:robust_planning:appendix}
The planning horizon is $T=30\,\textrm{s}$. Constraints are given as
\begin{align*}
\frac{n_1}{\|n_1\|}(p(t)-c_1)\leq 0\text{ for }t\in[0,20), 
\quad
&n_1=(1,-5),\ c_1=(0,-0.1),
\\
\frac{n_2}{\|n_2\|}(p(t)-c_2)\leq 0\text{ for }t\in[20,T], 
\quad
&n_2=(-1,-5),\ c_2=(2,-0.1),
\\
-\Delta p_{\textrm{goal}}\leq p(T)-p_{\textrm{goal}}\leq\Delta p_{\textrm{goal}},
\quad
&p_{\textrm{goal}}=(2,0.05), \ 
\Delta p_{\textrm{goal}}=(0.3, 0.14),
\end{align*}
The feasible control set is given by $U=\{u\in\R^2:\|u\|_\infty\leq \bar{u}_{\textrm{max}}\}$ with $\bar{u}_{\textrm{max}}=0.2$.  We optimize over a space $\U$ of stepwise-constant controls $u(t)=\sum_{s=0}^{T-1}\bar{u}_s\mathbf{1}_{[s,s+1)}(t)$ where $\bar{u}_s\in U$ for all $s=0,\dots,T-1$, and $\mathbf{1}_{[s,s+1)}(t)=1$ if $t\in[s,s+1)$ and $0$ otherwise. % 

We assume that $F\in B(0,F_{\textrm{max}})\subset\R^2$ with $F_{\textrm{max}}=0.005\,\textrm{N}$ and $m\in[30,34]\,\textrm{kg}$. To obtain tighter bounds with our analysis, we make a change of variables. We define 
$M=\frac{1}{m}$, $\overline{M}=\frac{1}{32}$, $\Delta M=M-\overline{M}$, and $\gamma=\frac{9}{4}$, so that
$$
\frac{1}{m}=M=\overline{M}+\Delta M=\overline{M}+\frac{1}{\gamma}(\gamma\Delta M).
$$
One verifies that $(\gamma\Delta M)\in[-0.005,0.005]=[-F_{\textrm{max}},F_{\textrm{max}}]$. Thus, the parameters $x=(\gamma\Delta M,F)$ satisfy $x\in \X=B(0,r)\subset\R^3$ for $r=\sqrt{2}F_{\textrm{max}}$, which is a non-empty $r$-smooth compact set.

Given a piecewise-constant control $u\in\U$, the trajectory of the system is given by
$$
p_u^x(t)=
p(0)+v(0)t+
\frac{1}{2m}\left(
    Ft^2 + 
    \sum_{k=0}^{s-1}\bar{u}_k(2(t-k)-1)
    +\bar{u}_s\Delta t^2
    \right)
$$
for any time $t=s+\Delta t$ with $s\in\bN$ and $|\Delta t|<1$. % 
Thus, the map $(m,F)\to p_u^{m,F}(t)$ is a submersion and Corollaries \ref{cor:robust_programming} and \ref{cor:reachability_analysis} apply. Specifically, the convex hull of the reachable positions of the system and the constraints of the problem can be accurately approximated using a finite number of inputs $(m_i,F_i)$, and sampling the boundary $\partial\X$ is sufficient.

For any $t\in\bN\cap [0,T]$ (so that $s=t$ and $\Delta t=0$) and defining $\Delta_k=2(t-k)-1$, 
\begin{align*}
p_u^x(t)
&=
p(0)+v(0)t+
\frac{1}{2m}\left(t^2F+\sum_{k=0}^{t-1}\bar{u}_k\Delta_k\right)
\\
&=
p(0)+v(0)t+
\frac{1}{2}\left(\overline{M}\left(t^2F+\sum_{k=0}^{t-1}\bar{u}_k\Delta_k\right)
+\frac{1}{\gamma}(\gamma\Delta M)\left(t^2F+\sum_{k=0}^{t-1}\bar{u}_k\Delta_k\right)
\right)
\\
&=
p(0)+v(0)t+
\frac{\overline{M}}{2}\sum_{k=0}^{t-1}\bar{u}_k\Delta_k
+
\left(\frac{\overline{M}t^2}{2}\right)F
+
\left(\frac{1}{2\gamma}\sum_{k=0}^{t-1}\bar{u}_k\Delta_k\right)(\gamma\Delta M)
+
\left(\frac{t^2}{2\gamma}\right)(\gamma\Delta M)F
\end{align*}
which is quadratic in $x=(\gamma\Delta M,F)$, so the differential $x\mapsto \dd (p_u^x(t))_{x}$ is $\bar{H}_t=(t^2/2\gamma)$-Lipschitz. 

By rearranging terms, for $x_1,x_2\in\X$,
\begin{align*}
\left\|p_u^{x_1}(t)-p_u^{x_2}(t)\right\|
&=
\bigg\|
\underbrace{\frac{1}{2\gamma}
\begin{bmatrix}
t^2(\gamma\overline{M}+(\gamma\Delta M)_2)I_{2\times2}
&
\left(\sum_{k=0}^{t-1}\bar{u}_k\Delta_k+t^2F_1\right)
\end{bmatrix}}_{A(x_1,x_2,t)}
\underbrace{\begin{bmatrix}
F_1-F_2
\\
(\gamma\Delta M)_1-(\gamma\Delta M)_2
\end{bmatrix}}_{x_1-x_2}
\bigg\|
\\
&\leq 
\|A(x_1,x_2,t)\|\|x_1-x_2\|.
\end{align*}
Since $\|A(x_1,x_2,t)\|\leq\sqrt{2}\|A(x_1,x_2,t)\|_\infty=\sqrt{2}\max_{ij}(|A_{ij}(x_1,x_2,t)|)$, we obtain
\begin{align*}
\|A(x_1,x_2,t)\|
&\leq
\frac{1}{\sqrt{2}\gamma}
\max\left(
\left|t^2(\gamma\overline{M}+(\gamma\Delta M)_2)\right|,
\ 
\left|\sum_{k=0}^{t-1}\bar{u}_{k,1}\Delta_k+t^2F_{1,1}\right|
\right)
\\
&\leq
\frac{t^2}{\sqrt{2}\gamma}
\max\left(
(\gamma M_{\textrm{max}}+F_{\textrm{max}}),
\, 
(\bar{u}_{max}+F_{\textrm{max}})
\right)
\end{align*}
where $M_{\textrm{max}}=1/30$ and since $\gamma\Delta M\in[-F_{\textrm{max}},F_{\textrm{max}}]$ and $\sum_{k=0}^{t-1}\Delta_k=t^2$. 
Defining 
$$
\bar{L}=\frac{T^2}{\sqrt{2}\gamma}
\max\left(
(\gamma M_{\textrm{max}}+F_{\textrm{max}}),
\, 
(\bar{u}_{max}+F_{\textrm{max}})
\right),
\quad
\bar{H}=\frac{T^2}{2\gamma},
$$ 
we conclude that the submersion $x\mapsto p_u^x(t)$ is $\bar{L}$-Lipschitz and its differential $x\mapsto \dd (p_u^x(t))_{x}$ is $\bar{H}$-Lipschitz for all $t\in[0,T]$. 

Let $\delta>0$ and $X^M=\{(\gamma\Delta M_i,F_i)\}_{i=1}^M\subset\partial\X$ be a $\delta$-covering of $\partial\X$. Applying Corollary \ref{cor:reachability_analysis}, 
$$
\dH\big(\hull(\Y_u(t)),\hat{\Y}_u^M(t)\big)\leq \frac{1}{2}\left(\frac{\bar{L}}{r}+\bar{H}\right)\delta^2.
$$
In contrast, a naive Lipschitz bound would give (see Corollary \ref{cor:conservative_finite_sample}) 
$$
\dH\big(\hull(\Y_u(t)),\hat{\Y}_u^M(t)\big)\leq \bar{L}\delta.
$$

$X^M$ is constructed as a Fibonacci lattice \cite{Gonzalez2009} with $M=100$ points, which gives an internal $\delta$-covering of $\partial\X$ for $\delta=10^{-3}$. We then evaluate the bound above and obtain $\dH\big(\hull(\Y_u(t)),\hat{\Y}_u^M(t)\big)\leq \epsilon$ with $\epsilon=\left(\bar{L}/r+\bar{H}\right)\delta^2/2=0.025$. We use this value of $\epsilon$ to pad the constraints as described in Section \ref{sec:applications:optimization}. We refer to our open-source implementation for further details.

\section{Prior error bounds in the convex setting}% 
\label{sec:error_bounds_dumbgen_convex}
In this section, we report previous error bounds for completeness. Specifically, \cite[Theorem 1]{Dumbgen1996} gives an error bound for reconstructing convex sets with smooth boundary.  
\\[-2mm]

\noindent\begin{minipage}{0.65\linewidth}
\begin{theorem}\label{thm:dumbgen}\cite[Theorem 1]{Dumbgen1996}
Let $\X\subset\R^n$ be a convex compact set such that $\Int(\X)\neq\emptyset$. Let $R>0$. Assume that 
for any $x\in\partial\X$, there exists a unique $n(x)\in\R^n$ with $\|n(x)\|=1$ such that $y^\top n(x)\leq x^\top n(x)$ for all $y\in\X$ and 
\begin{align*}
\|n(x)-n(y)\|\leq \frac{1}{R}\|x-y\|\quad 
&\text{for all }x,y\in\partial\X.
\end{align*}
Let $\delta>0$ and $Z_\delta\subset\partial\X$  be such that $\partial\X\subset Z_\delta+B(0,\delta)$. Then,
$$
\dH(\X,\hull(Z_\delta))\leq \frac{\delta^2}{R}.
$$
\end{theorem}
\end{minipage}% 
\hspace{1mm}
\begin{minipage}{0.35\linewidth}
    \centering	
 \includegraphics[width=0.8\linewidth]{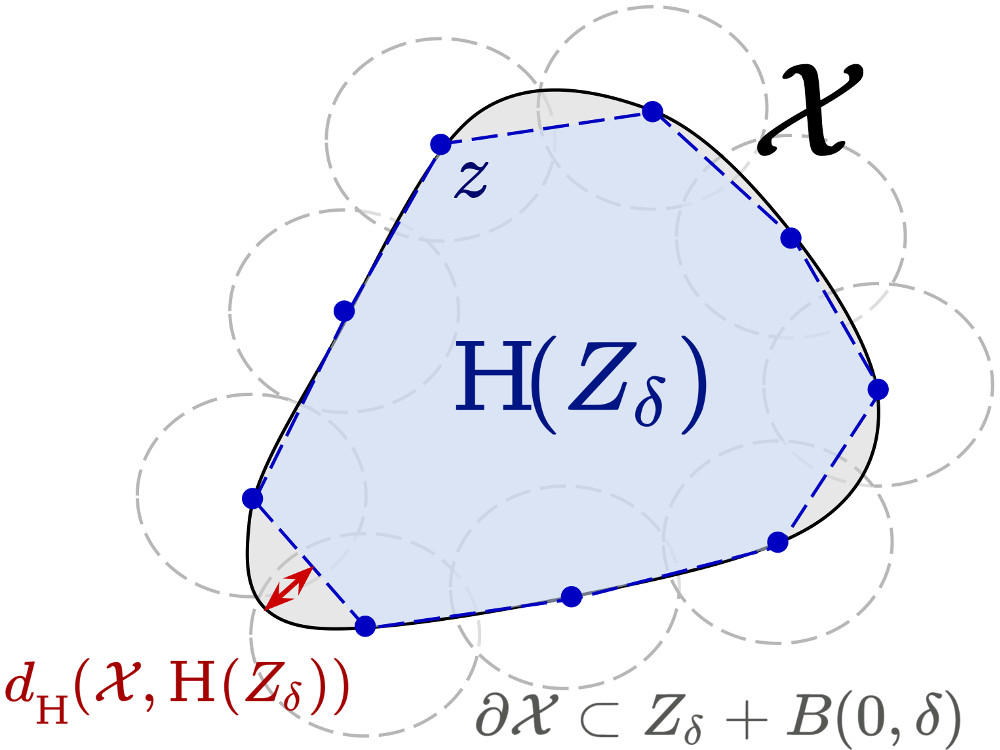}
\captionof{figure}{Error bound in the convex setting \cite[Theorem 1]{Dumbgen1996}.}
\end{minipage}% 

For completeness, we provide a proof of this result at the end of this section. 
Theorem \ref{thm:dumbgen} implies that convex sets with smooth boundaries (i.e., with a Lipschitz-continuous normal vector field) are accurately reconstructed using the convex hull of points on the boundary. Theorem \ref{thm:error_bound:smooth_boundary} improves this error bound by a factor of $2$.

By assuming that $\Y=f(\X)$ is convex and $f$ is a diffeomorphism, Theorem \ref{thm:dumbgen} can be used to derive an error bound for the estimation of the image of sets $\X$ with smooth boundary. In contrast, Theorem \ref{thm:error_bound:smooth_boundary} does not require a convexity assumption and applies to submersions $f:\R^m\to\R^n$ as well, which allows studying problems where the dimensionality $m$ of the input set $\X$ is larger than the dimensionality $n$ of the output set $\Y$.

\begin{corollary}\label{cor:dumbgen} Let $r>0$, $\X\subset\R^n$ be a non-empty path-connected compact set, $f:\R^n\to\R^n$, and $\Y=f(\X)$. 
Let $\delta>0$ and $Z_\delta\subset\partial\X$  be such that $\partial\X\subset Z_\delta+B(0,\delta)$. Assume that 
% \textcolor{blue}{(A1)} 
$\X$ is $r$-smooth,  
% \textcolor{blue}{(A2)} 
$f$ is a $C^1$ diffeomorphism such that $(f,f^{-1},\dd f)$ are $(\bar{L},\underline{L},\bar{H})$-Lipschitz, and 
% \textcolor{blue}{(A3)} 
$\Y$ is convex. 
Then,
$$
\dH(\Y,\hull(f(Z_\delta)))\leq \frac{(\bar{L}\delta)^2}{R},
\quad \text{
where }R = \frac{1}{\left(\frac{\bar{L}}{r}+\bar{H}\right)\underline{L}^2}.
$$
\end{corollary}
\begin{proof}
Let $R=\left(\left(\frac{\bar{L}}{r}+\bar{H}\right)\underline{L}^2\right)^{-1}$. 
By Corollary \ref{cor:rsmooth:diffeo}, $\Y$ is $R$-smooth\rev{, since $\X$ is $r$-smooth and $f$ is a diffeomorphism}. 

Thus, by Theorem \ref{thm:walther1999} (note that $\Int(\Y)\neq\emptyset$ since $\Int(\X)\neq\emptyset$ and $f$ is bijective), $\partial\Y$ is an $(n-1)$-dimensional submanifold in $\R^n$ with the outward-pointing normal $n(x)$ at $x\in\partial\Y$ satisfying 
$\|n(x)-n(y)\|\leq\frac{1}{R}\|x-y\|$ for all $\forall x,y\in\partial\Y
$ (and $y^\top n(x)\leq x^\top n(x)$ for all $y\in\Y$).

Since $Z_\delta$ is a $\delta$-cover of $\partial\X$ and $f$ is $\bar{L}$-Lipschitz, $f(Z_\delta)$ is an $(\bar{L}\delta)$-cover of $\partial\Y$. 

The result follows from Theorem \ref{thm:dumbgen} using the last two results. % and (A3). 
\end{proof}

\begin{proof}[Proof of Theorem \ref{thm:dumbgen} \cite{{Dumbgen1996}}]
First, we define the support function of $\X$ as 
$h(\X,\cdot):\R^n\to\R, \ u\mapsto h(\X,u):=\sup_{x\in \X}x^\top u$.  
Since $\X$ and $\hull(Z_\delta)$ are both convex, non-empty, and compact, by \cite[Lemma 1.8.14]{Schneider2014}, 
\begin{align*}
\dH(\X,\hull(Z_\delta))
&=
\max_{u\in\partial B(0,1)}|h(\X,u)-h(\hull(Z_\delta),u)|
=
|h(\X,u_0)-h(\hull(Z_\delta),u_0)|
\\
&=
\bigg|\sup_{x\in\X}x^\top u_0-\sup_{z\in \hull(Z_\delta)}z^\top u_0\bigg|
\end{align*}
for some $u_0\in\R^n$ with $\|u_0\|=1$. 
Let $x_0\in\partial\X$ be such that $h(\X,u_0)=x_0^\top u_0$. Then, $u_0=n(x_0)$. Indeed, by assumption, $n(x_0)$ is the unique vector in $\partial B(0,1)$ that satisfies $x^\top n(x_0)\leq x_0^\top n(x_0)=h(\X,u_0)$ for all $x\in\X$.
Thus, 
\begin{align*}
\dH(\X,\hull(Z_\delta))
&=
\bigg|x_0^\top n(x_0)-\sup_{z\in \hull(Z_\delta)}z^\top n(x_0)\bigg|
=
\bigg|\inf_{z\in Z_\delta} (x_0-z)^\top n(x_0)\bigg|
=
\bigg|(x_0-z)^\top n(x_0)\bigg|
\end{align*}
for some $z\in Z_\delta\subset\partial\X$ with $\|x_0-z\|\leq\delta$. % 
In addition,
\begin{align*}
(x_0-z)^\top n(x_0)&\geq 0, 
\hspace{4.2cm}
\text{(since $z\in\X$)}
\\
\text{and}\qquad 
(x_0-z)^\top n(x_0)&=(x_0-z)^\top (n(x_0)-n(z))+(x_0-z)^\top n(z)
\\
&\leq (x_0-z)^\top (n(x_0)-n(z)), \quad \text{(since $z\in\partial\X$)}.
\end{align*}
Thus,
\begin{align*}
\dH(\X,\hull(Z_\delta))
&=
\bigg|(x_0-z)^\top n(x_0)\bigg|
\\
&=
(x_0-z)^\top n(x_0)
\\
&\leq 
(x_0-z)^\top (n(x_0)-n(z))
\\
&\leq 
\|x_0-z\| \|n(x_0)-n(z)\|
\\
&\leq 
\delta^2/R,
\end{align*}
where the last inequality follows by assumption.
\end{proof}

\end{document}